\documentclass[11pt]{amsart}

\usepackage{amsmath}
\usepackage{amssymb,amsfonts,latexsym}
\usepackage{geometry}
\usepackage[all]{xy}

\usepackage{verbatim}

\newtheorem{thm}{Theorem}[section]
\newtheorem{cor}[thm]{Corollary}
\newtheorem{lem}[thm]{Lemma}
\newtheorem{prop}[thm]{Proposition}
\theoremstyle{definition}

\newtheorem{defn}[thm]{Definition}
\theoremstyle{remark}
\newtheorem{rem}[thm]{Remark}

\newtheorem{exam}[thm]{Example}

\long\def\forget#1\forgotten{}

\newcommand{\eps}{\varepsilon}

\newcommand{\mQ}{\mathbb{Q}}

\newcommand{\mC}{\mathbb{C}}

\newcommand{\mP}{\mathbb{P}}

\newcommand{\K}{\mathbb{K}}

\newcommand{\abs}[1]{\lvert #1 \rvert}

\newcommand{\ra}{\rightarrow}

\newcommand\suchthat{{\,:\ }}

\DeclareMathOperator\Red{Red}
\def\({\left(}
\def\){\right)}
\def\divides{{\,|\,}}

\newcommand\oline[1] {{\overline{#1}}}

\newcommand\Hom{{\operatorname{Hom}}}

\newcommand\Aut{{\operatorname{Aut}}}
\newcommand\End{{\operatorname{End}}}

\newcommand\lcm{{\operatorname{lcm}}}

\newcommand\Orb{{\operatorname{Orb}}}

\newcommand\ceil[1]{{\lceil {#1} \rceil}}
\DeclareMathOperator\Mon{Mon}
\DeclareMathOperator\fpr{fpr}
\DeclareMathOperator\PSL{PSL}
\DeclareMathOperator\PGL{PGL}

\DeclareMathOperator\Sym{Sym}

\newcommand{\mybar}[1]{#1\llap{$\overline{\phantom{\rm#1}}$}}

\usepackage[colorlinks,pagebackref,pdftex, bookmarks=false]{hyperref}

\numberwithin{equation}{section}
\numberwithin{table}{section}

\begin{document}

\title
{Monodromy groups of indecomposable coverings of bounded genus}%

\def\ann{Department of Mathematics, University of Michigan, Ann Arbor, MI 48109--1043, USA}
\def\technion{Department of Mathematics, Technion - IIT, Haifa 3200, Israel}
\author{ Danny Neftin}
\address{\technion}
\email{dneftin@tx.technion.ac.il}%
\author{ Michael E. Zieve}
\address{\ann}
\email{zieve@umich.edu}


\begin{abstract}
For each nonnegative integer $g$, we classify the ramification types and monodromy groups of indecomposable coverings of complex curves $f\colon X\to Y$ where $X$ has genus $g$, under the hypothesis that $n:=\deg(f)$ is sufficiently large and the monodromy group is not $A_n$ or $S_n$.  This proves a conjecture of Guralnick and several conjectures of Guralnick and Shareshian.
\end{abstract}

\maketitle

\section{Introduction}\label{sec:intro}
\subsubsection*{The classification of monodromy groups} 
Since the work of Riemann and Hurwitz, it has been known that many properties of a branched covering $f\colon X\to Y$ of complex curves are determined by the values of certain fundamental invariants of the covering.  This has led to much effort in studying the possibilities for the values of these invariants.  The most important invariant of $f$ is its \emph{monodromy group} $\Mon(f)$, namely the Galois group of the Galois closure of the function field extension $\mC(X)/\mC(Y)$, viewed as a group of permutations of $\Hom_{\mC(Y)}(\mC(X),\mybar{\mC(Y)})$.  Equivalently, writing $B$ for the branch locus of $f$ and choosing a base point $P_0\in Y(\mathbb{C})\setminus B$, the group $\Mon(f)$ is isomorphic as a permutation group to the image of the fundamental group of $Y\setminus B$ under the monodromy representation $\pi_1(Y\setminus B,P_0)\to\Sym(f^{-1}(P_0))$.  In many respects, all degree-$n$ coverings with monodromy group $A_n$ or $S_n$ behave similarly to random degree-$n$ coverings, so that it is of particular interest to describe the coverings with different monodromy groups.  
Our main result achieves this goal by describing all coverings $f\colon X\to Y$ whose monodromy group is neither $A_{\deg(f)}$ nor $S_{\deg(f)}$ nor a member of a finite list of groups which depends only on the genus of $X$.  In particular, for any prescribed value of the genus of $X$, we describe all coverings $f:X\to Y$ of sufficiently large degree whose monodromy group is not $A_{\deg(f)}$ or $S_{\deg(f)}$. The main source of such ``exceptional" coverings is \emph{decomposable} coverings, namely the coverings $f\colon X\to Y$ which can be written as the composition $f=f_1\circ f_2$ of coverings $f_2\colon X\to Z$ and $f_1\colon Z\to Y$ which both have degree at least $2$.  We determine the exceptional \emph{indecomposable} coverings, in the sense that we determine all occurring possibilities for their key invariants, namely the monodromy group and the \emph{ramification type} of the covering.  Here the ramification type of a covering $f\colon X\to Y$ with branch locus $B$ is the multiset $\{E_f(P):P\in B\}$, where $E_f(P)$ is the multiset of ramification indices (i.e., local multiplicities) of $f$ at the points in $f^{-1}(P)$.  As we illustrate below, knowledge of
the monodromy group and ramification type of a covering provides sufficient information to answer a wide assortment of questions about the covering.


The first steps towards classifying monodromy groups and ramification of indecomposable coverings of bounded genus were
taken over 100 years ago by Chisini \cite{Chisini} and Ritt \cite{Ritt2}, who did this for prime-degree rational functions 
$f:\mP^1\ra\mP^1$ under the assumption that the monodromy group is solvable.
These results were generalized soon afterwards by Zariski \cite{Zar1,Zar2} to rational functions of arbitrary degree whose monodromy group has a trivial two-point stabilizer.
Nearly 65 years later, Guralnick and Thompson \cite{GT} reignited the study of monodromy groups 
by connecting them with modern results in group theory, including the classification of finite simple groups.
In particular, the combination of the papers
\cite{GT,Asch, Shih, Gur1, Neu1, Neu2, GN, LWay, GM, LSa, LSh, FM}
yields a classification of
the simple groups which occur as composition factors of the monodromy group of an indecomposable covering $f:X\ra \mP^1$ whose degree is sufficiently large compared to the genus $g_X$.
However, this remarkable achievement has not had many applications, since it seems that in applications one needs to know the full monodromy group rather than just the composition factors.

 
In this paper, building on the above papers and also on both \cite{GS} and our previous paper \cite{NZ2}, we give a complete list of the possible monodromy groups 
of indecomposable coverings $f:X\to\mP^1$ whose degree $n$ is sufficiently large compared to the genus $g_X$.  Moreover, in case the monodromy group is not $A_n$ or $S_n$, we also determine all possibilities for the ramification type of $f$.
\begin{thm}\label{thm:general} 
Fix a nonnegative integer $g$. There exists a constant $N_g$ such that every indecomposable covering $f:X \ra\mP^1$ of genus $g_X = g$ and degree   $n \geq N_g$ satisfies one of the following:
\begin{enumerate}
\item $\Mon(f) = A_n$ or $S_n$;
\item $\Mon(f) = A_\ell$ or $S_\ell$ with $n=\ell(\ell-1)/2$, and $g_X=0$. Moreover, the ramification type of $f$ is given in Table~\emph{\ref{table:f}}. 
\item   $A_\ell^2< \Mon(f) \leq S_\ell^2\rtimes C_2$, where $C_2$ permutes the two copies of $S_\ell$,  $n=\ell^2$ and $g_X\leq 1$. Moreover, the ramification type of $f$ is given in \cite[Table 3.1]{NZ2}\footnote{The ramification types in \cite[Table 3.1]{NZ2} are represented concisely by conjugacy classes of elements in $N_{S_n}(\Mon(f))=S_\ell\wr S_2$. The ramification of $f$ is then the list of cycle lengths of these elements, see  \cite[\S3]{NZ2}.}. 
\item  $\Mon(f)\leq C_p^i\rtimes C_k$, $n=p^i$, and $g_X=0$
where $p$ is prime, $i\leq 2$,  and $k\leq 6$. Moreover, the ramification type of $f$ is given in Table~\emph{\ref{table:solvable}}.
\end{enumerate}
\end{thm}

Theorem \ref{thm:general} 
proves a conjecture of Guralnick \cite[Conjecture 1.0.1]{GS}. 
In this result, case (1) is the general case, which occurs for any randomly chosen degree-$n$ covering.
It is expected that there exist covers in case (1) whose ramification type is nearly any collection of partitions of $n$ which is consistent with the Riemann--Hurwitz formula for $f$.
By contrast, the possible ramification types
in cases (2)--(4) are severely restricted.
For instance, in case (2) there are approximately $2\varphi(\ell)$ ramification types for any prescribed $\ell$, and in case (3) there are approximately $6\varphi(\ell)$, 
where $\varphi(\ell)$ denotes Euler's totient function. 
%
The ramification types occurring in case (4) appear in Zariski's work, and are completely determined in Proposition \ref{lem:normal-closure}.
Up to composition with linear fractionals, the rational functions in case (4) are $X^n$, Chebyshev polynomials $T_n(X)$, and the composition factors of certain Latt\`es maps. The Latt\`es maps in question are the maps $L_n:E/\langle\sigma\rangle\ra E/\langle\sigma\rangle$ induced by the multiplication by $n$ map on an elliptic curve $E$ upon  taking the quotient by a nontrivial $\sigma\in \Aut(E)$. 
\begin{rem}
The lists of ramification types in items (2)--(4) of Theorem \ref{thm:general} cannot be shortened.  Namely, we show in 
Section \ref{sec:ramification} and Proposition \ref{prop:An-Sn-ram-types} (resp.\ \cite[Proposition 10.1]{NZ2} and Proposition  \ref{lem:isogeny-mon}) that each of the types in (2) (resp., (3) and (4)) occurs as the ramification type of an indecomposable covering $f:X\ra\mP^1$ with monodromy group as in (2) (resp., (3) and (4)). In particular, Proposition~\ref{prop:An-Sn-ram-types} proves eight of the ten assertions in \cite[Conjecture 3.29]{GS} of Guralnick--Shareshian, and Remark~\ref{rem:GS-correct} corrects the remaining two assertions.
\end{rem}

\begin{rem}
The analogue of Theorem~\ref{thm:general} holds (much more simply) for degree-$n$ coverings $X\to Y$ with $g_Y>0$.  For, Riemann--Hurwitz implies that $g_X-1\ge n(g_Y-1)$, so that if $g_Y>1$ then there are no coverings at all with $n\ge g_X$.
If $g_Y=1$ then \cite[Proposition~2.3]{GN} asserts
that, if $f:X\ra Y$ is an indecomposable covering of degree $n\ge (2g_X-1)^2$ for which $\Mon(f)$ is not $A_n$ or $S_n$, then $n$ is prime, $\Mon(f)=C_n$, and $f$ is unramified.
\end{rem}

In the special case of {\it polynomials},  
the above classification was carried out in  \cite{Mul}.
This case is significantly easier, since it is known that there are very few possibilities for the monodromy group of an indecomposable covering $f:X\to Y$ in which some point of $Y(\mC)$ has a unique preimage \cite{Feit,Jones}.
The resulting monodromy groups are used in a wide range of topics including:   the near-injectivity of polynomial maps $f:\mQ\ra\mQ$  \cite{AZ,CHZ} and  local-global injectivity principles  \cite{Ros}; the finiteness of common values $f(\mQ)\cap g(\mQ)$ of polynomials $f,g\in\mQ[x]$ and properties of the associated curve $f(x)=g(y)$ \cite{Fried, CN, BT}; Ritt functional decompositions \cite{MZ,  Pak2} and invariant subvarieties under polynomial maps \cite{MS}; common value sets $f(\mathbb 
F_p)=g(\mathbb 
F_p)$, arithmetic equivalence \cite{Mul4}, and  fake values  ($a\in \mQ\cap f(\mQ_p)$ for almost all $p$)  \cite{Cor};   the finiteness of complements $\Red_f:=\{t_0\in\mQ\,:\, f(t_0,x)\in\mQ[x]\text{ is reducible}\}$ of Hilbert sets \cite{DF,DW, KN,Mul3}; the minimal ramification problem \cite{BSEF}; and complex and arithmetic dynamics \cite{BDGHT, KN2, Pink1, MZ, GTZ}.

Similarly to the case of polynomials, Theorem \ref{thm:general}  plays a key role in many problems concerning rational functions and coverings of bounded genus. These include the conjecture of the second author that every rational function $f:\mQ\ra \mQ$ is at most  $16$-to-$1$ over all but finitely many values \cite{Segol}; the study of rational points on the curves $f(x)=g(y)$, $f,g\in \mQ(x)$ and determining when are these reducible \cite[\S 5.2 of V1]{KN}; finding  rational functions $f_1\ldots,f_r$ whose fake values cover $\mQ$ \cite{KK};  determining when $\Red_f$ is finite for ``general" polynomials $f\in\mQ(t)[x]$ \cite{MN} and other families \cite{KN}; determining which groups appear as the Galois groups of infinitely many inequivalent trinomials \cite[\S 1.4]{GS}; rational  curves of gonality $n$, and the finiteness of primitive degree $n$-points other than $A_n$ and $S_n$-points on such curves \cite[Theorem 2]{KS}; and as pointed out by  Chinburg, in view of the analogy between genera of curves over finite fields and root discriminants of number fields \cite[\S 3.5]{Stich},  in determining the possibilities for the Galois groups of primitive number fields of bounded root discriminant.

\subsubsection*{Main result}
One of the most difficult cases in the classification of monodromy groups is when $G$ is $A_\ell$ or $S_\ell$, but where the permutation representation of $G$ is not the standard action on $\ell$ points. 
This is the final class of permutation groups that has to be analyzed in order to complete the proof of
Theorem~\ref{thm:general}, in light of previous results.
Note that for a Galois covering $\tilde f:\tilde X\ra \mP^1$ with monodromy group $G\in \{A_\ell,S_\ell\}$, and a maximal subgroup $H< G$, the covering $\tilde f$ induces an indecomposable covering $\tilde X/H\ra \mP^1$. Since every indecomposable covering $f:X\ra\mP^1$ with alternating or symmetric monodromy is isomorphic to the covering $\tilde X/H\ra \mP^1$ induced by some $\tilde f$ and $H$ as above, Theorem \ref{thm:general} reduces, in the case of alternating or symmetric monodromy groups, to finding all such pairs $\tilde f,H$ for which $g_{\tilde X/H}$ is less than some function of $\ell$.  We achieve this in the following result.
\begin{thm}\label{thm:Sn}  There exist constants $a,N>0$ such that for 
every 
Galois covering $\tilde f: \tilde X\ra \mathbb{P}^1$ with monodromy group $G=A_\ell$ or $S_\ell$, with $\ell\geq N$, and every maximal subgroup $H$ of $G$ such that $H\neq A_\ell$, one of the following conditions holds: 
\begin{enumerate}
\item $H$ is the group of elements of $G$ fixing some prescribed point in $\{1,2,\dots,\ell\}$.
\item $H$ is the group of elements of $G$ preserving some prescribed two-element subset of $\{1,2,\dots,\ell\}$,
 and the ramification type of the natural projection $\tilde X/H\ra \mP^1$ is given in Table~\emph{\ref{table:f}}. In all of these cases $g_{\tilde X/H}=0$.
\item $g_{\tilde X/H} > a\ell$.
\end{enumerate}
\end{thm}
Note that part (1)  corresponds to the generic case (1) in Theorem \ref{thm:general}, and part (2) corresponds to part (2) of Theorem \ref{thm:general}. 
Here as well, it is necessary to include the exceptional ramification types in (2) since these are shown to occur in Section \ref{sec:ramification} and  Proposition \ref{prop:An-Sn-ram-types}. Moreover in part (3), the linear growth rate in $\ell$ is optimal since there are many other families of coverings with monodromy group $A_\ell$ or $S_\ell$ for which $g_{\tilde X/H}$ grows linearly with $\ell$ when $H$ is the stabilizer of a two-element set, 
see Example \ref{exam:almost-Gal}.

Theorem~\ref{thm:Sn}  proves \cite[Conjecture 1.0.4]{GS}
of Guralnick--Shareshian for sufficiently large $\ell$. 
In \cite{GS} (resp., \cite{Br}), Theorem \ref{thm:Sn} is proven for all coverings $\tilde f$ with at least five (resp., four) branch points\footnote{See Remark \ref{rem:GS-correct} for the required minor adjustments to the lists of ramification types in \cite{GS} and \cite{Br}.}. However, it does not seem feasible to extend the method of \cite{GS} and \cite{Br} to treat the most difficult case where $\tilde f$ has three branch points. Also, the fixed point ratio method, which was used for (almost) simple groups other than $A_\ell$ or $S_\ell$, does not seem useful in this case. 

\subsubsection*{About the proof} 
The proof of Theorem \ref{thm:Sn} does not rely on the classification of finite simple groups.
In contrast, the papers \cite{GS} and \cite{Br}
implicitly appeal to this classification by applying its consequences about $3$-homogeneous groups. Instead of the classification, we use a classical result of Jordan \cite{Jordan}, whose proof was simplified by Babai--Seress \cite{BS}. Namely, there exists a constant $c_1>0$ such that every group $G\leq S_\ell$ which is $t$-transitive for all $t\leq c_1(\log\ell)^2/\log\log\ell$ must be $G=A_\ell$ or $S_\ell$. 
Let $X_t$ be the quotient $\tilde X/\oline H_t$ by a stabilizer $\oline H_t$ of a set of cardinality $t$.
We combine Jordan's theorem with character theoretic results of Guralnick--Shareshian \cite{GS} and Livingstone--Wagner \cite{LW}, to reduce the problem to proving that the difference $g_{X_t}-g_{X_{t-1}}$ is large for every $2\leq t\leq k$, where $k=\ceil{c_1(\log\ell)^2/\log\log\ell}$. This is in contrast to~\cite{GS}, where $g_{X_t}-g_{X_{t-1}}$ is shown to be large for $t=2$ and $3$, and then the classification of $3$-homogenous groups (and hence the classification of finite simple groups) is applied.

Note that there is no natural projection from $X_t$ to $X_{t-1}$, but there is a correspondence 
\[
\xymatrix{
& Y_t \ar[dl] \ar[dr]^{\pi_t} &  \\
X_{t-1} & & X_t,
}
\]
where $\pi_t:Y_t\ra X_t$ is the natural projection from the quotient $Y_t:=\tilde X/H_t$ by a $t$-point stabilizer  $H_t$. 
The difficulty in showing that $g_{X_t}-g_{X_{t-1}}$ is large lies in bounding the difference $g_{Y_t}-t!g_{X_t}$ (or equivalently the Riemann--Hurwitz contribution $R_{\pi_t}$ of $\pi_t$), which becomes an extremely complicated expression when described group theoretically. In Sections \ref{sec:set-point} and \ref{sec:large-genus}, we bound the main term of this expression in order to prove the claim when $g_{Y_1}$ is sufficiently large. This method allows us to assume $g_{Y_1}$ is bounded, but breaks down completely when $g_{Y_1}$ is small. 

To overcome the difficulty of estimating $R_{\pi_t}$, we incorporate two main new ingredients. First, we use Castelnuovo's inequality for the genus of a curve on a split surface, which, unlike previous arguments in this topic, is of geometric origin. For simplicity, let us consider here the case $t=2$.  In this case, Castelnuovo's  inequality allows us to bound $g_{Y_2}$ in terms of $g_{X_2}$  and $g_{Y_1}$, see Section~\ref{sec:castel}, 
yielding a linear bound $g_{Y_2}<c\ell -d$ for some constants $c,d>0$. 

Our second key new ingredient is a proof that a linear bound on $g_{Y_2}$ forces the ramification of the natural projection $h:Y_1\ra  \mP^1$ to be similar to that of a Galois extension, that is,  all but a bounded number of preimages in $h^{-1}(P)$ have the same ramification index under $h$ for each point $P$ of $\mP^1$. The latter argument appears in Section \ref{sec:almost-Galois} and is based on an idea from Do--Zieve~\cite{DZ}.

The final steps of the proof involve determining which ramification data for $h$ that are similar to those of Galois extensions  make the difference $g_{X_2}-g_{X_1}$ small. 
The list of such ramification data is quite large, and we use an explicit version of the so-called ``translation" method (Lemma \ref{lem:hurwitz1}) in order to show that many of these ramification data do not correspond to an indecomposable covering $h$. This shows that the ramification of $h$ appears in Table \ref{table:two-set-stabilizer}. 

Theorem  \ref{thm:Sn} and consequently  Theorem \ref{thm:general} are deduced from Theorem \ref{thm:main} in Section \ref{sec:reduction}.  Theorem \ref{thm:main} summarizes the above genus bounds obtained in Sections \ref{sec:set-point}-\ref{sec:ram}, and is proved in Section \ref{sec:proof}. In Sections \ref{sec:ramification}, \ref{sec:ram}, and Appendix \ref{sec:app}, we list the exceptional ramification types from parts (2) and (4) of Theorem \ref{thm:general} and show they correspond to rational maps.

\subsubsection*{Acknowledgements} We thank Ted Chinburg, Thao Do,  Robert Guralnick,  Ariel Leitner, and Tali Monderer for helpful comments.
The first author is grateful for the support of  ISF grants No.\ 577/15, 353/21, and the NSF for support under a Mathematical Sciences Postdoctoral Fellowship. The second author thanks the NSF for support via grant DMS-1162181. We also thank the United States-Israel Binational Science Foundation (BSF) for its  support under Grant No.\ 2014173, and the SQuaREs program of the American Institute of Math.


\section{Preliminaries and notation}
\subsection{Monodromy and ramification}\label{sec:prelim} Fix an algebraically closed field $\K$ of characteristic $0$.  
A nonconstant morphism $h:{Y_1}\ra Y$ of (smooth projective) curves over $\K$ is called {\it a covering}. 
A covering $\tilde h:\tilde {Y_1}\ra Y$ is called {\it the Galois closure of $h$} if it is a minimal Galois covering which factors as $\tilde h = h\circ h_0$, for some covering $h_0:\tilde Y_1\ra Y_1$.   
The {\it monodromy group}  $G=\Mon(h)$ of $h$ is then defined to be the automorphism group  $\Aut(\tilde h)$. 
Note that $G$ is a permutation group of degree $\deg h$ via its action on the set $S:= H\backslash G$, 
where $H:=\Mon(h_0)$. 
All group actions 
are right actions. Permutation multiplication is left to right, e.g.~$(1,2)(1,3)=(1,2,3)$. All facts stated in this section are derived from a basic reference on function fields, e.g.~\cite{Stich}, via the correspondence between function fields and curves  \cite[Section~7.1]{Fulton}.   

For every subgroup $H\leq G$, the covering $\tilde h$ induces a covering $\tilde Y_1/H\ra Y$ which we call the natural projection. If $\tilde h = h \circ h_0$ is as above, 
then $h_0$ is Galois, and $h$ is equivalent to the natural projection $\tilde Y_1/H\ra Y$ for $H=\Aut(h_0)$, that is, there is an isomorphism $\mu:\tilde Y_1/H\ra Y_1$ such that $h\circ \mu$ is the natural projection $\tilde Y_1/H\ra Y$. 

The subgroup $I$ consisting of all elements $\sigma\in G$ such that $\tilde Q^\sigma = \tilde Q$ is called {\it the decomposition subgroup} of $\tilde Q$. It coincides with the {\it inertia subgroup} of $\tilde Q$ since $\K$ is algebraically closed. 
Since $G$ acts transitively on $\tilde h^{-1}(P)$ and the stabilizer of $\tilde Q$ is $I$, the action of $G$ on $\tilde h^{-1}(P)$ is equivalent to its action on $I\backslash G$, and in particular all decomposition subgroups of points in $\tilde h^{-1}(P)$ are conjugate (as point stabilizers). 
Since $\K$ is of characteristic $0$, $\tilde h$ is tamely ramified, and hence all inertia subgroups   are cyclic.  
A generator of an inertia subgroup of a place in $\tilde h^{-1}(P)$ is called a {\it branch cycle} over $P$. 

The following is a well known description of the points in $h^{-1}(P)$ using $\Mon(h)$. Let $I\backslash G/H$ denote the double coset space $IxH,x\in G$, and let $e_h(Q)$ denote the {\it ramification index} of a point $Q$ of $Y_1$ under $h$. 
\begin{lem}\label{lem:correspondence} Let $h:Y_1\ra Y$ be a covering and $\tilde h:\tilde Y_1\ra Y$ a Galois covering that factors as $\tilde h =h\circ h_0$. Let $G:=\Mon(h)$  and $H:=\Mon(h_0)$. Fix points $P$ of $Y$ and  $\tilde Q\in \tilde h^{-1}(P)$, and let $I$ be the decomposition subgroup of $\tilde Q$. Then there is a natural bijection $\psi_{H,\tilde Q}:h^{-1}(P)\ra I\backslash G/H$ such that $e_h(Q) = \abs{\psi_{H, \tilde Q}(Q)}/\abs H$ for all $Q\in h^{-1}(P)$. 
\end{lem}
\begin{proof} 
As noted earlier there is an isomorphism  $\phi_{\tilde Q}:\tilde h^{-1}(P)\ra I\backslash G$ of $G$-sets such that $\phi_{\tilde Q}(\tilde Q^\sigma)= I\sigma$ for all $\sigma\in G$. Let $Q\in h^{-1}(P)$ and choose $\sigma \in G$ such that $\tilde Q^\sigma\in h_0^{-1}(Q)$. 
Since  $h_0^{-1}(Q)$ coincides with the orbit of $H$ on $\tilde Q^\sigma$, one has  
  $\tilde Q^\tau\in h_0^{-1}(Q)$ if and only if $I\sigma H=I\tau H$. Thus $\phi(h_0^{-1}(Q))$ is 
the orbit of $H$ on $I\sigma$, that is, the double coset $I\sigma H$. Thus, $\psi_{H,\tilde Q}(Q):=\phi_{\tilde Q}(\tilde Q^\sigma) = I\sigma H$ is well defined, as it is independent of the choice of $\sigma$, and clearly gives a one to one correspondence. 
It also follows that the cardinality of the orbit $h_0^{-1}(Q)$ equals the cardinality of the orbit of $H$ on $I\backslash G$ which is $\abs{I\sigma H}/\abs{I}$. Since $e_{h}(Q) = e_{\tilde h}(\tilde Q^\sigma)/e_{h_0}(\tilde Q^\sigma)$, since $e_{\tilde h}(\tilde Q^\sigma)=\abs I$ as $\tilde h$ is Galois, and since $e_{h_0}(\tilde Q^\sigma) = \abs H/\abs{h_0^{-1}(Q)}$ as $h_0$ is Galois, we have
\[
 e_h(Q) = \frac{e_{\tilde h}(\tilde Q^\sigma)}{e_{h_0}(\tilde Q^\sigma)} = \frac{\abs I}{\abs H/\abs{h_0^{-1}(Q)}}= \frac{\abs{I\sigma H}}{\abs H}. \qedhere
 \]
\end{proof}
Moreover,  note that if $h=h_2\circ h_1$ so that $\tilde h = h_2\circ h_1\circ h_0$, and if $H\leq H_2\leq G$ is a point stabilizer of $h_2$ so that $H_2=\Mon(h_1\circ h_0)$, then the bijection $\psi_{H,\tilde Q}$ induces the bijection $\psi_{H_2,\tilde Q}$. Indeed by definition $\psi_{H_2,\tilde Q}(h_1(Q)) = I\sigma H_2$ whenever $\psi_{H,\tilde Q}(Q)=I\sigma H$, for all $Q\in h^{-1}(P)$.

The {\it ramification type} of $P$ under $h$ is the multiset $E_h(P):=[e_h(Q_1),\ldots, e_h(Q_r)]$, where $Q_1,\ldots, Q_r$ are the preimages in  $h^{-1}(P)$.  
Denote by $\Orb_S(U)$  the set of orbits on $S$ of the group generated by a subset $U\subseteq G$. Then Lemma \ref{lem:correspondence} implies that  $E_h(P)$ equals the multiset of lengths of orbits in  $\Orb_{S}(x)$, where $x$ is a branch cycle over $P$ and $S=H\backslash G$. The {\it ramification type} of $h$ is then defined to be the multiset $\{E_h(P)\neq [1^\ell]\suchthat P\in Y(\K)\}$. 

The Riemann--Hurwitz formula $2(g_{Y_1}-\ell g_Y+\ell-1)=\sum_{P\in Y(\K)}R_h(P)$ describes the {\it genus} $g_{Y_1}$ in terms of $g_Y$ and the contributions $R_h(P)$, where
 \[
 R_h(P) := \sum_{r\in E_h(P)}(r-1) = \ell-\abs{\Orb_S(x)}.
 \]
For a set $T$ of points in $Y$ we denote by $R_h(T)$ the sum $\sum_{P\in T}R_h(P)$. 
We note that if $\hat h = h \circ \pi$ for a covering $\pi$ of degree $m$, then $R_{\hat h}(P)$ satisfies the chain rule:
\begin{equation}\label{equ:chain} 
\begin{split} R_{\hat h}(P) & =   \sum_{\hat Q\in \hat h^{-1}(P)}(e_{\hat h}(\hat Q)-1) = \sum_{\hat Q\in \hat h^{-1}(P)}\biggl(e_{\pi}(\hat Q) -1 + e_\pi(\hat Q)\left(e_{h}(\pi(\hat Q))-1\right) \biggr) \\ 
&  =  \sum_{Q\in h^{-1}(P)}\biggl( R_{\pi}(Q) \,+ \,(e_{h}(Q) -1)\,\cdot\!\!\sum_{\hat Q\in \pi^{-1}(Q)}e_\pi(\hat Q) \biggr)
\\
& = R_\pi(h^{-1}(P)) + mR_{h}(P).
\end{split}
\end{equation}

A \emph{ramification data} of degree $\ell$ is a finite multiset
$\{A_1,\ldots,A_r\}$ of partitions of $\ell$, so that $\ell=\sum_{\alpha\in A_i}\alpha$ for each $i$. A ramification data $A_1,\ldots,A_r$ is said to be of {\it genus $g$} if $2(g+\ell-1)=\sum_{i=1}^r(\ell-\abs{A_i})$. It is not known which ramification data arise as ramification types of coverings. We write $E_h(P) = [e_1,\ldots,e_r,e^*]$ to denote that $E_h(P)$ is the union of the multiset $[e_1,\dots,e_r]$ with a (possibly empty) multiset consisting of copies of $e$. 

\subsection{Setup}\label{sec:setup}
We use the following basic setup throughout the paper. Let $h:Y_1\ra Y$ be a covering of degree $\ell$ with monodromy group $G=\Mon(h)$ acting on a set $S=H\backslash G$ such that $G\cong A_\ell$ or $S_\ell$ as permutation groups. Let $\tilde h_1:\tilde Y_1\ra Y$ be the Galois closure of $h$. 

Enumerating the elements in $S$ by $\{1,\ldots,\ell\}$, we let $H_t$ be the pointwise stabilizer of $1,\ldots,t$, and $\oline H_t$ the setwise stabilizer of $\{1,\ldots,t\}$, so that the action of $G$ on $\oline H_t\backslash G$ (resp., $H_t\backslash G$) is equivalent to its action on the set $\oline S^{(t)}$ (resp., $S^{(t)}$) of $t$-element subsets (resp., ordered $t$-tuples of pairwise distinct elements) in $S$ for $t\leq \ell$.  Henceforth we write $t$-set to mean $t$-element subset.
Let $\pi_t:Y_t\ra X_t$, $h_t:Y_t\ra Y$, $f_t:X_t\ra X_0$ be the natural projections from $Y_t:=\tilde Y_1/H_t$, and $X_t:=\tilde Y_1/\oline H_t$, so that $\Mon(f_t)$ (resp., $\Mon(h_t)$) is $G$ with its action on $\oline S^{(t)}$ (resp., $S^{(t)}$) for $t\leq \ell$. The following diagram is commutative: 
 \[
\xymatrix{
& \tilde Y_1\ar[dr]^{\oline H_t} \ar[dl]_{H_t} & \\
 Y_t \ar[dr]_{h_t} \ar[rr]^{\pi_t} & & X_t \ar[dl]^{f_t} \\
  & Y & 
}
\]
Note that the original covering $h$ is equivalent to the covering $h_1$. 


\section{Reduction to genera of set stabilizers}\label{sec:reduction}

In this section we derive Theorem \ref{thm:Sn} from the following theorem, and then prove Theorem \ref{thm:general}. This section follows \cite{GS}.
We use Setup \ref{sec:setup}:
for a degree $\ell$ covering $h:Y_1\ra Y$ with monodromy group $G$ acting on a set $S$, let $X_t$ be the quotient by the stabilizer $\oline H_t\leq G$ of a subset of $S$ of cardinality $t$, for every integer $0\leq t\leq \ell/2$.
\begin{thm}\label{thm:main}
 There exist constants $c,d>0$ 
 such that for every covering  $h:Y_1\ra Y$ of degree $\ell$ and  monodromy group $G \in\{A_\ell, S_\ell\}$, one has 
\[
g_{X_t}-g_{X_{t-1}}>(c\ell-dt^{15})\frac{\binom{\ell}{t}}{\binom{\ell}{2}}
\]
 for all integers $2\leq t\leq \ell/2$ if the ramification type of $h$ does not appear in Table~\ref{table:two-set-stabilizer}, and for all integers $3\leq t\leq \ell/2$ otherwise. 
\end{thm}
Note that the proof also gives $g_{Y_1}=0$ if the ramification of $h$ appears in Table \ref{table:two-set-stabilizer}. 
\begin{proof}[Proof of Theorem \ref{thm:Sn} assuming Theorem \ref{thm:main}]
Let $\log$ denote the natural logarithm. 
Let $c_1:=8\log 2$ be the constant given in \cite{BS}. Thus by setting $M_\ell:=\ceil{c_1(\log\ell)^2/\log\log\ell}$, every $M_\ell$-transitive subgroup of $S_\ell$  contains $A_\ell$. 
 
Let $c$ and $d$ be the constants from Theorem \ref{thm:main}. 
Put $a:=c/2$ and note that as $M_\ell$ is subpolynomial in $\ell$,  
there exists a constant $N_{c,d}$ such that $a\ell>dt^{15}$ 
for every $2\leq t\leq M_\ell$ and $\ell\geq N_{c,d}$. 
Put $N:=\max\{183,N_{c,d}\}$. 

 
Let $f:X\ra Y$ be an indecomposable covering with monodromy group $G=A_\ell$ or $S_\ell$, with $\ell\geq N$, and  $g_Y=0$. Let $H\leq G$ be the monodromy group of the natural projection $\tilde X\ra X$ from the Galois closure $\tilde X$ of $f$. 
Let $Y_1$ be the quotient of $\tilde X$ by the point stabilizer $H_1\leq G$ in the natural action on $\{1,\ldots,\ell\}$, 
and $h:Y_1\ra Y$ the natural projection which has monodromy group $G$, equipped with an action on the set $S = H_1\backslash G$ of cardinality $\ell$. As a subgroup of $S_\ell$, the group $H$ has a natural action on the set $\oline S^{(t)}$ of all subsets of $S$ of cardinality $t$, for $t=0,\ldots, \ell$. Put $\abs{\Orb_{\oline S^{(0)}}(H)}=1$. 

As $\ell\geq 5$,  \cite[Lemma 2.0.13]{GS} gives
\begin{equation}\label{equ:XkXk-1}
g_X \geq  \sum_{1\leq t \leq \lfloor \ell/2\rfloor} \Bigl(\abs{\Orb_{\oline S^{(t)}}(H)}-\abs{\Orb_{\oline S^{(t-1)}}(H)}\Bigr)\cdot\bigl(g_{X_t}-g_{X_{t-1}}\bigr).
\end{equation}
Moreover,  $\abs{\Orb_{\oline S^{(t)}}(H)}\geq \abs{\Orb_{\oline S^{(t-1)}}(H)}$ and $g_{X_{t}}\geq g_{X_{t-1}}$ for $1\leq t\leq \lfloor \ell/2\rfloor$, by the Livingstone--Wagner theorem \cite[Theorem 1]{LW}, and Guralnick--Shareshian \cite[Lemma~2.0.12]{GS}\footnote{The inequality $g_{X_k}\geq g_{X_{k-1}}$ holds more generally for every $k$-transitive subgroup $G\leq S_\ell$, as an immediate consequence of \cite[Theorem 4.34]{Mul2} and \cite[Lemma 4.15]{Mul3}.}, respectively, where the latter uses the assumption $g_Y=0$. 

Combining this with the assumptions  
$\abs{\Orb_{\oline S^{(0)}}(H)}=1$, $\ell\geq 5$, and that $H$ does not contain $A_\ell$, \eqref{equ:XkXk-1} gives
\begin{equation}\label{equ:k-difference}
g_X \geq  (\abs{\Orb_{\oline S^{(t)}}(H)}-\abs{\Orb_{\oline S^{(t-1)}}(H)})(g_{X_t}-g_{X_{t-1}})
\end{equation} for $1\leq t\leq \lfloor \ell/2 \rfloor$.
The proof splits into cases according to types of actions of $H$ on $S$.

\noindent {\bf Case 1}:  $H$ acts intransitively on $S$. Then $H$ must be the stabilizer of a $t$-set for some $2\leq t\leq \ell/2$: for, since $H$ permutes each of its orbits, $H$ is contained in the stabilizer (in $G$) of each orbit.  Since $H$ is intransitive,
there are at least two such orbits, so at least one of them has size at most $\ell/2$, and the
stabilizer of any orbit is a proper subgroup of $G$.  Since $f$ is indecomposable, $H$ is maximal, and hence equals
the stabilizer of each of its orbits, and in particular equals the stabilizer of an orbit which has size $t\leq \ell/2$.

If $t = 2,3,$ the theorem follows from Theorem \ref{thm:main} as then either 
\begin{equation}\label{equ:temp} 
g_{X_t}> (c\ell-dt^{15})>(c\ell-a\ell) = a\ell,
\end{equation} or $t=2$ and the ramification of $h$ appears in Table \ref{table:two-set-stabilizer}. 
If $4\leq t\leq \lfloor \ell/2\rfloor$, then as above $g_X=g_{X_t}\geq g_{X_3}$, and the claim follows from \eqref{equ:temp} with $t=3$. As noted in Section \ref{sec:ramification},  $g_{\tilde X/H}=0$ if $t=2$ and the ramification of $h$ appears in Table \ref{table:two-set-stabilizer}. 

\noindent{\bf Case 2}: $H$ acts transitively on $S$. 
By Theorem \ref{thm:main}, as $a\ell>dt^{15}$ we have
\begin{equation}\label{equ:from-main} 
 g_{X_t}-g_{X_{t-1}}\geq (c\ell-dt^{15})\frac{\binom{\ell}{t}}{\binom{\ell}{2}}> (c\ell-a\ell)\frac{\binom{\ell}{t}}{\binom{\ell}{2}}\geq a\ell,
\end{equation}
for $\ell\geq N_{c,d}$, and $2<t\leq M_\ell$, 
and also for $t=2$ if the ramification of $h$ does not appear in Table~\ref{table:two-set-stabilizer}.  
Thus the claim follows from \eqref{equ:k-difference} if $\abs{\Orb_{\oline S^{(t)}}(H)}>\abs{\Orb_{\oline S^{(t-1)}}(H)}$ for some $t\in\{2,3,\ldots,M_\ell\}$. Since in addition $\abs{\Orb_{\oline S^{(1)}}(H)}=1$ as $H$ is transitive, henceforth assume
\begin{equation}\label{equ:(4)}
\begin{array}{l}
\abs{\Orb_{\oline S^{(t)}}(H)}=\abs{\Orb_{\oline S^{(2)}}(H)}\text{ for all }2\leq t\leq M_\ell, 
\text{ and either } \\
\abs{\Orb_{\oline S^{(2)}}(H)}=1 
 \text{ or the ramification type of $h$ appears in Table \ref{table:two-set-stabilizer}}.
\end{array}
\end{equation}

\noindent {\bf Case 2a}: $H$ is imprimitive.
Then $H$ must be the stabilizer in $G$ of a partition into  $t$ parts of size $\ell/t$, where $1<t<\ell$:
for, $H$ preserves some such partition, and the stabilizer of such a partition is a proper subgroup
of $G$, so since $H$ is a maximal subgroup it must equal this stabilizer.

Note that $\abs{\Orb_{\oline S^{(2)}}(H)}=2$: for, $H$ acts doubly transitively on the set of parts of the partition, and also acts doubly
transitively on the elements in a given part, so the two orbits are [two points in the same part]
or [two points in different parts].

Also $\abs{\Orb_{\oline S^{(3)}}(H)}=3$ if $2<t<\ell/2$: 
for, $H$ acts as $S_t$ on the set of parts, and also acts as
$(S_{\ell/t})^{t-1}$ on any prescribed $t-1$ parts.  So 
the orbits are 
[three points in different parts], [two in one part and one in another], [three points in one part].

Similarly, if $t=2$ or $\ell/2$ 
we have $\abs{\Orb_{\oline S^{(4)}}(H)}=3$: for, if $t=2$ the orbits are [four points in one part], [three points in one part and one in the other], [two points in both parts], and if $t=\ell/2$ the orbits are [four points in different parts], [two points in one part, and two in other two parts], [two points in two parts].

We get that $\abs{\Orb_{\oline S^{(3)}}(H)}>\abs{\Orb_{\oline S^{(2)}}(H)}$ if $2<t<\ell/2$ and $\abs{\Orb_{\oline S^{(4)}}(H)}> \abs{\Orb_{\oline S^{(2)}}(H)}$ if $t=2$ or $\ell/2$, contradicting  \eqref{equ:(4)}.

\noindent {\bf Case 2b}: $H$ is primitive. In this case the condition $\abs{\Orb_{\oline S^{(3)}}(H)}=\abs{\Orb_{\oline  S^{(2)}}(H)}$ implies that $\abs{\Orb_{\oline  S^{(2)}}(H)}=1$,
by \cite[Proposition on p.165]{CNS}.
Hence $\abs{\Orb_{\oline S^{(t)}}(H)}=1$ by \eqref{equ:(4)} for all $1\leq t\leq M_\ell$. 
By \cite[Theorem 2]{LW}, this condition implies that $H$ is $M_\ell$-transitive. 
By  \cite{BS}, 
every $M_\ell$-transitive group 
$H$ is $A_\ell$ or $S_\ell$, contradicting the assumption $H\neq A_\ell,S_\ell$. 
\end{proof}

\begin{proof}[Proof of Theorem \ref{thm:general}] 
Let $G$ denote the monodromy group of $f$ and $n=\deg f$. Since $f$ is indecomposable, $G$ is primitive. 
The Aschbacher--O'Nan--Scott theorem \cite{GT} divides primitive groups $G$ according to the interaction of their $1$-point stabilizer $H\leq G$, and a minimal normal subgroup $Q$ of $G$. Let $L$ be a minimal normal subgroup of $Q$, and $L_1=L,L_2,...,L_t$ the conjugates of $L$ in $G$. Then one of the following holds:
\begin{enumerate}
\item[(A)]  $\abs{L}$ is prime; 
\item[(B)]  $G$ has more than one minimal normal subgroup; 
\item[(C)] $Q$ is the unique minimal normal subgroup of $G$, and $Q=L_1\times L_2\times\dots\times L_t$, where
$L$ is simple nonabelian and either
\item[(C1)] $H\cap Q=1$; or
\item[(C2)] $H\cap Q\ne 1$ but $H\cap L = 1$; or
\item[(C3)] $H\cap Q=H_1\times \cdots \times H_t$, where $H_i=H\cap L_i\ne 1$ for $1\leq i\leq t$.
\end{enumerate}

%
In case (A),  Neubauer's proof of \cite[Theorems 1.4-1.6]{Neu1}, which refines Guralnick--Thompson \cite[Theorem A]{GT}, 
shows that for sufficiently  large $n$, either $g_X>n/1024$ or $g_X=0$ and $G$ is of derived length at most $2$. In the latter cases,   \cite[Proposition 3.8]{GT} combined with the formula for the genus of a Galois closure (Remark \ref{rem:abh})  show that in these cases the genus $g_{\tilde Y_1}$ of the Galois closure $\tilde f:\tilde Y_1\ra Y_1$ is at most  $1$. 
Proposition \ref{lem:normal-closure} gives all possible ramification types for indecomposable $h$ and its corresponding monodromy groups $G$ in case $g_{\tilde Y_1}\leq 1$. 
In Case (B), Shih \cite{Shih} shows that $g_X>0$, 
and as shown in \cite{GMN} (and predicted in \cite[pg.\ 353]{Gur1}) his proof yields $g_X>cn-d$ for some $c,d>0$. 
In case  (C1), the proof of Guralnick--Thompson \cite[Theorem C1]{GT}   
 gives $g_X>n/2000$. 
In case (C2),  Aschbacher \cite{Asch} shows $g_X> n/336$ when $n$ is larger than a constant. 

Henceforth assume $G$ is as in case (C3).
For $t>8$, the theorem follows from Guralnick--Neubauer \cite[Corollary 8.7]{GN}, which shows that $g_X> (1/1250)n^{1-1/t}$. 
Consider the case $t=1$, so that $L\leq G\leq \Aut(L)$ for a finite simple group $L$, and first assume $L$ is nonalternating. By assuming $n$ is sufficiently large, we may assume $L$ is of Lie type. Let $q$ be the cardinality of the field over which $L$ is defined. 
Letting $\fpr(G)$ denote the maximal ratio $\abs{\{\text{fixed points of }x\}}/n$ over all $x\in G\setminus\{1\}$, Guralnick \cite[Theorem 1]{Gur1} shows that for any $0<\eps<1/85$, there exists a constant $c_\eps>0$ such that if $\fpr(G)<\eps$ then $g_X>c_\eps n$. 
Assuming $q\geq 23$, Liebeck--Saxl \cite[Theorem 1]{LSa} show that  $\fpr(G)\leq 4/(3q)$ if $L\neq \PSL_2(q)$, and that $\fpr(G)<(\sqrt{q}+1+(2/\sqrt{q}))/(q+1)$  if $L=\PSL_2(q)$. 
Pick $0<\eps<1/85$ such that $4/(3q)<\eps$ for $q>113$ and $(\sqrt{q}+1+(2/\sqrt{q}))/(q+1)<\eps$ for $q>(86)^2$, so that $\fpr(G)<\eps$ for such $q$. 
If the action of $G$ is not a subspace action, then Liebeck--Shalev \cite[Theorem~1.1]{LSh}, extending 
\cite{LWay}, show that $\fpr(G)<\eps$ for sufficiently large $n$.
If the action of $G$ is a subspace action, then Frohardt--Magaard \cite[Proposition 5.1]{FM} show that there exists a constant $N_q$, depending only on $q$,  such that  $g_X>n/2000$ for $n\geq N_q$. 
Thus, letting $\tilde c_\eps:=\min\{c_\eps,1/2000\}$, one has $g_X>\tilde c_\eps n$ for sufficiently large $n$, by the combination of \cite{Gur1} and \cite{LSa} if $q>113$ for $L\neq \PSL_2(q)$ and $q>(86)^2$ for $L=\PSL_2(q)$, and by the combination of \cite{LSh} and \cite{FM} otherwise.

Finally, the proof in case (C3) is completed for $t=1$ by Theorem \ref{thm:Sn} which covers the cases where $L$ is alternating, and for $t>1$ (and in particular for $2<t\leq 8$) by \cite[Theorem~1.2]{NZ2} which gives positive constants $c_t,d_t$, depending only on $t$, such that either $t=2$ and the ramification of $f$ appears in \cite[Table 3]{NZ2}, or $g_X>c_tn^{1-1/t}-d_t$. 
\end{proof}
The classification of finite simple groups is used in cases (B), (C1), (C2), and the case (C3) with $t=1$. It seems plausible that a classification free proof will be given in all cases but (C3) with $t=1$, that is, for all but almost simple monodromy groups. 


\section{The ramification types in Theorems \ref{thm:general}, \ref{thm:Sn}, and \ref{thm:main}}\label{sec:ramification}
In this section we complete the statements of Theorems \ref{thm:general}, \ref{thm:Sn}, and \ref{thm:main} by presenting the ramification types excluded in the
statement of Theorem \ref{thm:main}, as well as the ramification types from case (2) of Theorems \ref{thm:general} and \ref{thm:Sn}.
Specifically, in Table~\ref{table:two-set-stabilizer} we present the ramification types of some degree-$\ell$ coverings $h:\mP^1\ra \mP^1$ having monodromy group $A_\ell$
or $S_\ell$, and in Table~\ref{table:f} we present the ramification types of the degree-$\ell(\ell-1)/2$ coverings $f:X\ra \mP^1$ obtained from the coverings in
Table~\ref{table:two-set-stabilizer} by means of the action of $A_\ell$ or $S_\ell$ on the set of two-element subsets of $\{1,2,\dots,\ell\}$.  Crucially, the ramification
types in Table~\ref{table:two-set-stabilizer} have the property that the induced degree-$\ell(\ell-1)/2$ covering $f:X\ra \mP^1$ has $X$ being of genus zero.
We emphasize that each ramification type listed in Tables~\ref{table:two-set-stabilizer} and \ref{table:f} actually occurs for some indecomposable rational function with
monodromy group $A_\ell$ or $S_\ell$.  We will prove this in Proposition~\ref{prop:An-Sn-ram-types} for the types in Table~\ref{table:two-set-stabilizer}, and in this section
we will indicate how this fact implies the corresponding fact for each ramification type in Table~\ref{table:f}.

\begin{table}
\caption{Ramification types for some coverings $h:\mP^1\ra\mP^1$ of degree $\ell\geq 13$ and monodromy group $A_\ell$ or $S_\ell$.
Here $a\in\{1,\ldots, \ell-1\}$ is odd, $(a,\ell)=1$, and in each type $\ell$ satisfies the necessary congruence conditions to make all exponents integral. The third column refers to the corresponding ramification type in \cite{GS}. }
\begin{equation*}\label{table:two-set-stabilizer}
\begin{array}{| l | l | l |}
\hline
I1.1 & {[\ell], [a,\ell-a], \left[1^{\ell-2}, 2\right]}  & \\
\hline
I2.1 & {[\ell], [1^3,2^{(\ell-3)/2}], [1,2^{(\ell-1)/2}], \left[1^{\ell-2}, 2\right] } & \text{\cite[Proposition 3.0.24(e)]{GS}}\\
I2.2 & {[\ell], [1^2,2^{(\ell-2)/2}] \text{ twice}, \left[1^{\ell-2},2\right]} & \text{\cite[Proposition 3.0.24(c)]{GS}}\\
I2.3 & {[\ell], \left[1^3,2^{(\ell-3)/2}\right], [2^{(\ell-3)/2},3] } & \text{\cite[Proposition 3.0.25(b)]{GS}}\\
I2.4 & {[\ell], \left[1^2,2^{(\ell-2)/2}\right], [1,2^{(\ell-4)/2},3] } & \text{\cite[Proposition 3.0.25(d)]{GS}}\\
I2.5 & {[\ell], \left[1,2^{(\ell-1)/2}\right], [1^2,2^{(\ell-5)/2},3]} & \text{\cite[Proposition 3.0.25(f)]{GS}}\\
I2.6 & {[\ell], \left[1^3,2^{(\ell-3)/2}\right], [1,2^{(\ell-5)/2},4] } & \text{\cite[Proposition 3.0.25(a)]{GS}}\\
I2.7 & {[\ell], \left[1^2,2^{(\ell-2)/2}\right], [1^2,2^{(\ell-6)/2},4] } & \text{\cite[Proposition 3.0.25(c)]{GS}}\\
I2.8 & {[\ell], \left[1,2^{(\ell-1)/2}\right], [1^3,2^{(\ell-7)/2},4]} & \text{\cite[Proposition 3.0.25(e)]{GS}}\\
\hline
I2.9 & {[a,\ell-a], [1^2,2^{(\ell-2)/2}], [2^{\ell/2}], \left[1^{\ell-2},2\right] } & \text{\cite[Proposition 3.0.24(d)]{GS}}\\
I2.10 &  {[a,\ell-a], [1,2^{(\ell-1)/2}] \text{ twice}, \left[1^{\ell-2},2\right]} & \text{\cite[Proposition 3.0.24(f)]{GS}}\\
I2.11 & {[a,\ell-a],\left[2^{\ell/2}\right], [1^2,2^{(\ell-6)/2},4]  } & \text{\cite[Proposition 3.0.27(c)]{GS}}\\
I2.12 & {[a,\ell-a], \left[1,2^{(\ell-1)/2}\right], [1,2^{(\ell-5)/2},4]} & \text{\cite[Proposition 3.0.27(a)]{GS}}\\
I2.13 & {[a,\ell-a],\left[1^2,2^{(\ell-2)/2}\right], [2^{(\ell-4)/2},4] } & \text{\cite[Proposition 3.0.28(c)]{GS}}\\
I2.14 & {[a,\ell-a],\left[1,2^{(\ell-1)/2}\right], [2^{(\ell-3)/2},3] } & \text{\cite[Proposition 3.0.27(b)]{GS}}\\
I2.15 & {[a,\ell-a],\left[2^{\ell/2}\right], [1,2^{(\ell-4)/2},3]  } & \text{\cite[Proposition 3.0.27(d)]{GS}}\\
\hline
F1.1 & {\left[1^{\ell-2},2\right], [2^{\ell/2}], [1^2,2^{(\ell-2)/2}] \text{ thrice} } & \text{\cite[Proposition 3.0.24(a)]{GS}}\\
F1.2 & {\left[1^{\ell-2},2\right], [1^3,2^{(\ell-3)/2}], [1,2^{(\ell-1)/2}] \text{ thrice} } & \text{\cite[Proposition 3.0.24(b)]{GS}}\\
F1.3 & {\left[1^3,2^{(\ell-3)/2}\right], [2^{(\ell-3)/2},3], [1,2^{(\ell-1)/2}] \text{ twice} } & \text{\cite[Proposition 3.0.26(b)]{GS}}\\
F1.4 & {\left[2^{\ell/2}\right], [1,2^{(\ell-4)/2},3], [1^2,2^{(\ell-2)/2}] \text{ twice} } & \text{\cite[Proposition 3.0.26(d)]{GS}}\\
F1.5 & {\left[1^2,2^{(\ell-5)/2},3\right], [1,2^{(\ell-1)/2}] \text{ thrice} } & \text{\cite[Proposition 3.0.26(f)]{GS}}\\
F1.6 & {\left[1^3,2^{(\ell-3)/2}\right], [1,2^{(\ell-5)/2},4], [1,2^{(\ell-1)/2}] \text{ twice} } & \text{\cite[Proposition 3.0.26(a)]{GS}}\\
F1.7 & {\left[2^{\ell/2}\right], [1^2,2^{(\ell-6)/2},4], [1^2,2^{(\ell-2)/2}] \text{ twice} } & \text{\cite[Proposition 3.0.26(c)]{GS}}\\
F1.8 & {\left[1^3,2^{(\ell-7)/2},4\right], [1,2^{(\ell-1)/2}] \text{ thrice} } & \text{\cite[Proposition 3.0.26(e)]{GS}}\\
F1.9 & \left[2^{(\ell-4)/2},4\right], [1^2, 2^{(\ell-2)/2}]\text{ thrice}; & \\
\hline
F3.1 & {\left[1^2,2^{(\ell-2)/2}\right], [1,3,4^{(\ell-4)/4}], [4^{\ell/4}] } & \text{\cite[Conjecture 3.0.29(a)]{GS}}\\
F3.2 & {\left[1,2^{(\ell-1)/2}\right], [1,4^{(\ell-1)/4}], [2,3,4^{(\ell-5)/4}] } & \text{\cite[Conjecture 3.0.29(b)]{GS}}\\
F3.3 & {\left[1,2^{(\ell-1)/2}\right], [1,2,4^{(\ell-3)/4}], [3,4^{(\ell-3)/4}] } & \text{\cite[Conjecture 3.0.29(c)]{GS}}\\
\hline
F4.1 & {\left[1^2,2^{(\ell-2)/2}\right], [1,2,3^{(\ell-3)/3}], [6^{\ell/6}] } & \text{\cite[Conjecture 3.0.29(d)]{GS}}\\
F4.2 & {\left[1^2,2^{(\ell-2)/2}\right], [2,3^{(\ell-2)/3}], [2,6^{(\ell-2)/6}] } & \\
F4.3 & {\left[1,2^{(\ell-1)/2}\right], [1,3^{(\ell-1)/3}], [3,4,6^{(\ell-7)/6}] } & \text{\cite[Conjecture 3.0.29(e)]{GS}}\\
F4.4 & {\left[1,2^{(\ell-1)/2}\right], [1,2,3^{(\ell-3)/3}], [3,6^{(\ell-3)/6}] } & \text{\cite[Conjecture 3.0.29(h)]{GS}}\\
F4.5 & {\left[1^2,2^{(\ell-2)/2}\right], [1,3^{(\ell-1)/3}], [4,6^{(\ell-4)/6}] } & \text{\cite[Conjecture 3.0.29(i)]{GS}}\\
F4.6 & {\left[1,2^{(\ell-1)/2}\right], [2,3^{(\ell-2)/3}], [2,3,6^{(\ell-5)/6}] } & \text{\cite[Conjecture 3.0.29(j)]{GS}}\\
\hline
\end{array}
\end{equation*}\end{table}

\begin{table}
\caption{Ramification types for some coverings $f:\mP^1\ra\mP^1$ of degree $n=\ell(\ell-1)/2$, $\ell\geq 13$, and monodromy group $A_\ell$ or $S_\ell$. 
Here $a$ is an odd number in $\{1,\ldots ,\ell-1\}$, $(a,\ell)=1$, and $\ell$ satisfies the necessary congruence conditions to make exponents integral. The notation $k^*$ means that the rest of the entries in the multiset equal $k$.}
\begin{equation*}\label{table:f}
\begin{array}{| l | l |}
\hline
I1.1a & {[\ell^*], [a(\ell-a), a^{(a-1)/2}, \frac{\ell-a}{2}, (\ell-a)^{*}], \left[2^{\ell-2},1^{*}\right]}  \\
I1.1b & [\frac{\ell}{2},\ell^*], [a(\ell-a), a^{(a-1)/2}, (\ell-a)^{*}], \left[2^{\ell-2},1^{*}\right] \\
\hline
I2.1 & [\ell^*], [1^{(\ell+3)/2}, 2^*], [1^{(\ell-1)/2},2^*], [2^{\ell-2},1^*] \\
I2.2 & {[\frac{\ell}{2}, \ell^*], [1^{\ell/2},2^*]\text{ twice}, [2^{\ell-2},1^*]} \\
I2.3 & {[\ell^*], [1^{(\ell+3)/2}, 2^*], [3, 1^{(\ell-3)/2}, 6^{(\ell-3)/2}, 2^*]} \\
I2.4 & {[\frac{\ell}{2}, \ell^*], [1^{\ell/2},2^*], [3^2, 1^{(\ell-4)/2}, 6^{(\ell-4)/2}, 2^*]} \\
I2.5 & {[\ell^*], [1^{(\ell-1)/2}, 2^*], [3^3, 1^{(\ell-3)/2}, 6^{(\ell-5)/2}, 2^*]} \\
I2.6 & {[\ell^*], [1^{(\ell+3)/2}, 2^*], [1^{(\ell-5)/2}, 4^{\ell-3}, 2^*]} \\
I2.7 & {[\frac{\ell}{2}, \ell^*], [1^{\ell/2}, 2^*], [1^{(\ell-4)/2}, 4^{\ell-3}, 2^*]} \\
I2.8 & {[\ell^*], [1^{(\ell-1)/2}, 2^*], [1^{(\ell-1)/2}, 4^{\ell-3}, 2^*]} \\
\hline
I2.9 & {[a(\ell-a), a^{(a-1)/2},  (\ell-a)^*], [1^{\ell/2}, 2^*]\text{ twice}, [2^{\ell-2},1^*]} \\
I2.10 & {[a(\ell-a), a^{(a-1)/2}, \frac{\ell-a}{2}, (\ell-a)^*], [1^{(\ell-1)/2}, 2^*]\text{ twice}, [2^{\ell-2},1^*]} \\
I2.11 & {[a(\ell-a), a^{(a-1)/2}, (\ell-a)^*], [1^{\ell/2}, 2^*],  [1^{(\ell-4)/2}, 4^{\ell-3}, 2^*]} \\
I2.12 & {[a(\ell-a), a^{(a-1)/2}, \frac{\ell-a}{2}, (\ell-a)^*], [1^{(\ell-1)/2}, 2^*], [1^{(\ell-5)/2}, 4^{\ell-3},2^*]} \\
I2.14 & {[a(\ell-a), a^{(a-1)/2}, \frac{\ell-a}{2}, (\ell-a)^*], [1^{(\ell-1)/2}, 2^*], [3, 1^{(\ell-3)/2}, 6^{(\ell-3)/2}, 2^*]} \\
I2.15 & {[a(\ell-a), a^{(a-1)/2}, (\ell-a)^*], [1^{\ell/2}, 2^*], [3^2, 1^{(\ell-4)/2}, 6^{(\ell-4)/2}, 2^*]} \\
\hline
F1.1 & {[2^{\ell-2},1^*], [1^{\ell/2}, 2^*]\text{ four times}} \\
F1.2 & {[2^{\ell-2},1^*], [1^{(\ell+3)/2}, 2^*], [1^{(\ell-1)/2}, 2^*]\text{ thrice}} \\
F1.3 & {[1^{(\ell+3)/2}, 2^*], [3,1^{(\ell-3)/2}, 6^{(\ell-3)/2}, 2^*], [1^{(\ell-1)/2}, 2^*]\text{ twice}} \\
F1.4 & {[1^{\ell/2}, 2^*]\text{ thrice}, [3^2,1^{(\ell-4)/2}, 6^{(\ell-4)/2}, 2^*]} \\
F1.5 & {[3^3,1^{(\ell-3)/2}, 6^{(\ell-5)/2}, 2^*], [1^{(\ell-1)/2}, 2^*]\text{ thrice}} \\
F1.6 & {[1^{(\ell+3)/2}, 2^*], [1^{(\ell-5)/2}, 4^{\ell-3}, 2^*], [1^{(\ell-1)/2}, 2^*]\text{ twice}} \\
F1.7 & { [1^{(\ell-4)/2}, 4^{\ell-3}, 2^*], [1^{\ell/2}, 2^*]\text{ thrice}} \\
F1.8 & {[1^{(\ell-1)/2}, 4^{\ell-3}, 2^*], [1^{(\ell-1)/2}, 2^*]\text{ thrice}} \\
\hline
F3.1 & {[1^{\ell/2},2^*], [3^2, 2^{(\ell-4)/4}, 12^{(\ell-4)/4}, 4^*], [2^{\ell/4}, 4^*]} \\
F3.2 & {[1^{(\ell-1)/2},2^*], [ 2^{(\ell-1)/4}, 4^*], [1, 3, 6, 2^{(\ell-5)/4}, 12^{(\ell-5)/4}, 4^*]} \\
F3.3 & {[1^{(\ell-1)/2},2^*], [1,2^{(\ell+1)/4}, 4^*], [3,2^{(\ell-3)/4}, 12^{(\ell-3)/4}, 4^*]} \\
\hline
F4.1 & {[1^{\ell/2},2^*], [1,2,6^{(\ell-3)/3},3^*], [3^{\ell/6}, 6^*] } \\
F4.2 &  {[1^{\ell/2},2^*], [1, 6^{(\ell-2)/3}, 3^*],  [1, 3^{(\ell-2)/6}, 6^*] } \\
F4.3 & {[1^{(\ell-1)/2},2^*], [3^*],  [2, 4, 3^{(\ell-1)/6}, 12^{(\ell-4)/3},  6^*] } \\
F4.4 & {[1^{(\ell-1)/2},2^*], [1,2, 6^{(\ell-3)/3}, 3^*],  [3^{(\ell+3)/6}, 6^*] } \\
F4.5 & {[1^{\ell/2},2^*], [3^*],  [2, 4, 3^{(\ell-4)/6}, 12^{(\ell-4)/3},  6^*] } \\
F4.6 & {[1^{(\ell-1)/2},2^*], [1,6^{(\ell-2)/3}, 3^*],  [1, 3^{(\ell+1)/6}, 6^*] } \\
\hline
\end{array}
\end{equation*}
\end{table}

The ramification types for the coverings $f:X_2\ra \mP^1$ in Table \ref{table:f} are obtained from those of $h:Y_1\ra \mP^1$ in Table~\ref{table:two-set-stabilizer} using the following process. 
As in Setup \ref{sec:setup},  $h:Y_1\ra \mP^1$ is a degree $\ell$ covering with monodromy group $G\in\{A_\ell,S_\ell\}$ acting on an $\ell$-element set $S$, and with
Galois closure $\tilde X\to\mP^1$, and 
$f:X_2\ra \mP^1$ is the natural projection from the quotient $X_2:=\tilde X/\oline H_2$ of $\tilde X$ by the stabilizer $\oline H_2$ of a $2$-set in $S$, so that $\Mon(f)=G$ in its action on the set $\oline S^{(2)}$ of $2$-sets in $S$. 

Let $P$ be a point of $\mP^1$ and $x\in G$ a branch cycle  for $h$ over $P$. 
As in Section~\ref{sec:setup}, the ramification type $E_{h}(P)$ (resp., $E_f(P)$) equals the multiset   
of orbit cardinalities $\{\abs o \suchthat o\in \Orb_S(x)\}$ (resp., $\{\abs o \suchthat o\in \Orb_{\oline S^{(2)}}(x)\}$). 
The orbit cardinalities of $x$ on $\oline S^{(2)}$ are deduced from those on $S$ by means of the following lemma: 
\begin{lem}\label{lem:h-to-f}
Let $R_1, R_2\subseteq S$ be orbits of $x\in S_\ell$ having cardinalities $r_1,r_2$, respectively.
Let $T$ be the set of unordered pairs $\{a,b\}$ of distinct elements $a,b$ with $a\in R_1$ and $b\in R_2$.
Then the orbits of the action of $x$ on $T$ consist of
\begin{enumerate}
\item  $(r_1,r_2)$ orbits of cardinality $\lcm(r_1,r_2)$ if $R_1\neq R_2$;
\item $(r_1-1)/2$ orbits of cardinality $r_1$ if $R_1=R_2$ and $r_1$ is odd;
\item one orbit of cardinality $r_1/2$, and $r_1/2-1$ orbits of cardinality $r_1$ if $R_1=R_2$ and $r_1$ is even.
\end{enumerate}
\end{lem}
\begin{proof}
Let $a\in R_1$ and $b\in R_2$.
The orbit of every element $\{c,d\}\in T$ under the action of $x$ is of cardinality $\lcm(r_1,r_2)$ unless $R_1=R_2$, $r_1$ is even, and $a$ is the image of $b$ under the action of $x^{r_1/2}$. 
Since there are $r_1r_2$ (resp., $r_1(r_1-1)/2$) elements in $T$ if $R_1\neq R_2$ 
(resp., $R_1=R_2$), there are $r_1r_2/\lcm(r_1,r_2) = (r_1,r_2)$ (resp., $(r_1-1)/2$) such orbits if $R_1 \neq R_2$ (resp., if $R_1=R_2$ and $r_1$ is odd), proving (1) and (2). 
In case (3), all pairs $\{a,a^{x^{r_1/2}}\}$ are in the same orbit of $x$ which has cardinality $r_1/2$. 
As there are $r_1(r_1-2)/2$ pairs in $T$ which are not of the form $\{a,a^{x^{r_1/2}}\}$, these comprise $r_1/2-1$ orbits, proving (3).  
\end{proof}
Applying this lemma for each of the ramification types of $h$ in Table \ref{table:two-set-stabilizer} over each branch point, 
we obtain the corresponding ramification types for $f$. As noted after Theorem \ref{thm:main}, $g_{Y_1}=0$ for each ramification type in Table \ref{table:two-set-stabilizer}. Hence the formula \eqref{equ:M-RH-t=2} gives $g_{X_2}=0$ for every $f:X_2\ra\mP^1$ whose ramification type appears in Table \ref{table:f}. 

The ramification types for $h$ are labeled according to the case in the proof of Proposition \ref{prop:two-set} from which they arise. The ramification types for $f$ are labeled 
as their corresponding ramification type of $h$
with additional subcases if a single ramification type for $h$ has several corresponding types for $f$, 
depending on whether $\ell$ is odd or even. 
\begin{rem}\label{rem:GS-correct}
Note that Item I1.1 in Table \ref{table:two-set-stabilizer} corresponds to the rational function $X^a(X-1)^{\ell-a}$; Item F1.9 does not appear in \cite{GS}; the slightly incorrect Item (g) of \cite[Conjecture 3.0.29]{GS} (which violates Riemann--Hurwitz) is replaced by Item F4.2; and Item (f) of \cite[Conjecture 3.0.29]{GS} does not correspond to a covering with primitive monodromy group, by the special case of
Lemma \ref{lem:non-existing} corresponding to case F1.N4 of Table \ref{table:non-existence}.
The ramification types F1.7 and F1.9 (resp., I2.11 and I2.13) in Table~\ref{table:two-set-stabilizer} correspond to the same ramification type in Table~\ref{table:f}, so we
do not include an entry for cases F1.9 or I2.13 in Table~\ref{table:f}.
\end{rem}

\section{Relating set and point stabilizers }\label{sec:set-point}

As in Setup \ref{sec:setup}, let $h:Y_1\ra Y$ be a covering of degree $\ell$ with monodromy group $G\in\{A_\ell,S_\ell\}$ acting on a set $S=\{1,\ldots,\ell\}$, let $X_t$ be the quotient by the stabilizer $\oline H_t$ of the set $\{1,\ldots,t\}$, and $Y_t$ the quotient by the pointwise stabilizer $H_t$ of $\{1,\ldots,t\}$.  

A key ingredient in proving Theorem \ref{thm:main} is bounding the contribution $R_{\pi_t}$ of the natural projection $\pi_t: Y_t\ra X_t$. The following proposition describes the main term of the Riemann--Hurwitz contribution from $\pi_t$. This contribution accounts for the extent to which the orbits of branch cycles on $t$-tuples are longer than their orbits on $t$-sets. For $t=2$, such orbits appear whenever the branch cycle has an even length orbit.

For $n\in\mathbb{N}\cup\{0\}$, let $v_p(n)$ be the largest integer such that $p^{v_p(n)}\divides n$.
\begin{prop}\label{lem:pi-form}
There exists a constant $E_0>0$ satisfying the following property. 
Let $t,\ell\in\mathbb{N}$ be integers such that $t\geq 2$ and $\ell>t^2$. 
Let $h:Y_1\ra Y$ be a covering of degree $\ell$ with monodromy group $G\in\{A_\ell,S_\ell\}$.
Let $P$ be a point of $Y$ and $x$ a branch cycle of $h$ over $P$. Then
\begin{enumerate}
\item $R_{\pi_2}(f_2^{-1}(P)) = \abs{\{\text{even }r\in E_{h}(P)\}}$;
\item 
\[
R_{\pi_t}(f_t^{-1}(P)) \leq \binom{t}{2}\sum_{\theta_1,\ldots,\theta_{t-1}} \frac{\hat r_1\cdots \hat r_{t-1}}{\lcm(r_1,\ldots,r_{t-1})} 
+  E_0 t^4\cdot  \frac{(\ell-2)!}{(\ell-t)!}
\]
where $r_i := \abs{\theta_i}$, and $\hat r_i := r_i - \abs{\{j\suchthat j<i, \theta_j = \theta_i\}}$ for $i=1,\ldots,t-1$, and $\theta_1,\ldots, \theta_{t-1}$ run through orbits of $x$ with $r_1$ even and $v_2(r_1)>v_2(r_j)$, for $j=2,\ldots,t-1$. 
\end{enumerate}
\end{prop}
\begin{proof}
{\bf Step I:} {\it We first express $R_{\pi_t}(f_t^{-1}(P))$ using orbits of $x$ on $t$-sets and $t$-tuples. }
As in Section \ref{sec:setup}, let $h_t:Y_t\ra Y$ and $f_t:X_t\ra Y$ denote the natural projections, so that their monodromy act on the set $S^{(t)}$ of $t$-tuples of pairwise distinct elements of $S$ and on the set $\oline S^{(t)}$ of $t$-subsets of $S$, respectively. 

Let $O_1,\ldots, O_s$ be the orbits of $x$ on $\oline S^{(t)}$. 
For a $t$-tuple $U\in S^{(t)}$, denote by $\oline U\in \oline S^{(t)}$  its  underlying $t$-set, and by $r_U$ the length of its orbit.
For given $i\in\{1,\ldots,s\}$, we claim that the length $r_U$ is independent of the choice of a $t$-tuple $U\in S^{(t)}$ for which $\oline U\in O_i$. 
Given a $t$-set $\oline V\in O_i$, the element $x^{\abs{O_i}}$ defines a permutation on the elements of $\oline V$. Let $e_{i}$ be the order of this permutation.  
The orbit of every $t$-tuple $V\in S^{(t)}$ with underlying set $\oline V$ is then of length $r_V=e_i\cdot \abs{O_i}$. 
Since the orbit of every $t$-tuple $U\in S^{(t)}$ with $\oline U\in O_i$ contains such a tuple $V$, we have $r_U=r_V=e_i\cdot \abs{O_i}$, proving the claim. 

Since every $t$-subset $\oline U$ in $O_i$ has $t!$ tuples in $S^{(t)}$ with underlying set $\oline U$, the number of $t$-tuples in $S^{(t)}$ with underlying set in $O_i$ is $t!\cdot \abs{O_i}$, for $i=1,\ldots,s$. 
By the above claim, the action of $x$ on such tuples breaks into equal length orbits $O_i^j$, $j=1,\ldots,s_i$, where $s_i := t!\abs{O_i}/\abs{O_i^1}$, for  $1\leq i\leq s$.
Thus the Riemann--Hurwitz contributions are
\[
R_{f_t}(P) = \sum_{i=1}^s (\abs{O_i} -1) \quad\text{and} \quad
R_{h_t}(P) = \sum_{i=1}^s\sum_{j=1}^{s_i} (\abs{O_i^j}-1). 
\]
 Since $s_i = t!\abs{O_i}/\abs{O_i^1}$ and $\abs{O_i^j}=\abs{O_i^1}$, it follows that:
\[
R_{f_t}(P) = \sum_{i=1}^s (\abs{O_i} -1 ) = \sum_{i=1}^s\frac{1}{s_i}\sum_{j=1}^{s_i} (\abs{O_i}-1)  = 
\frac{1}{t!}\sum_{i=1}^s\sum_{j=1}^{s_i} \Bigl(\abs{O_i^j}-\frac{\abs{O_i^j}}{\abs{O_i}}\Bigr).
\]

Since $R_{h_t}(P) = t! R_{f_t}(P) + R_{\pi_t}(f_t^{-1}(P))$ by the chain rule \eqref{equ:chain}, we get
\begin{equation}\label{equ:pit-form}
R_{\pi_t}(f_t^{-1}(P)) = R_{h_t}(P) -t!R_{f_t}(P)  = \sum_{i=1}^s\sum_{j=1}^{s_i}\Bigl(\frac{\abs{O_i^j}}{\abs{O_i}}-1\Bigr).
\end{equation}

{\bf Step II:} {\it To bound $R_{\pi_t}(f_t^{-1}(P))$ using \eqref{equ:pit-form},
we bound the number of pairs $(i,j)$ for which $\abs{O_i^j}/\abs{O_i}$ is divisible by a given power $q$ of a prime $p$.}
 Let $U:=[u_1,\ldots, u_t]\in S^{(t)}$ be a tuple in $O_i^j$. 
If $q\divides (\abs{O_i^j}/\abs{O_i})$, then $\abs{O_i}$ divides $(r_U/q)$ and hence $x^{r_U/q}$ defines a permutation of order $q$ on $\oline U$. As $q$ is a prime power this implies that $q$ and $U$ satisfy the condition
\begin{enumerate}\item[$(\text{\emph{Cyc}})_{q,U}$] There exist $v_1,\ldots,v_q\in  \oline U$ that form a cycle under the action of $x^{r_U/q}$. 
\end{enumerate}

We next bound the total number of tuples $U\in S^{(t)}$ satisfying $(\text{\emph{Cyc}})_{q,U}$ for a given $q$. 
Note that if $v_1,\ldots,v_q$ is a cycle of $x^{r_U/q}$, then it is also a cycle of $x^{r_1/q}$ where $r_1$ is the length of the orbit of $v_1$. 
Thus, a choice of $v_1$ determines  $v_2,\ldots,v_q$ as the images of $v_1$ under $x^{r_1/q}$. 
Thus, there are at most $\ell$ choices for the first entry $v_1$ and these determine $v_2,\ldots,v_q$; there are $t!/(t-q)!$ choices for the positions of $v_1,\ldots,v_q$ in $U$; and $(\ell-q)!/(\ell-t)!$ choices for the rest of the entries in $U$. 
Hence in total the number of tuples $U\in S^{(t)}$ satisfying $(\text{\emph{Cyc}})_{q,U}$ is at most $ \frac{t!}{(t-q)!} \ell\frac{(\ell-q)!}{(\ell-t)!}$. 
Since the orbit of each such $U$ is of length at least $q$, we have
\begin{equation}\label{equ:largest-power} 
\abs{\{(i,j) \suchthat q\text{ divides $\abs{O_i^j}/\abs{O_i}$}\}} \leq  \frac{1}{q}\frac{t!}{(t-q)!} \ell\frac{(\ell-q)!}{(\ell-t)!}. 
\end{equation}

{\bf Step III:} {\it We now apply \eqref{equ:largest-power} to estimate \eqref{equ:pit-form}. }
Note that every integer $e>1$ is bounded by $e\leq q!$ 
where $q$ is the largest prime power dividing $e$, 
since at worst every integer $1<u<q$ is a prime power dividing $e$. 
Hence, each summand $\abs{O_i^j}/\abs{ O_i}-1$ in \eqref{equ:pit-form}  
is bounded by $q!-1$ where $q$ is the largest power dividing $\abs{O_i^j}/\abs{O_i}$. 
Combining this bound with \eqref{equ:largest-power} for each $q>2$,
we get the following estimate of \eqref{equ:pit-form}:
\begin{equation}\label{equ:first-q}
R_{\pi_t}(f_t^{-1}(P)) \leq \abs{\{(i,j)\suchthat \frac{\abs{O_i^j}}{\abs{O_i}}=2\}} + \sum_{q}\biggl( \frac{1}{q}\frac{t!}{(t-q)!} \ell\frac{(\ell-q)!}{(\ell-t)!}(q!-1)\biggr)
\end{equation}
where $q$ runs over all prime powers between $3$ and $t$. 
Since $3\le q\leq t$ and $\ell>t^2$, 
a routine calculation shows that the summands in \eqref{equ:first-q} are strictly decreasing as functions
of $q$.
%
%
Hence each of the summands corresponding to prime powers between $3$ and $t$ 
is bounded by the summand at $q=3$ which is at most $E_1t^3\frac{(\ell-2)!}{(\ell-t)!}$, for some constant $E_1>0$. 
Since there are less than $t$ prime powers between $3$ and $t$, \eqref{equ:first-q} gives
\begin{equation}\label{equ:q>2}
R_{\pi_t}(f_t^{-1}(P))  < \abs{\{(i,j) \suchthat \frac{\abs{O_i^j}}{\abs{O_i}}=2\}} + E_1 t^4\frac{(\ell-2)!}{(\ell-t)!}. 
\end{equation}

{\bf Step IV:} {\it We use a similar argument in order to bound the number of pairs $(i,j)$ with $\abs{O_i^j}/\abs{O_i}= 2$. }
Note that for a tuple $U=[u_1,\ldots,u_t]\in S^{(t)}$ in an orbit $O_i^j$ with $\abs{O_i^j}/\abs{O_i}=2$, 
the element $x^{r_U/2}$ is a permutation of order $2$ on  $\oline U$. 
We divide such tuples $U$ into two types according to whether  $x^{r_U/2}$ acts on $\oline U$ as a transposition. If $x^{r_U/q}$ does not act as a transposition then $U$ satisfies:
\begin{enumerate}\item[$(\oline{\text{\emph{Trans}}})_{2,U}$] 
There exist $u_1,u_2,v_1,v_2\in \oline U$, such that $u_1,u_2$ and $v_1,v_2$ are pairwise disjoint length $2$ cycles of $x^{r_U/2}$. 
\end{enumerate}


We next bound the number of tuples $U$ satisfying $(\oline{\text{\emph{Trans}}})_{2,U}$. 
As in Step II, if $u_1,u_2$ is a cycle of $x^{r_U/2}$, then it is also a cycle of $x^{r_1/2}$, where $r_1$ is the length of the orbit of $u_1$ under $x$. Hence, a choice of $u_1$ (resp., $v_1$)  determines $u_2$ (resp., $v_2$), as its image under $x^{r_1/2}$. It follows that there are $\ell$ choices for $u_1$, which then determine $u_2$; $\ell-2$ choices for $v_1$ which then determine $v_2$; $t!/(t-4)!$ choices for the positions for $u_1,u_2,v_1,v_2$ in $U$; and $\ell!/(\ell-4)!$ choice for the rest of the entries in $U$. 
Hence  in total the number of tuples $U\in S^{(t)}$ satisfying $(\oline{{\text{\emph{Trans}}}})_{2,U}$ is at most: 
\begin{equation}\label{equ:v2ab}
 \frac{t!}{(t-4)!}\ell(\ell-2)\frac{(\ell-4)!}{(\ell-t)!}  < E_2t^4 \frac{(\ell-2)!}{(\ell-t)!},\text{ for some constant $E_2>0$}.
\end{equation} 

We next count the number of tuples $U\in S^{(t)}$ satisfying: 
\begin{enumerate}\item[$({\text{\emph{Trans}}})_{2,U}$] 
$x^{r_U/2}$ acts on $\oline U$ as a transposition $(u_1,u_t)$, for some $u_1,u_t\in \oline U$. 
\end{enumerate}
As the number of tuples $U$ satisfying $({\text{\emph{Trans}}})_{2,U}$ with given positions of $u_1,u_t$ in $U$, is independent of the choice of these positions, it suffices to count the number of tuple $U=[u_1,\ldots,u_t]\in S^{(t)}$ such that 
$x^{r_U/2}$ acts on $\oline U$ as the transposition $(u_1,u_t)$. 

Letting $\theta_i$ be an orbit of $x$ and $r_i=\abs{\theta_i}$ for $i=1,\ldots,t$,  a tuple $U=[u_1,\ldots, u_t]\in S^{(t)}$ with $u_i\in \theta_i$, $i=1,\ldots,t$, has an orbit of length $r_U=\lcm (r_1,\ldots,r_t)$. For such a tuple $U$, the element  $x^{r_U/2}$ acts on $\oline U$ as the transposition $(u_1,u_t)$ if and only if 
 $\theta_1=\theta_t$, and $v_2(r_1) = v_2(r_U)>0$, and $v_2(r_i)<v_2(r_U)$ for all $1<i<t$.  

Fix orbits $\theta_1,\ldots, \theta_{t-1}$ satisfying the latter constraints, that is, $v_2(r_1)=v_2(r)>0$ and $v_2(r_i)<v_2(r_1)$, $i=2,\ldots,t-1$, where $r_i:=\abs{\theta_i}$ and $r:=\lcm(r_1,\ldots,r_{t-1})$. We count the number of tuples $U=[u_1,\ldots,u_t]\in S^{(t)}$ with $u_i\in \theta_i$, $i=1,\ldots,t-1$, and $u_t=u_1^{x^{r/2}}$. 
Since there are $\hat r_1:=r_1$ choices for $u_1\in \theta_1$,
and $\hat r_{i}:=r_i-\abs{\{j\suchthat j<i, \theta_{j}=\theta_i\}}$ choices for $u_i\in \theta_i$, for each $i=2,\ldots,t-1$, 
 the number of such tuples $U$ is $\hat r_1\hat r_2\cdots \hat r_{t-1}$. 

In total we get that the number of tuples $U\in S^{(t)}$  satisfying $({\text{\emph{Trans}}})_{2,U}$ is 
\begin{equation}\label{equ:transp}
\frac{t(t-1)}{2}\sum_{\theta_1,\ldots,\theta_{t-1}} \frac{\hat r_1\cdots \hat r_{t-1}}{\lcm(r_1,\ldots,r_{t-1})},
\end{equation}
where $r_i=\abs{\theta_i}$, and $\hat r_i=r_i-\abs{\{j\suchthat j<i, \theta_{j}=\theta_i\}}$, and $\theta_1,\ldots \theta_{t-1}$ run through orbits of $x$ with even $r_1$ and  $v_2(r_1)>v_2(r_i)$, for $i=2,\ldots,t-1$.
Plugging the bound \eqref{equ:v2ab} on tuples $U\in S^{(t)}$ satisfying $(\oline{{\text{\emph{Trans}}}})_{2,U}$ 
and the count  \eqref{equ:transp} of orbits of tuples $U\in S^{(t)}$ satisfying $({\text{\emph{Trans}}})_{2,U}$ into \eqref{equ:q>2} proves part (2) with $E_0:=E_1+E_2$. 

{\bf Step V:} {\it The case $t=2$}. For $t=2$, there are no $2$-tuples $U\in S^{(2)}$ satisfying $(\text{\emph{Cyc}})_{q,U}$ with $q>2$, nor tuples satisfying $(\oline{\text{\emph{Trans}}})_{2,U}$. Moreover, the number of orbits $O_i^j$ consisting of tuples $U\in S^{(2)}$ that satisfy $(\text{\emph{Trans}})_{2,U}$ equals the number of even length orbits of $x$ by 
\eqref{equ:transp}. 
In total, evaluating \eqref{equ:pit-form} when $t=2$, then gives
\begin{align*}
R_{\pi_t}(f_t^{-1}(P))  &=  \abs{\{(i,j)\suchthat \frac{\abs{O_i^j}}{\abs{O_i}}=2\}} = \abs{\{(i,j)\suchthat ({\text{\emph{Trans}}})_{2,U} \text{ holds for every }U\in O_i^j\}} \\
&=  \abs{\{\text{even length orbits of }x\}}. \qedhere
\end{align*}
\end{proof}
\begin{rem}\label{rem:Y2-table}
By Proposition \ref{lem:pi-form}, if the ramification type of $h$ appears in Table \ref{table:two-set-stabilizer} and is different from $[\ell], [a,\ell-a], [2,1^{\ell-2}]$, then the sum of the contributions $R_{\pi_2}(f_2^{-1}(P))$ over all points $P$ of $Y$ is at least $(2\ell-5)/3$, as this is the minimal number of even entries among the ramification types in Table \ref{table:two-set-stabilizer} with $\ell\geq 13$.
\end{rem}

The proof of Theorem \ref{thm:main} uses the following proposition to cancel out the main term of $R_{\pi_t}(f_t^{-1}(P))$ in Proposition \ref{lem:pi-form} with the Riemann--Hurwitz contribution of the natural projection $h_1^{t-1}:Y_{t-1}\ra Y_1$ for $t\geq 2$. 
\begin{prop}\label{lem:t-1-bound} Let $h:Y_1\ra Y$ be a degree $\ell$ covering 
with monodromy group ${G\in\{A_\ell,S_\ell\}}$ acting on $S$. 
Let $3\leq t\leq \ell/2$ be an integer,
 $P$ a point of $Y$, 
 and $x$ a branch cycle over $P$. Then
\[
\sum_{(\theta_1,\ldots, \theta_{t-1})\in O_h(P)}\frac{\hat r_1\cdots \hat r_{t-1}}{\lcm(r_1,\ldots, r_{t-1})} \leq R_{h_1^{t-1}}(h^{-1}(P)), 
\]
where  $r_i := \abs{\theta_i}$, and $\hat r_i := r_i - \abs{\{j\suchthat j<i, \theta_j =\theta_i\}}$, $i=1,\ldots,t-1$, and $O_h(P)$ consists of tuples $(\theta_1,\ldots, \theta_{t-1})\in \Orb_S(x)^{t-1}$ such that $v_2(r_1)>v_2(r_j)$ for $j=2,\ldots, t-1$.
\end{prop}
The proof relies on the following lemma.
\begin{lem}\label{cor:RH-up-left}
Let $h:Y_1\ra Y$ be a degree $\ell$ covering with monodromy group $G\in\{A_\ell,S_\ell\}$ acting on $S$. 
Let $t\geq 2$, and $P$ be a point of $Y$, and $x\in G$ a branch cycle over $P$. 
Then 
\[ R_{h_1^{t-1}}(h^{-1}(P)) =  \sum_{(\theta_1,\ldots,\theta_{t-1})\in \Orb_S(x)^{t-1}} \frac{\hat r_1\cdots \hat r_{t-1}}{\lcm(r_1,\ldots,r_{t-1})}\Bigl(\frac{\lcm(r_1,\ldots,r_{t-1})}{r_1}-1\Bigr),\]
where $r_i := \abs{\theta_i}$ and 
$\hat r_i := r_i - \abs{\{j\suchthat j<i, \theta_j = \theta_i\}}$ for $i=1,\ldots, t-1$. 
\end{lem}
\begin{proof}
As in Setup \ref{sec:setup},   the monodromy group of the natural projection $h_{t-1}:Y_{t-1}\ra Y$ is $G$ with its action on 
the set $S^{(t-1)}$ of $(t-1)$-tuples of distinct elements from $S$. 
Hence 
$h_{t-1}$ is of degree $\ell!/(\ell-(t-1))!$ and $h_1^{t-1}$ is of degree 
$(\ell-1)!/(\ell-(t-1))!$. 

By the chain rule \eqref{equ:chain}, one has $R_{h_1^{t-1}}(h^{-1}(P))  = R_{h_{t-1}}(P) - \frac{(\ell-1)!}{(\ell-(t-1))!}R_h(P)$. Since $R_{h_{t-1}}(P) = \frac{\ell!}{(\ell-(t-1))!} - \abs{\Orb_{S^{(t-1)}}(x)}$ and $R_{h}(P) = \ell - \abs{\Orb_S(x)}$, this gives
\begin{equation}\label{equ:only-orb}
R_{h_1^{t-1}}(h^{-1}(P))  = \frac{(\ell-1)!}{(\ell-(t-1))!}\abs{\Orb_S(x)} - \abs{\Orb_{S^{(t-1)}}(x)}. 
\end{equation}

Letting  $p:\Orb_{S^{(t-1)}}(x) \ra \Orb_S(x)$ denote the projection onto the first coordinate, we partition $\Orb_{S^{(t-1)}}(x)$ as the union of the fibers of $p$. 
Hence  \eqref{equ:only-orb} yields:  
\begin{equation}\label{equ:Rhk}
 R_{h_1^{t-1}}(h^{-1}(P)) = \sum_{\theta\in \Orb_S(x)} \Bigl( \frac{(\ell-1)!}{(\ell-(t-1))!} - \abs{p^{-1}(\theta)}\Bigr).
 \end{equation}
Given $\theta\in \Orb_S(x)$, we note that the total number of tuples $U=[u_1,\ldots,u_{t-1}]\in S^{(t-1)}$ contained in orbits in  $p^{-1}(\theta)$ is $\abs{\theta}\cdot \frac{(\ell-1)!}{(\ell-(t-1))!}$: for, there  are $\abs{\theta}$ choices for $u_1$, and $(\ell-1)!/(\ell-(t-1))!$ choices for the rest of the entries. 
It follows that $$\sum_{\hat\theta\in p^{-1}(\theta)}\abs{\hat\theta} = \abs{\theta}\cdot \frac{(\ell-1)!}{(\ell-(t-1))!}.$$ Hence \eqref{equ:Rhk} amounts to: 
\begin{equation}\label{equ:Rhk2}
 R_{h_1^{t-1}}(h^{-1}(P)) =  \sum_{\theta \in \Orb_S(x)}\sum_{\hat \theta\in p^{-1}(\theta)}\Bigl( \frac{\abs{\hat \theta}}{\abs{\theta}} - 1 \Bigr). 
\end{equation}
  
For orbits $\theta_1,\ldots,\theta_{t-1}$ denote $r_i:=\abs{\theta_i}$, and $\hat r_i := r_i - \abs{\{j\suchthat j<i, \theta_j = \theta_i\}}$ for $i=1,\ldots, t-1$. Let $U(\theta_1,\ldots,\theta_{t-1})$ be the set of orbits  of $x$ on $(\theta_1\times \cdots \times \theta_{t-1})\cap S^{(t-1)}$, so that  $p^{-1}(\theta)$ partitions into the sets $U(\theta_1,\ldots,\theta_{t-1})$ where $\theta_2,\ldots,\theta_{t-1}$ run through orbits of $x$, and $\theta_1 =\theta$. 
For fixed $\theta_1,\ldots, \theta_{t-1}$, the number of tuples in orbits in $U(\theta_1,\ldots,\theta_{t-1})$ is  $\hat r_1\cdots \hat r_{t-1}$, and the length of each orbit $\hat \theta\in U(\theta_1,\ldots,\theta_{t-1})$ is $\lcm(r_1,\ldots,r_{t-1})$, so that the number of orbits in $U(\theta_1,\ldots,\theta_{t-1})$ is $\hat r_1\cdots\hat r_{t-1}/\lcm(r_1,\ldots,r_{t-1})$. 
Hence \eqref{equ:Rhk2} yields: 
\begin{align*}
R_{h_1^{t-1}}(h^{-1}(P))  & = \sum_{\theta_1,\theta_2,\ldots,\theta_{t-1}\in \Orb_S(x)}\sum_{\hat \theta\in U(\theta_1,\ldots,\theta_{t-1})}\Bigl( \frac{\abs{\hat \theta}}{\abs{\theta_1}} -1 \Bigr) \\
&  = \sum_{\theta_1,\ldots,\theta_{t-1}\in \Orb_S(x)} \frac{\hat r_1\cdots \hat r_{t-1}}{\lcm(r_1,\ldots,r_{t-1})}\left(\frac{\lcm(r_1,\ldots,r_{t-1})}{r_1}-1 \right). \qedhere
\end{align*}
\end{proof}
\begin{proof}[Proof of Proposition \ref{lem:t-1-bound}]
For a tuple $(\theta_1,\ldots,\theta_{t-1})$, denote $r_i:=\abs{\theta_i}$ and let $\hat r_i$ denote $r_i-\abs{\{j\suchthat j<i, \theta_j=\theta_i\}}$, for $i=1,\ldots, t-1$. 
Let $\oline O_h(P)$ be the set of all $(\theta_1,\ldots,\theta_{t-1})\in O_h(P)$ such that 
$r_j\divides r_1$ for $j=2,\ldots,t-1$. 
As $t\geq 3$, we can define a map $\phi:\oline O_h(P)\ra \Orb_S(x)^{t-1}$ 
which swaps the first and $i$-th entry of  $(\theta_1,\ldots,\theta_{t-1})\in \oline O_h(P)$, 
where $2\leq i\leq t-1$ is the smallest index for which $r_i = \min_{2\leq j\leq t-1}(r_j)$. 
Note that since $v_2(r_1)>v_2(r_j)$ for all $j=2,\ldots,t-1$ and $(\theta_1,\ldots, \theta_{t-1})\in O_h(P)$, the image $\phi(\oline O_h(P))$ is disjoint from $O_h(P)$. 

Restricting the sum in Lemma \ref{cor:RH-up-left} to the set $ \phi(\oline O_h(P)) \cup O_h(P) \setminus \oline O_h(P)$ we get
\begin{equation}\label{equ:sum-restrict} R_{h_1^{t-1}}(h^{-1}(P)) \geq  \sum_{(\theta_1,\ldots,\theta_{t-1})\in\phi(\oline O_h(P)) \cup O_h(P) \setminus \oline O_h(P)}
\frac{\hat r_1\cdots\hat r_{t-1}}{\lcm(r_1,\ldots,r_t)}\left( \frac{\lcm(r_1,\ldots,r_{t-1})}{r_1} -1 \right).
\end{equation}
By the definitions of $\oline O_h(P)$ and $\phi$, we have $\lcm(r_1,\ldots,r_{t-1})/r_1\geq 2$ for every tuple $(\theta_1,\ldots,\theta_{t-1})$ in $\phi(\oline O_h(P))$ or in $O_h(P) \setminus \oline O_h(P)$. Hence \eqref{equ:sum-restrict} gives
\begin{equation}\label{equ:shorter-sum}
R_{h_1^{t-1}}(h^{-1}(P)) \geq  \sum_{(\theta_1,\ldots,\theta_{t-1})\in\phi(\oline O_h(P)) \cup O_h(P) \setminus \oline O_h(P)}
\frac{\hat r_1\cdots\hat r_{t-1}}{\lcm(r_1,\ldots,r_t)}. 
\end{equation}
Noting that the product $\hat r_1\cdots \hat r_{t-1}$ and $\lcm(r_1,\ldots,r_{t-1})$ are preserved by $\phi$, we deduce that the right hand side of \eqref{equ:shorter-sum} equals
$\sum_{(\theta_1,\ldots,\theta_{t-1})\in O_h(P)}\frac{\hat r_1\cdots\hat r_{t-1}}{\lcm(r_1,\ldots,r_{t-1})}$. 
\end{proof}

As a consequence of Lemma \ref{cor:RH-up-left}, the following example shows that when almost all entries are the same in the ramification types of a family of coverings $h:Y_1\ra \mP^1$ of genus $\leq 1$ and monodromy $A_\ell$ or $S_\ell$, then  $g_{Y_2}$ and hence  $g_{X_2}$ grow at most linearly with $\ell$: 
\begin{exam}\label{exam:almost-Gal}
Fix $\eps>0$, an integer $r\geq 3$, and consider  degree-$\ell$ coverings $h:Y_1\ra\mP^1$ of genus  $\leq 1$, monodromy group $A_\ell$ or $S_\ell$, and ramification type $[A_1,k_1^*], \ldots, [A_r,k_r^*]$,  
for lists $A_1,\ldots,A_r$ of positive integers whose sum is at most $\eps$, and positive integers $k_1, \ldots, k_r$. We claim that the genus of $Y_2$ and hence of $X_2$ is at most $r\eps \ell+r\eps^2/2$. Hence, when further fixing the multisets $A_i$ and integers $k_i$, $i=1,\ldots,r$, and letting the degree $\ell$ grow (by allowing the multiplicities of $k_i$  grow), the genus of coverings with such ramification type is bounded by a constant multiple of $\ell$. Many infinite families of such genus $1$ coverings with $\eps\leq 10$ are given by \cite[Theorem 1.1]{KLN}. 

To prove the claim, let $P_i$ denote the branch point of ramification $[A_i,k_i^*]$, $i=1,\ldots,r$. By Lemma \ref{cor:RH-up-left} for $t=2$, we have 
$$R_{h_1^2}(h^{-1}(P_i))=\sum_{r_1,r_2\in E_h(P_i)}(r_1,r_2)\bigl(\frac{r_2}{(r_1,r_2)}-1\bigr)=\sum_{r_1,r_2\in E_h(P_i)}\bigl(r_2-(r_1,r_2)\bigr),$$ 
for $i=1,\ldots,r$. Thus, this contribution is at most $$\sum_{r\in A_i}\Bigl((r-1)\bigl(\eps+\frac{\ell}{k_i}\bigr)+(k_i-1)\frac{\ell}{k_i}\Bigr)
< \eps^2+2\eps\ell$$ for all $i$. The total contribution is then smaller than $2r\eps\ell+r\eps^2$, and  hence $g_{Y_2}<r\eps\ell+r\eps^2/2$ by the Riemann--Hurwitz formula for $h_1^2$. 
\end{exam}


\section{Theorem \ref{thm:main} when $g_{Y_1}$ is bounded from below}\label{sec:large-genus}
The following lemma and proposition prove Theorem \ref{thm:main} for coverings $h:Y_1\ra Y$ for which $g_{Y_1}$ is large in comparison to $t$.  We use Setup \ref{sec:setup}, so that $X_t$ (resp., $Y_t$) is the quotient by the stabilizer of a $t$-set (resp., $t$-tuple of distinct elements). 
\begin{prop}\label{lem:t=3,4} 
There exist positive constants $c_2,d_2,\beta$ 
such that for every covering $h:Y_1\ra Y$ of degree $\ell$, with monodromy group $A_\ell$ or $S_\ell$, and genus $g_{Y_1}\geq \beta t^4$ (resp., $g_{Y_1}\geq 2$), 
one has
\[
g_{X_t}-g_{X_{t-1}} > (c_2\ell - d_2 t^8)\frac{\binom{\ell}{t}}{\binom{\ell}{2}}, \quad \text{for $t\geq 3$ (resp., $t=2$).}
\] 
Furthermore for $t\geq 3$, one can set $c_2=2$. 
\end{prop}

\begin{rem}\label{rem:main-RH} \label{lem:2set-genus-formula}
As in the setup of Proposition \ref{lem:t=3,4}, let $\pi_t$, $h_{t-1}$ and $h_{t-1}^t$ be the natural projections $Y_t\ra X_t$, $Y_{t-1}\ra Y$, and $Y_t\ra Y_{t-1}$, respectively. 
The proof relies on the following formula
\begin{equation}\label{equ:M-RH}
\begin{split}
2t!(g_{X_t}-g_{X_{t-1}}) & =  
 2(\ell-2t+1)(g_{Y_{t-1}}-1) 
  \\
 & \!\!\! + \sum_{P\in Y(\K)}\bigl( R_{h_{t-1}^t}({h_{t-1}}^{-1}(P)) + tR_{\pi_{t-1}}({f_{t-1}}^{-1}(P)) - R_{\pi_t}(f_t^{-1}(P)) \bigr) 
\end{split} 
\end{equation}
and its immediate consequence
\begin{equation}\label{equ:M-RH-inequ}
2t!(g_{X_t}-g_{X_{t-1}})  \geq 2(\ell-2t+1)(g_{Y_{t-1}}-1) - \sum_{P\in Y(\K)}R_{\pi_t}(f_t^{-1}(P)). 
\end{equation}
Indeed, \eqref{equ:M-RH} follows from the Riemann--Hurwitz 
formula for the natural projections $h_{t-1}^t$, $\pi_t$, and $\pi_{t-1}$:
\begin{align*}
2(g_{Y_t}-1) = & 2(\ell-t+1)(g_{Y_{t-1}}-1) + \sum_{P\in Y(\K)}R_{h_{t-1}^t}({h_{t-1}}^{-1}(P)) \\
2t(g_{Y_{t-1}}-1) = & 2t!(g_{X_{t-1}}-1) + t\sum_{P\in Y(\K)}R_{\pi_{t-1}}({f_{t-1}}^{-1}(P)),  \text{ and }\\
2(g_{Y_{t}}-1) = & 2t!(g_{X_{t}}-1) + \sum_{P\in Y(\K)}R_{\pi_{t}}(f_{t}^{-1}(P)),
\end{align*}
by subtracting the first and second equalities from the third. 

In particular for $t=2$ and $\ell\geq 5$, as $\pi_1$ is the identity map, $R_{\pi_2}(f_2^{-1}(P))$ is the number of even $r\in E_h(P)$ by Proposition \ref{lem:pi-form},  and as $R_{h_1^2}(h^{-1}(P)) = \sum_{r_1,r_2\in E_h(P)}\Bigl(r_1-(r_1,r_2)\Bigr)$ by Lemma \ref{cor:RH-up-left}, the equality \eqref{equ:M-RH} amounts to
\begin{equation}\label{equ:M-RH-t=2} 
4(g_{X_2}-g_{X_1}) = 2(\ell-3)(g_{Y_1}-1) + \sum_{P\in Y(\K)}\Bigl( -\abs{\{\text{even $r\in E_{h}(P)$}\}} + \sum_{r_1,r_2\in E_{h}(P)} \bigl(r_1 - (r_1,r_2)\bigr) \Bigr).
\end{equation}
\end{rem} 
The proof of Proposition \ref{lem:t=3,4} for $t=2$ relies on the following estimate of the latter sum: 
\begin{lem}\label{lem:2set-nonnegative}\label{lem:bound-rs}
Let $1\leq r_1\leq \ldots \leq r_u\leq \ell$ be integers whose sum is $\ell$. 
\begin{enumerate} 
\item If the quantity 
\[
\mu_r:=  -\abs{\{\text{even $r_i\suchthat 1\leq i\leq u$}\}} \,\,+ \!\!\!\sum_{1\leq i,j\leq u} \bigl(r_i - (r_i,r_j)\bigr)
\]
is negative then $r_1=\cdots = r_u=\ell/u$ and $\ell/u$ is even;
\item If $\sum_{i,j\leq u} \bigl(r_i-(r_i,r_j)\bigr)\leq \ell/2u$ then $r_1=\cdots = r_u=\ell/u$.
\end{enumerate} 
\end{lem}
\begin{proof} 
Let $S_i := \sum_{j=1}^u \bigl(r_i - (r_i,r_j)\bigr)$ for $i=1,\ldots,u$, and let $s$ be the greatest common divisor of  $r_1,\ldots , r_u$.
Since $r_u$ (resp.~$r_1$) is maximal (resp., minimal) among the $r_i$'s, if $r_u\neq s$, then $r_u>\ell/u>r_1$ and 
$S_u\geq r_u - (r_u,r_1)\geq r_u - r_u/2 = r_u/2 > \ell/2u,$  proving (2). 

If no $r_i$ equals $s$ then $S_i\geq 1$ for each $i$, so that $\mu_r$ is nonnegative.
 Now assume there is an $i$ for which $r_i$ divides all $r_j$'s (so that $r_i=s$).  If all the $r_i$'s
 equal one another then obviously $\mu_r$ is $0$ if $s$ is odd, and $-\ell/s$ if $s$ is even.
 If there are $k$ values $i$ for which $r_i = s$ (with $0<k<\ell/s$), then if $r_j\ne s$ we get
 $S_j \geq k(r_j-s)$,  so since   $\sum_{j: r_j\ne s} r_j = (\ell-ks)$  it follows that
 $\sum_j S_j \geq k(\ell-ks)-ks\cdot \abs{\{ j \suchthat r_j\ne s\}}$.
Since in addition $u-k=\abs{\{ j \suchthat r_j\ne s\}} \leq (\ell-ks)/2s$, we get 
\begin{align*}
\mu_r\geq \sum_j (S_j-1) &\geq k(\ell-ks) - ks(u-k) -\Bigl( k + (u-k)\Bigr)
\\
& \geq  k(\ell-ks) - k - (ks+1)\frac{\ell-ks}{2s}.
\end{align*}
 For fixed $s$ and $\ell$, this is a quadratic polynomial in $k$ with leading coefficient  $-s/2$,  so on
 any interval it is minimized only at an endpoint.  In particular, as  $1\leq k\leq\ell/s - 2$  the
 function is minimized only when $k$ is either $1$ or $\ell/s-2$, where its values are  $(\ell-s-1)/2 - \ell/(2s)$
 and  $\ell-2s-\ell/s+1$.  If we now fix $\ell$ and vary $s$, then these values are minimized only when $s$ is either as small or as big as possible; for $s=1$ the values are $-1$, for $s=\ell/3$ they are $\ell/3-2$; note that for $s=\ell/2$, we have $r_1=r_2=\ell/2$.
 Since the sum of the $S_j-1$'s is an integer, it follows that it is nonnegative unless $s=1$ and the
 sum equals $-1$, but in that case some $r_i$ is odd so the sum we want is again nonnegative.
\end{proof}
We separate the proof according to whether $t=2$ or $t>2$. For $t=2$ we use Lemma \ref{lem:2set-nonnegative} to bound $R_{\pi_2}$ while for $t>2$ we use the bound from Proposition \ref{lem:t-1-bound}.  In both cases the proof splits to cases according to the number of branch points. 
\begin{proof}[Proof of Proposition \ref{lem:t=3,4} for $t=2$]
For $t=2$ we assume $g_{X_1}\geq 2$. We will show $4(g_{X_2}-g_{X_1})\geq \ell/2-6$. Since the claim is trivial if $\ell/2<6$, we may assume $\ell\geq 12$.  By equality \eqref{equ:M-RH-t=2} of Remark \ref{lem:2set-genus-formula} and Lemma \ref{lem:2set-nonnegative}  we have
\begin{equation}\label{equ:gX-bound-br-pt} 4(g_{X_2}-g_{X_1}) \geq 2(\ell-3)(g_{Y_1}-1) - \sum_{B_s}\frac{\ell}{s_P} \geq 2(\ell-3)(g_{Y_1}-1) - \frac{B\ell}{2}
\end{equation}
where $B_s$ is the set of branch points $P$ such that $E_{h}(P)=[s_P^{n/s_P}]$ with $s_P$ even, and 
$B:=\abs{B_s}$.
Since Riemann--Hurwitz for  $h$ gives
\begin{equation}\label{equ:two-pt-lower}
2(g_{Y_1}-1) \geq 2\ell(g_Y-1) + \sum_{P\in B_s} \Bigl(\ell-\frac{\ell}{s_P}\Bigr) \geq -2\ell + \frac{B\ell}{2},
\end{equation}
we have
\begin{align*}
4(g_{X_2}-g_{X_1}) & \geq  2(\ell-3)(g_{Y_1}-1) - \frac{B\ell}{2} \geq  (\ell-3)\Bigl(-2\ell+\frac{B\ell}{2}\Bigr) - \frac{B\ell}{2} \\
& = \ell\Bigl( 6-2\ell+\frac{B(\ell-4)}{2}\Bigr). 
\end{align*}
If $B\geq 5$ then the right hand side is at least $\ell^2/2 -4\ell>\ell/4-6.$
If $B\leq 3$ then \eqref{equ:gX-bound-br-pt} yields
\[
4(g_{X_2}-g_{X_1})\geq 2(\ell-3)(g_{Y_1}-1) - 3\ell/2,
\]
and since $g_{Y_1}=g_{X_1}\geq 2$ we get  $4(g_{X_2}-g_{X_1}) \geq  \frac{\ell}{4}-6$.
The only remaining case is $B=4$.
In this case, 
the expression $2(\ell-3)(g_{Y_1}-1) - \sum_{P\in B_s}\ell/s_P$
in \eqref{equ:gX-bound-br-pt} is bigger than $\ell/4-6$ if $g_{Y_1}=g_{X_1}\geq 3$ or if $s_P\geq 4$ for some $P\in B_s$.
Thus, we may assume that $g_{Y_1}=g_{X_1}=2$ and there are four branch points of type $[2^{\ell/2}]$.
The only ramification types which satisfy these conditions and 
Riemann--Hurwitz for $h:X_1\ra Y$ are
\[ [2^{\ell/2}]^4, [1^{\ell-4},2^2]; \quad \text{or} \quad [2^{\ell/2}]^4, [1^{\ell-3},3]; \quad\text{or} \quad
[2^{\ell/2}]^4, [1^{\ell-2},2], [1^{\ell-2},2].\]
By 
equality \eqref{equ:M-RH-t=2},
in all three cases one has
$4(g_{X_2}-g_{X_1}) > \ell/4 - 6$ for $\ell\geq 12$.  
\end{proof}
\begin{proof}[Proof of Proposition \ref{lem:t=3,4} for $t\geq 3$]
Let $E_0$ be the constant from Proposition \ref{lem:pi-form}. 
Note that we will pick $c_2=2$ and let $d_2\geq \max\{4E_0,1\}$ be a constant which will be determined by the proof. Since it suffices to prove the proposition when $2\ell-d_2 t^8\geq 0$, we can assume $\ell\geq d_2t^8/2\geq \max\{t^8/2, 2E_0t^4+1\}$. 
By  \eqref{equ:M-RH-inequ}, we have
\begin{equation*}\label{equ:previous}
2t!(g_{X_t}-g_{X_{t-1}})  \geq 2(\ell-2t+1)(g_{Y_{t-1}}-1) - \sum_{P\in Y(\K)}R_{\pi_t}(f_t^{-1}(P)). 
\end{equation*}
 Combining this inequality with the Riemann--Hurwitz formula for the natural projection $h_1^{t-1}:Y_{t-1}\ra Y_1$ gives 
\begin{equation}\label{equ:XtXt-1Y1} 
\begin{split}
2t!(g_{X_t}-g_{X_{t-1}})  &\geq  2(\ell-2t+1)\frac{(\ell-1)!}{(\ell-t+1)!}(g_{Y_1}-1)  \\ &\quad + \sum_{P\in Y(\K)}\Bigl( (\ell-2t+1)R_{h_1^{t-1}}(h^{-1}(P)) - R_{\pi_t}(f_t^{-1}(P)) \Bigr).
\end{split}
\end{equation}
To show that the right hand side is positive and large we will cancel out the contribution $R_{\pi_t}(f_t^{-1}(P))$ using the former two terms. 
As $\ell\geq t^2$, we may apply 
 Proposition \ref{lem:pi-form}, and Proposition \ref{lem:t-1-bound}, to obtain $R_{\pi_t}(f_t^{-1}(P)) \leq \binom{t}{2}R_{h_1^{t-1}}(h^{-1}(P)) + E_0t^4\frac{(\ell-2)!}{(\ell-t)!}$, 
for every point $P$ of $Y$. 
Hence \eqref{equ:XtXt-1Y1} gives
\begin{equation}\label{equ:BE(t)}
\begin{split}
2t!(g_{X_t} - g_{X_{t-1}}) & \geq 2(1-\eps)\frac{(\ell-1)!}{(\ell-t)!}(g_{Y_1}-1) -  E_0 Bt^4\frac{(\ell-2)!}{(\ell-t)!} \\
& \quad+ \sum_{P\in Y(\K)}\left(\ell-2t+1-\binom{t}{2}\right)R_{h_1^{t-1}}(h^{-1}(P))
\end{split}
\end{equation}
where $B$ denotes the number of branch points of $h$ and $\eps:=t/(\ell-t+1)$. 
Since $\ell\geq t^8/2$, we have $\ell-2t+1>\binom{t}{2}$ and hence  \eqref{equ:BE(t)} gives 
\begin{equation}\label{equ:BE(t)2}
2t!(g_{X_t} - g_{X_{t-1}}) \geq 2(1-\eps)\frac{(\ell-1)!}{(\ell-t)!}(g_{Y_1}-1) - E_0 B t^4\frac{(\ell-2)!}{(\ell-t)!}.
\end{equation}
To cancel out the term $E_0 B t^4(\ell-2)!/(\ell-t)!$, we divide the argument into two cases according to how large $B$ is. 
First assume $B\leq 2\ell + C$, where $C:=2(E_0 t^4+1)/(1-\eps)$. 
Set $\beta := 2E_0+1$. 
Note that as $\ell>3t$, one has $\eps<1/2$ and hence $C/2<2E_0t^4+3<\beta t^4$.
Assuming $g_{Y_1}\geq \beta t^4$,  since $\beta t^4>C/2$, 
since $B\leq 2\ell + C$ and $C(1-\eps)=2(1+E_0t^4)$,  \eqref{equ:BE(t)2} gives
\begin{equation}\label{equ:gamma1}
\begin{split}
2t!(g_{X_t}-g_{X_{t-1}}) & \geq 2\frac{(\ell-1)!}{(\ell-t)!}(1-\eps)\frac{C}{2} - E_0 (2\ell+C) t^4\frac{(\ell-2)!}{(\ell-t)!} \\
& = 2\frac{(\ell-1)!}{(\ell-t)!}(E_0t^4+1) - E_0(2\ell+C)t^4\frac{(\ell-2)!}{(\ell-t)!} \\
& = \frac{(\ell-2)!}{(\ell-t)!}\Bigl(2\ell-2 - 2E_0 t^4 - E_0 C t^4\Bigr) > (2\ell-\gamma_1t^8)\frac{(\ell-2)!}{(\ell-t)!}
\end{split}
\end{equation}
for some absolute constant $\gamma_1>0$.

Henceforth assume  $B\geq 2\ell + C.$  
Since $2(g_{Y_1}-1) \geq -2\ell + B$ by the Riemann--Hurwitz formula for $h$, \eqref{equ:BE(t)2} gives
\begin{equation}\label{equ:temp2}
\begin{split}
2t!(g_{X_t}-g_{X_{t-1}})\geq & \frac{(\ell-2)!}{(\ell-t)!}\Bigl((1-\eps)(-2\ell+B)(\ell-1) - E_0 B t^4\Bigr) \\ 
 = & \frac{(\ell-2)!}{(\ell-t)!}\Bigl(B\bigl((1-\eps)(\ell-1) - E_0 t^4\bigr) -2\ell(1-\eps)(\ell-1)\Bigr). 
\end{split}
\end{equation}
Since $B\geq 2\ell + C$,  and $(1-\eps)(\ell-1) > (\ell-1)/2> E_0t^4$,  \eqref{equ:temp2} gives
\begin{align*}
2t!(g_{X_t}-g_{X_{t-1}})\geq & 
\frac{(\ell-2)!}{(\ell-t)!}\Bigl((2\ell+C)\bigl((1-\eps)(\ell-1) - E_0t^4\bigr) -2\ell(1-\eps)(\ell-1)\Bigr) \\
= & \frac{(\ell-2)!}{(\ell-t)!}\Bigl(-2E_0t^4\ell + C(1-\eps)(\ell-1)-CE_0 t^4 \Bigr) 
\end{align*}
Since  $C(1-\eps) = 2(1+E_0 t^4)$, we get 
\begin{equation}\label{equ:gamma2}
2t!(g_{X_t}-g_{X_{t-1}})\geq  \frac{(\ell-2)!}{(\ell-t)!}\Bigl(2\ell -2 -2E_0t^4-E_0Ct^4 \Bigr)>  (2\ell-\gamma_2t^8)\frac{(\ell-2)!}{(\ell-t)!}
\end{equation}
for some absolute constant $\gamma_2>0$. 
The result follows with $d_2:=\max\{\gamma_1,\gamma_2,1,4E_0\}$ from \eqref{equ:gamma1} and \eqref{equ:gamma2}. 
\end{proof}

\section{Almost Galois ramification}\label{sec:almost-Galois}
In this section we prove severe restrictions on the ramification of a degree-$\ell$
covering $h:Y_1\ra Y$ with monodromy group in $\{A_\ell,S_\ell\}$ for which the
induced curve $Y_2$ has genus at most a constant multiple of $\ell$.
The argument in this section is based on the method introduced in~\cite{DZ}.

We first show that the ramification indices over every point behave similarly to ramification in Galois extensions. More precisely, for every point $P$ of $Y$, either (1) all but a bounded number of bounded elements in $E_h(P)$ equal a common value, or (2) $E_h(P)$ contains only a small number of small values.
We use Setup \ref{sec:setup}.
\begin{prop} \label{lem:DZ}
There exists a constant $\lambda_{1}$ satisfying the following property. 
Let $\alpha>0$ be an integer,
and $h:Y_1\ra Y$ a covering of degree $\ell\geq \lambda_{1}\alpha^2$ with monodromy group $G\in\{A_\ell, S_\ell\}$, 
such that the 
quotient $Y_2$ of the Galois closure of $h$ by a two-point stabilizer of $G$ satisfies
$g_{Y_2}\le \alpha\ell$.
Then 
any point $P$ of $Y$ satisfies one of the following:
\begin{enumerate}
\item there is some $k$ with $1\le k\le 6$ such that the number of $h$-preimages of $P$ with ramification index $k$ is at least
$\ell/k - 2(\alpha+1)(k+1)$;
\item for each $k$ with $1\le k\le 6$, the number of $h$-preimages of $P$
with ramification index $k$ is at most $4(\alpha+1)$.
\end{enumerate}
\end{prop}
\begin{proof} We pick $\lambda_1$ such that  $\lambda_1\alpha^2\geq 156(\alpha+1)(4\alpha+5)$. 
Suppose that (2) does not hold, so that there is an integer $k\le 6$ for which $E_h(P)$
contains more than $4(\alpha+1)$ copies of $k$.  Let $k$ be the smallest such integer.
We will show that $E_h(P)$ contains at least $\ell/k-2(\alpha+1)(k+1)$ copies of $k$, so that (1) holds.
Let $h_1^2:Y_2\ra Y_1$ be the natural projection. 
Combining Lemma~\ref{cor:RH-up-left} 
with
the Riemann--Hurwitz formula for $h_1^2$ yields
\begin{equation} \label{rh0}
\begin{split}
\sum_{r_1,r_2\in E_h(P)} &\bigl(r_1-(r_1,r_2)\bigr) = R_{h_1^2}(h^{-1}(P))
\le\sum_{Q\in Y(\mathbb{K})} R_{h_1^2}(h^{-1}(Q)) \\
&= 2\bigl(g_{Y_2}-1-(\ell-1)(g_{Y_1}-1)\bigr)
\le 2\bigl(\alpha\ell-1+(\ell-1)\bigr) < 2(\alpha+1)\ell,
\end{split}
\end{equation}
where the second inequality holds because $g_{Y_2}\le \alpha\ell$ and $g_{Y_1}\geq 0$.
Let $m$ be the number of copies of $k$ in $E_h(P)$, and let $R$ be the sum of the
elements in $E_h(P)$ which are less than $k$.  By minimality of $k$, for each $k_0<k$
there are at most $4(\alpha+1)$ copies of $k_0$ in $E_h(P)$, so that
\[
R \le \sum_{k_0=1}^{k-1} 4(\alpha+1)k_0 = 2(\alpha+1)(k^2-k).
\]
Restricting the left side of (\ref{rh0})
to pairs $(r_1,r_2)$ with $r_1>r_2=k$ yields
\[
m\!\!\!\sum_{\substack{r\in E_h(P)\\ r> k}}\!\!\! (r-(r,k)) < 2(\alpha+1)\ell.
\]
For any integer $r>k$, the value $(r,k)$ is a proper divisor of $r$ and hence is
at most $r/2$, so that $r-(r,k)\ge r/2$ and thus
\[
m\frac{\ell-km-R}{2} = m\!\!\!\sum_{\substack{r\in E_h(P) \\ r>k}} \frac{r}{2} \le
m\!\!\!\sum_{\substack{r\in E_h(P)\\ r> k}}\!\!\! \bigl(r-(r,k)\bigr) < 2(\alpha+1)\ell.
\]
Putting $d:=2(\alpha+1)(k^2-k)$ and $s:=(\ell-d)/k$, we deduce that
\begin{equation}\label{equ:rhs}
km(s-m)=m(\ell-km-d)\le m(\ell-km-R)<4(\alpha+1)\ell.
\end{equation}
Let $u$ be the unique element of $\{m, s-m\}$ such that $u\leq s/2$.
The hypothesis $\ell\ge \lambda_1\alpha^2\geq 156(\alpha+1)(4\alpha+5)$ implies that $s>\max(\ell/(2k),16(\alpha+1))$,
 so that $s-u\ge s/2>\ell/(4k)>0$ and thus by
\eqref{equ:rhs} we have
\[
u<\frac{4(\alpha+1)\ell}{k(s-u)} < 16(\alpha+1) < s,
\]
whence
\begin{equation} \label{zogby}
ku < \frac{4(\alpha+1)\ell}{s-u} < \frac{4(\alpha+1)\ell}{s-16(\alpha+1)} = \frac{4(\alpha+1)\ell k}{\ell-d-16(\alpha+1)k}. 
\end{equation}
The hypothesis $\ell\ge \lambda_1\alpha^2\geq 156(\alpha+1)(4\alpha+5)$ implies that the right side of \eqref{zogby} is
at most $(4\alpha+5)k$, so that $u<4\alpha+5$.
%
%
The definition of $k$ implies that $m>4(\alpha+1)$, so since $u\in\{m,s-m\}$ it follows that
$u=s-m$. As in addition $u\leq 4(\alpha+1)$, one has
\[
m\ge s-4(\alpha+1) = \frac{\ell-d}k-4(\alpha+1)=\frac{\ell-2(\alpha+1)(k^2-k)}k-4(\alpha+1)
=\frac{\ell}k -2(\alpha+1)(k+1).
\]
This shows that (1) holds, which completes the proof.
\end{proof}
\begin{defn}
Let $\alpha,\ell$ be positive integers such that $\ell\ge \lambda_{1}\alpha^2$, 
and let $P$ be a point of $Y$.
If there is any integer $k$ with $1\le k\le 6$ for which $E_h(P)$ contains at least
$\ell/k-2(\alpha+1)(k+1)$ copies of $k$ then we define $m_h(P)$ to be the least such integer
$k$.  If there is no such integer $k$, and case (2) of Proposition~\ref{lem:DZ} holds, then we define $m_h(P)$ to be $\infty$. 

We define the {\it error} $\eps_h(P)$ to be $2(\alpha+1)(m_h(P)+1)m_h(P)$ if $m_h(P)<\infty$
and $4(\alpha+1)\sum_{k=1}^6 k = 84(\alpha+1)$ if $m_h(P)=\infty$; thus if $m_h(P)<\infty$ then
$\eps_h(P)$ is an upper bound on the sum of the elements of $E_h(P)$ different from $m_h(P)$,
while if $m_h(P)=\infty$ then $\eps_h(P)$ is an upper bound on the sum of the elements of
$E_h(P)$ which are at most $6$.  
\end{defn}

\begin{cor} \label{cor:DZ}
There exists a constant $\lambda_{2}>0$ satisfying the following property. 
Let  $\alpha>0$ be an integer, and
$h:Y_1\ra Y$  a covering of  degree $\ell\geq \lambda_{2}\alpha^3$ and monodromy group $A_\ell$ or $S_\ell$, such that $g_{Y_2}\leq \alpha\ell$.  Let $M_h$ denote the multiset  of $m_h(P)$'s which are
bigger than $1$ (with $P$ varying over all points of $Y$). 
Then $g_Y\leq 1$; and $M_h$ is the empty set if $g_Y=1$, and $M_h$ is one of the following if $g_Y=0$:
\begin{itemize}
\item $\infty,\infty$
\item $\infty,2,2$
\item $3,3,3$
\item $2,3,6$
\item $2,4,4$
\item $2,2,2,2$;
\end{itemize}
$g_{Y_1}\leq \alpha+1$; there are at most $2(\alpha+1)$ points $P$ for which $m_h(P)=1$; and the sum of the quantities $R_{h}(P)$ over all such points is at most $8(\alpha+1)^2$.  
There exists also an absolute constant $\nu>0$ such that the number of preimages in $Y_1$ of each point $P$ with $m_h(P)=\infty$ is at most $\nu\alpha^2$. 
\end{cor}
\begin{proof}
Let $h_1^2:Y_2\ra Y_1$ the natural projection. 
By the Riemann--Hurwitz formula for $h_1^2$ and Lemma \ref{cor:RH-up-left}, we have
\begin{equation}\label{equ:h12-bound}
2(\alpha\ell +\ell -2)  \ge 2(g_{Y_2}-\ell g_{Y_1}+\ell-4) = \sum_{P\in Y(\K)} R_{h_1^2}(h^{-1}(P)) = \sum_{P\in Y(\K)} \sum_{r_1,r_2\in E_h(P)} \bigl(r_1-(r_1,r_2)\bigr).
\end{equation}
For $\ell\geq \lambda_{1}\alpha^2$, Proposition~\ref{lem:DZ} implies that if $m_h(P)=1$ then
\begin{align*}
\vert R_{h_1^2}(h^{-1}(P)) - \ell R_h(P)\vert & = 
\vert \sum_{r_1,r_2\in E_h(P)}\bigl( r_1-(r_1,r_2)\bigr) - \ell\sum_{r\in E_h(P)}(r-1)\vert  \\
  & \leq \vert \sum_{r_1,r_2\in E_{h}(P)\suchthat r_1,r_2>1}(r_1-(r_1,r_2)) - (\eps_h(P)-1)\sum_{r\in E_h(P)}(r-1) \vert \\
  & \le \eps_h(P)(\eps_h(P)-1) < \delta\alpha^2
  \end{align*}
for some absolute constant $\delta >0$.
In particular, $R_{h_1^2}(h^{-1}(P)) \ge \ell-\delta\alpha^2$ if $P$ ramifies under $h$.
Hence for $\ell\geq (2\alpha+3)\delta\alpha^2$, it follows from \eqref{equ:h12-bound} that there are at most $2(\alpha+1)$ points $P$ of $Y$ which ramify under $h$ and have $m_h(P)=1$. 
By definition,  for every point $P$ with  $m_h(P)=1$, we have $R_{h}(P) \le 4(\alpha+1)$, so the sum of this quantity over all such points $P$ is at most $8(\alpha+1)^2$.

The Riemann--Hurwitz formula for $h_1^2$ implies that 
$$g_{Y_1}-1\leq (g_{Y_2}-1)/(\ell-1)<\alpha+1,$$ for $\ell>\alpha-2$.  
Thus, the Riemann--Hurwitz formula for $h$ gives 
$\ell(g_Y-1) < g_{Y_1}-1 < \alpha+1,$ forcing $g_Y\leq 1$ for $\ell\geq \alpha+1$. 
Hence in total, the Riemann--Hurwitz formula for $h$ gives
\begin{equation} \label{inf}
 \begin{split} 
 \Bigl\vert \sum_{P\suchthat  m_h(P)>1} \bigl(\ell - \abs{h^{-1}(P)}\bigr) - 2\ell(1-g_Y)\Bigr\vert & = \Bigl\vert 2(g_{Y_1}-1) - \sum_{P\suchthat m_h(P)=1}R_h(P)\Bigr\vert \\
 & < \max\{2(\alpha+1),8(\alpha+1)^2 +2\} = 8(\alpha+1)^2+2. 
 \end{split}
\end{equation}
If $m_h(P)<\infty$ then by definition $E_h(P)$ contains at least $\ell/m_h(P) - 2(\alpha+1)(m_h(P)+1)$ entries equal to $m_h(P)$, and hence $\abs{E_h(P)}$ is at  most $2(\alpha+1)(m_h(P)+1)(m_h(P)-1)$ and at least $\ell/m_h(P)-2(\alpha+1)(m_h(P)+1)+1$. Thus, it is $\ell/m_h(P)$ up to a bounded constant depending on $\alpha$:
$$\Bigl\vert \abs{E_h(P)}-\frac{\ell}{m_h(P)}\Bigr\vert < \max\bigl\{2(\alpha+1)(m_h(P)+1)+1,  2(\alpha+1)(m_h(P)+1)(m_h(P)-1)\bigr\}.$$ 
If $m_h(P)=\infty$, then this number of points is at least $\ell/7-\eps_h(P)$.
Thus if $g_Y=1$, there exists a constant $\lambda_{3}>0$ such that the condition \eqref{inf} for $\ell\geq \lambda_{3}\alpha^2$ implies that $m_h(P)=1$ for all points $P$ of $Y$, as desired. 
If $g_Y=0$, there exists a constant $\lambda_{4}$, such that the only ways to add up numbers of the form $\ell(1-1/m_h(P))$ (with $2\le m_h(P)\le 6$) and
numbers in the interval $[6\ell/7,\ell]$ for $\ell\geq \lambda_{4}\alpha^2$, and get a sum that differs from  $2\ell$ by a bounded constant depending only on $\alpha$,
are the familiar ones, namely $\ell+\ell$, $\ell+\ell/2+\ell/2$, $2\ell/3+2\ell/3+2\ell/3$, $\ell/2+2\ell/3+5\ell/6$,
$\ell/2+3\ell/4+3\ell/4$, and $\ell/2+\ell/2+\ell/2+\ell/2$. Thus by picking the constant $\lambda_2$ so that  $$\lambda_{2}\alpha^3\geq \max\{\lambda_1\alpha^2, (2\alpha+3)\delta\alpha^2, \alpha+1, \lambda_{3}\alpha^2, \lambda_{4}\alpha^2\},$$ we have shown that all but the last assertion of Corollary \ref{cor:DZ} hold for $\ell\geq \lambda_{2}\alpha^3$. 

Finally, suppose that the  values $m_h(P)$ bigger than $1$ are $\infty,2,2$ so that $g_Y=0$.  Then (\ref{inf}) implies that the
number of points of $Y_1$ lying over the point $P$ with $m_h(P)=\infty$ is at most $\nu_1\alpha^2$ for some absolute constant $\nu_1>0$. 
Likewise if the $m_h(P)$'s bigger than $1$ are $\infty,\infty$ then (\ref{inf}) implies that the number of preimages over both points $P$
with $m_h(P)>1$  is at most $\nu_2\alpha^2$ for some absolute constant $\nu_2>0$.
Thus, the last assertion of Corollary \ref{cor:DZ} holds with $\nu:=\max\{\nu_1,\nu_2\}$. \end{proof}
\begin{defn} We say that a degree $\ell$ covering $h:Y_1\ra Y$ with monodromy group $A_\ell$ or $S_\ell$ has  $2$-point genus bound $\alpha\ell$
if $\ell\geq \lambda_{2}\alpha^3$, 
and $g_{Y_2}\leq \alpha\ell$. 
Corollary \ref{cor:DZ} then gives the possibilities for the values $m_h(P)$ for points $P$ of $Y$. 
\end{defn}
The constants $\lambda_{2}$ and $\nu$ can be improved, at the cost of a lengthier proof and larger errors $\eps_h(P)$. 


%
%
\section{Castelnuovo's inequality}\label{sec:castel}
We shall use Castelnuovo's inequality to prove that a covering has a $2$-point genus bound, and hence ramification of almost Galois type.

Recall \cite[Theorem 3.11.3]{Stich} that given two coverings $q_i:Z_i\ra Z_0$, for $i=1,2$, Casetlnuovo's inequality bounds the genus of a {\it minimal covering $p:Z\ra Z_0$ which factors through $q_1$ and $q_2$}
\footnote{Note that the curve $Z$ can always be identified with an irreducible component of the (normalization) of the fiber product of $q_1$ and $q_2$.}, that is, 
a covering which factors as $p=p_i\circ q_i$ for coverings $p_i:Z\ra Z_i$, $i=1,2$,  such that $p_1$ and $p_2$ do not have a common factorization $p_i = p_i'\circ w$, $i=1,2,$ with $\deg w>1$.
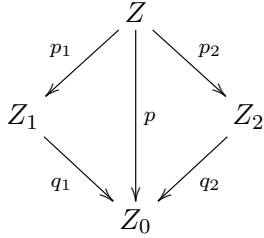
\[
\xymatrix{
& Z \ar[dl]_{p_1} \ar[dr]^{p_2} \ar[dd]^p & \\
Z_1 \ar[dr]_{q_1}& & Z_2\ar[dl]^{q_2} \\
& Z_0
}
\]
Namely, it gives $g_{Z}\leq n_1g_{Z_1} + n_2 g_{Z_2} + (n_1-1)(n_2-1)$, where $n_i:=\deg p_i$, $i=1,2$. 

We use the notation of Setup \ref{sec:setup}, so that the quotient by a stabilizer of a $t$-set (resp., $t$-tuple of distinct elements) is denoted by $X_t$ (resp., $Y_t$). 
\begin{prop}\label{prop:Castel} Let  $h:Y_1\ra Y$ be a covering of degree $\ell$ with monodromy group $G=A_\ell$ or $S_\ell$. Let $2\leq t\leq \ell/2$ be an integer, and $\alpha>0$ a constant such that 
$g_{X_t}-g_{X_{t-1}}<\frac{\alpha}{\ell}\binom{\ell}{t}$. Then 
\[
g_{Y_2}< \begin{cases}
g_{Y_1}(\ell+1) + (\alpha+1)(\ell -1)&\text{ if $t=2$ } \\
 \frac{1}{1-\eps}\Bigl((t-1)g_{Y_1} + \binom{t}{2} + \alpha \Bigr)\ell & \text{ if $t>2$,}
\end{cases}
\]
where $\eps:=t/(\ell-t+1)$.
\end{prop}
\begin{proof}
As in Setup \ref{sec:setup}, let $G$ act on  a set $S=\{1,\ldots,\ell\}$. 
Let $X_t^{(i)}$ denote the quotient of the Galois closure $\tilde Y_1$ by the subgroup $H_{i,t}\leq G$ of elements that fix $\{1,\ldots, i\}$ pointwise, and stabilize $\{1,\ldots,t\}$. In particular, 
$X_t^{(t-1)}$ is a $t$-point stabilizer and $X_t^{(0)}$ is a $t$-set stabilizer. 
Set  $Y_t:=X_t^{(t-1)}$ and $X_t := X_t^{(0)}$. 
Let $Y_{\{i\}}$ denote the quotient by the stabilizer $H_{\{i\}}$ of the element $i$, for $i\leq t$. 

Since $G$ is $t$-transitive, $[H_{i-1,t}:H_{i,t}]=t-i+1$ and $[H_{\{i\}}:H_{i,t}] = \frac{(\ell-1)!}{(\ell-t)!(t-i)!}$ for $i\leq t$. 
Since $H_{i-1,t}\cap H_{\{i\}} = H_{i,t}$ for $i=1,\ldots, t-1$, Castelnuovo's inequality gives
\begin{equation}\label{equ:castel} 
\begin{split} g_{X_t^{(i)}} & \leq \frac{(\ell-1)!}{(\ell-t)!(t-i)!} g_{Y_{\{i\}}} + (t-i+1)g_{X_t^{(i-1)}} + \Bigl(\frac{(\ell-1)!}{(\ell-t)!(t-i)!}-1\Bigr)(t-i) \\
& < \frac{(\ell-1)!}{(\ell-t)!(t-i)!} g_{Y_{\{i\}}} + (t-i+1)g_{X_t^{(i-1)}} + \frac{(\ell-1)!}{(\ell-t)!(t-i-1)!}.
\end{split} \end{equation}
For $t=2$ and $i=1$, as $X_1=Y_1$ and $g_{X_2}-g_{X_1}\leq \alpha(\ell-1)/2$, the first inequality in \eqref{equ:castel} gives
\begin{align*}
g_{Y_2} & \leq (\ell-1)g_{Y_1} + 2g_{X_2} + \ell-2 = (\ell+1)g_{Y_1} + 2(g_{X_2}-g_{X_1}) + \ell-2 \\
&  < (\ell+1)g_{Y_1} + (\alpha+1)(\ell-1) -1,
\end{align*}
as desired.

Henceforth we assume $t\geq 3$. Applying \eqref{equ:castel} iteratively for $i=t-1,\ldots , 1$, 
using the fact that $g_{Y_{\{i\}}}=g_{Y_1}$ since $G$ is transitive, we get
\begin{equation}\label{equ:sequence}
\begin{split} g_{Y_t}=g_{X_t^{(t-1)}} & < \frac{(\ell-1)!}{(\ell-t)!} g_{Y_1} + 2g_{X_t^{(t-2)}} + 
\frac{(\ell-1)!}{(\ell-t)!} \\
& < 2\frac{(\ell-1)!}{(\ell-t)!}g_{Y_1} + 3!g_{X_t^{(t-3)}} +(1+2)\frac{(\ell-1)!}{(\ell-t)!} \\
& \vdots \\
&  < (t-1)\frac{(\ell-1)!}{(\ell-t)!}g_{Y_1}  + t!g_{X_t^{(0)}} + (1+\cdots + t-1)\frac{(\ell-1)!}{(\ell-t)!}  \\
& = \frac{(\ell-1)!}{(\ell-t)!}\biggl((t-1)g_{Y_1} + \binom{t}{2}\biggr)  + t!g_{X_t}. 
\end{split}
\end{equation}
Since $t!(g_{X_t}-g_{X_{t-1}})<\alpha\frac{(\ell-1)!}{(\ell-t)!}$, the inequality \eqref{equ:sequence} yields
\begin{equation}\label{equ:iterate}
g_{Y_t} < \frac{(\ell-1)!}{(\ell-t)!}\biggl( (t-1)g_{Y_1} + \binom{t}{2} + \alpha \biggr)  + t!g_{X_{t-1}}.
\end{equation}  
By the Riemann--Hurwitz formula for the natural projections $\pi_{t-1}:Y_{t-1} \ra X_{t-1}$ and $Y_t\ra Y_{t-1}$,  one has $(t-1)!\cdot (g_{X_{t-1}}-1)\leq g_{Y_{t-1}}-1$ and  $ (\ell-t+1)(g_{Y_{t-1}}-1)\leq g_{Y_t}-1$.
Hence \eqref{equ:iterate} gives
\[ (\ell-t+1)(g_{Y_{t-1}}-1)\leq g_{Y_t}-1 < \frac{(\ell-1)!}{(\ell-t)!}\biggl( (t-1)g_{Y_1}+\binom{t}{2}+\alpha\biggr) + t(g_{Y_{t-1}}-1) + t!-1, \]
or equivalently
\begin{equation}\label{equ:almost-final}
(\ell-2t+1)(g_{Y_{t-1}}-1)< \frac{(\ell-1)!}{(\ell-t)!}\biggl( (t-1)g_{Y_1}+\binom{t}{2} +\alpha \biggr) +t!-1.
\end{equation}
Since $g_{Y_{t-1}}-1\geq \frac{(\ell-2)!}{(\ell-t+1)!}(g_{Y_2}-1)$ by the Riemann--Hurwitz formula for the natural projection $Y_{t-1}\ra Y_2$, \eqref{equ:almost-final} gives
\begin{equation}\label{equ:Castel3}
 g_{Y_2} \leq \frac{(\ell-t+1)!}{(\ell-2)!}(g_{Y_{t-1}}-1)+1< \frac{1}{1-\eps}\biggl( (t-1)g_{Y_1} + \binom{t}{2} + \alpha \biggr) (\ell-1)+ v,
\end{equation}
where $v := \frac{(t!-1)(\ell-t)!}{(1-\eps)(\ell-2)!}$. 
For $t\geq 3$ and $\ell\geq 2t$, a straight forward check shows that the right hand side of \eqref{equ:Castel3} is bounded by 
$\bigl( (t-1)g_{Y_1} + \binom{t}{2} + \alpha \bigr) \ell/(1-\eps),$ as desired. 
\end{proof}


\section{Ramification data of indecomposable coverings}
\label{sec:ram}
The final step of the proof of Theorem \ref{thm:main} is determining whether there exists a covering with a given ramification data, also known as the Hurwitz problem. 
In this section, we provide conditions under which a given ramification data does not correspond to an indecomposable covering. The following lemma is an explicit version of the ``translation" process in Guralnick--Shareshian \cite[Proposition 2.0.16 and Corollary 2.0.17]{GS}. 
\begin{lem}\label{lem:hurwitz1}
Let $p$ be a rational prime and $h:\mP^1\ra \mP^1$ a covering whose monodromy 
group is noncyclic, nondihedral, and is not $A_4$. Let $P_1, P_2, P_3$ be points of $\mP^1$. Then $h$ is decomposable if at least one of the following conditions holds:
\begin{enumerate}
\item All entries in $E_h(P_1)$ and $E_h(P_2)$ are divisible by $p$;
\item All entries in $E_h(P_1)$ are divisible by $p$, and there is a total of exactly two odd entries in $E_h(P_2)$ and $E_h(P_3)$;
\item All entries of $E_h(P_1)$ are even and there is a total of exactly two entries coprime to $3$ in $E_h(P_2)$ and $E_h(P_3)$. 
\end{enumerate}
\end{lem}
The proof uses the following version of Abhyankar's lemma. As in Section \ref{sec:castel}, we consider a minimal covering that factors through two given coverings.
\begin{lem}\label{lem:abh}
Let $q_i:Z_i\ra Z_0$, for $i=1,2$ be two coverings, and $p:Z\ra Z_0$ a minimal covering that factors through $q_1$ and $q_2$, say as $p=p_i\circ q_i$, for $i=1,2$. 
Let $Q_i$ be a point of $Z_i$, $i=1,2$ such that $q_1(Q_1)=q_2(Q_2)$. Then
\begin{enumerate} 
\item for every point $Q$ of $Z$ such that $p_1(Q) = Q_1$ and $p_2(Q)=Q_2$, one has
\[
e_p(Q) = \lcm(e_{q_1}(Q_1),e_{q_2}(Q_2))\text{ and }
e_{p_2}(Q)=
\frac{e_{q_1}(Q_1)}{\bigl(e_{q_1}(Q_1),e_{q_2}(Q_2)\bigr)}.
\]
\item If furthermore $\deg p = \deg q_1\cdot \deg q_2$ then the number of points $Q$ such that $p_1(Q)=Q_1$ and $p_2(Q)=Q_2$  is $\bigl(e_{q_1}(Q_1),e_{q_2}(Q_2)\bigr)$. 
\end{enumerate} 
\end{lem}
\begin{proof}
Set $P:=q_1(Q_1)=q_2(Q_2)$. 
Let $\tilde p:\tilde Z\ra Z_0$ be a Galois closure of $p$ and $G$ its monodromy group.
Writing $\tilde p = q_1\circ u_1 = q_2\circ u_2 = p\circ v$, we let $H_i\leq G$ be the monodromy group of $u_i$, $i=1,2$. Since $p$ is minimal, the monodromy group of $v$ is $H_1\cap H_2$. 
Let $x\in G$ be a branch cycle of $\tilde p$ over $P$. 
By Section \ref{sec:prelim}, there is a one to one correspondence between points $Q$ of $Z$ and orbits $\theta$ in $\Orb_{(H_1\cap H_2)\backslash G}(x)$ such that $e_{p}(Q)=\abs{\theta}$. Moreover as in Section \ref{sec:prelim}, the image $p_i(Q)$  corresponds to the restriction of $\theta$ to an orbit in  $\Orb_{H_i\backslash G}(x)$ for $i=1,2$. 
Let $\theta_i\in \Orb_{H_i\backslash G}(x)$ be the orbit corresponding to $Q_i$ and $r_i:=\abs{\theta_i} = e_{q_i}(Q_i)$  for $i=1,2$. 

Note that the natural map $\iota:(H_1\cap H_2)\backslash G\ra (H_1\backslash G)\times (H_2\backslash G)$ of $G$-sets is an injection. Identifying $(H_1\cap H_2)\backslash G$ with its image, the orbit of any pair in $\theta_1\times \theta_2$ has length $\lcm(r_1,r_2)$. Thus by the above correspondence, $e_{p}(Q) = \abs{\theta} = \lcm(r_1,r_2)$ for every point $Q$ of $Z$ with $p_i(Q)=Q_i$, $i=1,2$, and hence 
\[
e_{p_2}(Q) = \frac{e_{p}(Q)}{e_{q_2}(Q_2)} = \frac{\lcm(r_1,r_2)}{r_2} = \frac{r_1}{(r_1,r_2)},
\]
giving (1). 
If $\deg p = \deg q_1\cdot \deg q_2$, then $\iota$ is surjective and hence the number of orbits is the number $r_1r_2$ of pairs in $\theta_1\times \theta_2$ divided by the length $\lcm(r_1,r_2)$ of each orbit, and hence is $(r_1,r_2)$, giving (2). 
\end{proof}
\begin{rem}\label{rem:abh} Let $h:Y_1\ra\mP^1$ be a covering with Galois closure $\tilde h:\tilde Y_1\ra\mP^1$. Then Lemma \ref{lem:abh}.(1) applied iteratively shows that $e_{\tilde h}(P)$ is the least common multiple of all entries in $E_h(P)$ for every point $P$ in $\mP^1$. Indeed, each such entry and hence their least common multiple divide $e_{\tilde h}(P)$. On the other hand, letting $G=\Mon(h)$ and $H_1$ a  point stabilizer,  $\tilde h$ is the minimal Galois covering which factors through the natural projections $h^{\sigma}:\tilde Y_1/H_1^\sigma \ra \mP^1$, $\sigma\in G$. Thus, Lemma \ref{lem:abh}.(1) implies that $e_{\tilde h}(P)$ is the least common multiple of  $e_{h^{\sigma}}(Q_\sigma)$, $\sigma\in G$ for some points $Q_\sigma\in \tilde Y_1/H_1^\sigma$.
As $h^{\sigma}$, $\sigma\in G$ are isomorphic coverings, $E_{h^{\sigma}}(P) = E_{h}(P)$ for $\sigma\in G$, and hence $e_{\tilde h}(P)$ is also the least common multiple of entries from $E_h(P)$. 
\end{rem}
To ensure the assumption of Lemma \ref{lem:abh}.(2) is satisfied, we shall use: 
\begin{lem}\label{lem:no-common}
Let $q_i: Z_i\ra Z_0$, $i=1,2$ be coverings with no common factorization $q_i = u\circ q_i'$, $i=1,2$ with $\deg u>1$. 
If $q_1$ is Galois then the degree of a minimal covering $p$ that factors through $q_1$ and $q_2$ is $\deg q_1\cdot \deg q_2$.
\end{lem} 
\begin{proof}
Let $\tilde p:\tilde Z\ra Z_0$, $u_1,u_2,v, G, H_1,$ and $H_2$ be as in the proof of Lemma \ref{lem:abh}. 
 The minimality of $p$ shows that monodromy group of $v$ is $H_1\cap H_2$. Since $q_1$ and $q_2$ have no common factorization, we have $G=\langle H_1,H_2\rangle$. Since $q_1$ is Galois, $H_1$ is normal in $G$, and hence $G=H_1H_2$. It follows that the natural map $(H_1\cap H_2)\backslash G\ra (H_1\backslash G)\times (H_2\backslash G)$  is an isomorphism, and hence $\deg p = [G:H_1\cap H_2] = [G:H_1]\cdot [G:H_2] = \deg q_1\cdot \deg q_2$. 
\end{proof}
\begin{proof}[Proof of Lemma \ref{lem:hurwitz1}]
Consider a Galois covering $\phi:\mP^1\ra \mP^1$ such that in cases $(1), (2),$ and $(3)$, respectively, we have
\begin{enumerate}
\item $E_{\phi}(P_j)=[p]$ for $j=1,2$, and $\Mon(\phi)\cong C_p$; 
\item $E_{\phi}(P_1)=[p^2]$, $E_{\phi}(P_j)=[2^{p}]$ for $j=2,3$, and $\Mon(\phi)\cong D_{2p}$;
\item $E_{\phi}(P_1)=[2^6]$, $E_{\phi}(P_j)=[3^4]$ for $j=2,3$, and $\Mon(\phi)\cong A_4$.
\end{enumerate}
Note that in each of the cases (1), (2), and (3), there exists a covering $\phi'$ with the described ramification type and monodromy, namely, $X^p$,  $X^{p}+X^{-p}$, and $(X^4+8X)^3 / (X^3-1)^3$, respectively. The covering $\phi$ is then obtained by composing $\phi$ with a linear fractional which sends the branch points of $\phi'$ to $P_1,P_2,P_3$. 

Since $h$ is indecomposable, either  $\phi = h\circ \psi$ for some covering $\psi$, or $h$ and $\phi$ have no common factor $\phi = w\circ \phi'$, $h=w\circ h'$, with $\deg w>1$. 
The latter case contradicts the assumption that the monodromy group of $h$ is noncyclic, nondihedral, and is not $A_4$. 

Hencefore we assume $\phi$ and $h$ have no common factor $w$ as above. 
Let $p: Z\ra \mP^1$ be a minimal covering which factors through $\phi$ and $h$, 
say $p = \phi\circ \eta = h\circ \pi$. Then $\deg p = \deg \phi\cdot \deg h$ by Lemma \ref{lem:no-common}. 
Hence, we may apply Lemma \ref{lem:abh}.(2) to compute the ramification of $\pi$ in cases (1), (2), and (3), respectively, giving 
\begin{enumerate}
\item $R_{\pi}(h^{-1}(P_j)) = (p-1)\cdot \abs{\{r\in E_h(P_j): p\nmid r\}}$ \quad for $j=1,2$;
\item $R_{\pi}(h^{-1}(P_1)) = 2(p-1)\cdot \abs{\{r\in E_h(P_1): p\nmid r\}}$\quad and 
\[
R_{\pi}(h^{-1}(P_j)) = p \cdot\abs{\{r\in E_h(P_j): 2\nmid r\}}\quad \text{ for $j=2,3$;}
\]
\item $R_{\pi}(h^{-1}(P_1)) = 6 \cdot\abs{\{r\in E_h(P): 2\nmid r\}}$ \quad and 
\[
R_{\pi}(h^{-1}(P_j))=8\cdot \abs{\{r\in E_h(P_j): 3\nmid r\}} \quad \text{ for $j=2,3$}.
\]
\end{enumerate}
Since $g_{Z}\geq 0$, Riemann--Hurwitz for $\pi$ yields
\begin{equation}\label{equ:g123} \sum_{j=1}^3 R_{\pi}(h^{-1}(P_j))  = \sum_{P\in \mP^1(\K)}R_{\pi}(h^{-1}(P))  \geq 2\deg\pi -2 = 2\deg \phi -2.
\end{equation}
In case (1),  \eqref{equ:g123} shows that the total number of coprime to $p$ entries in $E_h(P_1)$ and $E_{h}(P_2)$ is at least two, contradicting (1).

In case (2), \eqref{equ:g123} gives
\begin{equation}\label{equ:cheb}
2(p-1) \cdot\abs{\{r\in E_h(P_1): p\nmid r\}}
+ p \cdot \abs{\{\text{odd }r\in E_h(P_j), j=2,3\}} 
\geq  4p-2. \end{equation}
By case (1), there exists at least one odd entry in $E_h(P_2)$ or in $E_h(P_3)$. Since $\sum_{r\in E_h(P_2)}r = \sum_{r\in E_h(P_3)}r$ mod $2$, there is an even number of such odd entries. Hence \eqref{equ:cheb} shows that either there is an entry in $E_h(P)$ that is coprime to $p$, or the total number of odd entries in $E_h(P_2)$ and $E_h(P_3)$ is at least four, contradicting  (2).

In case  (3), \eqref{equ:g123} gives
\begin{equation}\label{equ:A4}  6\cdot \abs{\{\text{odd }r\in E_h(P_1)\}} + 8\cdot \abs{\{r\in E_h(P_j)\suchthat j=2,3,\, r\text{ coprime to }3\}} \geq 22.
\end{equation}
By part (1) there exists at least one entry in $E_h(P_2)$ or $E_h(P_3)$ that is coprime to $3$. Since $\sum_{r\in E_h(P_2)}r \equiv \sum_{r\in E_h(P_3)}r \pmod 3$, the number of such coprime entries is at least $2$. Hence \eqref{equ:A4} shows either that there is an odd entry in $E_h(P_1)$, or the total number of coprime to $3$ entries in $E_h(P_2)$ and $E_h(P_3)$ is at least three, contradicting (3).
\end{proof}

The following proposition gives the list of ramification types in case (4) of Theorem \ref{thm:general}, and is then used to restrict ramification types of indecomposable polynomials. 
\begin{prop}\label{lem:normal-closure}\label{lem:isogeny-mon}
Let $h:Y_1\ra \mP^1$ be an indecomposable covering with Galois closure of genus $0$ or $1$.
Then the ramification type and monodromy group of $h$ appear in Table \ref{table:solvable}. Conversely, every  type in Table \ref{table:solvable} occurs as the ramification type of such covering.

\begin{table}
\caption{Ramification types of indecomposable coverings $h:Y_1\ra \mP^1$ of degree $\ell$ and Galois closure $\tilde h:\tilde Y_1\ra \mP^1$ of genus $g_{\tilde Y_1}\leq 1$. In each case, we assume $\ell$  satisfies the necessary congruence conditions to make all exponents integral. In all cases $p$ denotes a prime.}
\label{table:solvable}
\begin{equation*}
\begin{array}{| c | l | c | l |}
\hline
& \text{Ramification types for $h$ } & \text{$g_{\tilde Y_1}$} & \Mon(h) \\
\hline
(A1) & [\ell], [\ell] & 0 &  C_\ell, \text{$\ell$ prime}\\
(A2) & {[1,2^{(\ell-1)/2}]}, [1,2^{(\ell-1)/2}], [\ell] & 0 & D_{2\ell}, \text{$\ell$ prime} \\
(A3) & {[1^2,2^4]}, [1,3^3], [5^2] & 0 &  A_5 \\
(A4) & {[1^2,2^2]}, [3^2], [1,5] & 0 &  A_5 \\
(A5) & {[1,2^2]}, [1^2,3], [5]  & 0 & A_5 \\
(A6) & {[1^2,2]}, [1,3], [4]  & 0 & S_4 \\
(A7) & {[2^2]}, [1,3], [1,3]  & 0 & A_4 \\
\hline 
(E1) & [1,2^{(\ell-1)/2}] \text{ four times} 
& 1 & D_{2\ell}, \text{$\ell$ prime} \\
(E2) & [1,3^{(\ell-1)/3}], [1,3^{(\ell-1)/3}], [1,3^{(\ell-1)/3}] & 1 & C_\ell \rtimes C_3, \text{$\ell$ prime} \\
(E3) & [1,3^{(\ell-1)/3}], [1,3^{(\ell-1)/3}], [1,3^{(\ell-1)/3}] & 1 & (C_p)^2\rtimes C_3, \text{$\ell=p^2$, $p\equiv 2$ mod $3$} \\
(E4) & [1,2^{(\ell-1)/2}], [1,4^{(\ell-1)/4}], [1,4^{(\ell-1)/4}] & 1 & C_\ell\rtimes C_4, \text{$\ell$ prime} \\
(E5) & [1,2^{(\ell-1)/2}], [1,4^{(\ell-1)/4}], [1,4^{(\ell-1)/4}] & 1 & (C_p)^2\rtimes C_4, \text{$\ell=p^2$, $p\equiv 3$ mod $4$}  \\
(E6) & [1,2^{(\ell-1)/2}], [1,3^{(\ell-1)/3}], [1,6^{(\ell-1)/6}]  & 1 & C_\ell\rtimes C_6,  \text{$\ell$ prime} \\
(E7) & [1,2^{(\ell-1)/2}], [1,3^{(\ell-1)/3}], [1,6^{(\ell-1)/6}]  & 1 & (C_p)^2\rtimes C_6, \text{$\ell=p^2$, $p\equiv 5$ mod $6$}\\
\hline
(Q1) & [2], [2], [2], [2] & 1 & C_2 \\
(Q2) & [3], [3], [3] & 1 & C_3 \\
\hline
\end{array}
\end{equation*}\end{table}
\end{prop}
\begin{lem}\label{lem:elliptic-decomposition}
Let $Z$ be a curve of genus $1$ and let $O$ be a point of $Z$. Then the automorphism group $\Aut(Z)$ is the semidirect product $T\rtimes A^\times$, where $T$ is 
the group of translations by  points on the elliptic curve $(Z,O)$, and $A^\times$ is the group of automorphisms of $Z$ fixing $O$.
\end{lem}
\begin{proof} Considering the group structure $(Z,+,O)$ on $Z$, the claim follows since every automorphism $f:Z\ra Z$ can be written as a composition $\sigma_{f(O)} \circ g$, where $\sigma_{f(O)}\in T$ is the translation map $P\mapsto P+f(O)$ and $g := \sigma_{f(O)}^{-1}\circ f$ is an automorphism fixing $O$. 
\end{proof}

\begin{rem}\label{rem:extra-facts}\label{rem:lattes}
(1) The normal subgroup $T$ is independent of the choice of base point $O$, namely it is the subgroup consisting of all fixed point free automorphisms of $Z$.

(2) The endomorphism ring $A:=\End(Z)$ is isomorphic either to $\mathbb Z$ or to an order in an imaginary quadratic field, so that the multiplicative group $A^\times$  is  the group  $\mu_v$ of $v$-th roots of unity for some $v\in\{2,4,6\}$, cf.~\cite[Section III.9 and III.10]{Sil} and \cite[Example III.4.4]{Sil}. 
Moreover, up to isomorphism there is a unique elliptic curve $Z$ for which $A^\times\cong \mu_v$, for $v= 4,6$. 
The cardinality of the kernel of an endomorphism $\eta\in A$ equals the norm of its corresponding element in $\mathbb Z$ or in the imaginary quadratic field.

(3) For a finite subgroup $H\leq \Aut(Z)$, the quotient $Z/H$ is of genus $1$ if and only if $H\leq T$,  and is of genus $0$ otherwise. 

(4) Every covering $h:Y_1\ra Y$ of genus $1$ curves is Galois with abelian monodromy group \cite[Theorem 4.10(c)]{Sil}.
\end{rem}
    
\begin{proof}[Proof of Proposition \ref{lem:normal-closure}] 
Let $G=\Mon(h)$ and $H_1\leq G$ a point stabilizer, so that $H_1$ is maximal in $G$. 
Let $\tilde h:\tilde Y_1\ra \mP^1$ be the Galois closure of $h$, and $n:=\deg\tilde h$. 
At first assume  $g_{\tilde Y_1}=0$.  The ramification and mondromy groups of genus $0$ Galois coverings $\tilde h:\tilde Y_1\ra \mP^1$ is well-known as a consequence of Klein's classification of finite subgroups of $\Aut(\mP^1)\cong \PGL_2(\K)$.  Namely, it is one of $[n],[n]$ with $G\cong C_n$; $[n/2,n/2],[2^{n/2}],[2^{n/2}]$ with $G\cong D_{2n}$; $[2^6], [3^4], [3^4]$ with $G\cong A_4$; $[2^{12}], [3^{8}], [4^6]$ with $G\cong S_4$; and $[2^{30}], [3^{20}], [5^{12}]$ with $G\cong A_5$. 
As $H_1$ is maximal and does not contain a normal subgroup of $G$, 
we have $n=\ell$ is prime and $H_1=1$ if $G\cong C_\ell$; and $n=2\ell$ with $\ell$ an odd prime, and $H_1$ is generated by a reflection if $G\cong D_{2\ell}$; and $H_1\cong A_3$ if $G\cong A_4$; and $H_1\cong S_3$ if $G\cong S_4$; and $H_1\cong A_4$, or $D_{10}$, or $S_3$ if $G\cong A_5$. By Abhyankar's lemma \ref{rem:abh},   $e_{\tilde h}(P)$ is the least common multiple of entries in $E_h(P)$, for every point $P$ of $\mP^1$. For each possibility of $G$ and $H_1$, the latter constraint and the Riemann--Hurwitz formula force the ramification of $h$ to be one of types A in Table \ref{table:solvable}. Conversely, all ramifications types A1-A7 of Table \ref{table:solvable} occur as the ramification of the quotient map $\tilde Y_1/H_1\ra\tilde Y_1/G$ induced from $\tilde h$ as above.

Henceforth assume $g_{\tilde Y_1}=1$. Let $T\lhd \Aut(\tilde Y_1)$ be the abelian normal subgroup from Lemma~\ref{lem:elliptic-decomposition}, and let $N:=G\cap T$, so that 
$N\lhd G$.
First consider the case where $g_{Y_1}=1$, in which case Remark \ref{rem:extra-facts}.(3) implies that $H_1\subseteq N$ and that $N$ is a proper subgroup of $G$.  As $H_1$ is maximal, we have $H_1=N$. 
Hence the natural projection  $\phi:\tilde Y_1/N\ra \mP^1$  factors through $h$. 
Since $\phi$ is Galois as $N\lhd G$, and $\tilde h$ is the Galois closure of $h$, it follows that $N=1$. It follows that the natural map $G\ra \Aut(\tilde Y_1)/T$ is injective, and hence that $G\cong \mu_v$ for $v\in\{2,3,4,6\}$ by Remark \ref{rem:extra-facts}.(2). In particular $h$ is Galois. As $h$ is indecomposable, we get $G\cong C_2$ or $C_3$. By the Riemann--Hurwitz formula for $h$, its ramification is  (Q1) or (Q2) in Table \ref{table:solvable}. Conversely, (Q1),(Q2) are the ramification types of the projection $Y\ra\mP^1$ from an elliptic curve $Y$ to its $x$ and $y$ coordinates, resp.

Henceforth assume $g_{Y_1}=0$.  
Remark \ref{rem:extra-facts} then implies that $H_1$ is not contained in $N$. As $H_1$ is maximal, this gives $G=N\cdot H_1$. Let $\psi:\tilde Y_1/(H_1\cap N) \ra \mP^1$ be the natural projection. Since $H_1\cap N\lhd H_1$ as $N\lhd G$, and since $H_1\cap N\lhd N$ as $N$ is abelian, we have $H_1\cap N \lhd  N\cdot H_1=G$, so that $\psi$ is Galois. Since $\psi$ factors through $h$, and $\tilde h$ is the Galois closure of $h$, it follows that $H_1\cap N=1$, and hence $G=  N\cdot H_1 =  N\rtimes H_1$.

We next describe $H_1$ and $N$ using the decomposition in Lemma \ref{lem:elliptic-decomposition}. As $N$ consists of the fixed point free automorphisms in $G$ by Remark \ref{rem:extra-facts}.(1), it follows that a generator of $H_1$ and hence $H_1$ has a fixed point $O$ in $\tilde Y_1$, so that $(\tilde Y_1,O)$ is an elliptic curve, and $H_1$ is a subgroup of the group $A^\times$ of automorphisms fixing $O$. 
Let $Z_1:=\tilde Y_1/N$, let $\eta:\tilde Y_1\ra Z_1$ be the natural projection, and set $O'=\eta(O)$. Then  $(Z_1,O')$ is an elliptic curve by Remark~\ref{rem:extra-facts}.(3), and $\eta$ is an isogeny. 
\begin{equation}\label{equ:EC}
\xymatrix{
\tilde Y_1 \ar[r]^{\eta} \ar[d]_{\rho_1} & Z_1\ar[d]^{\rho_2} \\
Y_1 \ar[r]_h & \mP^1 
}
\end{equation}
By Remark \ref{rem:extra-facts}.(2), we can identify the endomorphism ring $A:=\End(\tilde Y_1)$  either with $\mathbb Z$ or  with an  order  in an imaginary quadratic field, so that the subgroup $H_1$ of  $A^\times$   identifies with a multiplicative subgroup of $v$-th roots of unity $ \mu_v$ for $v\in\{2,3,4,6\}$. Let $B$ be the subring $\mathbb Z[\mu_v]$ of $A$. 
Since $H_1$ is maximal, there are no nontrivial intermediate subgroups between $H_1$ and $ N\rtimes H_1$, and hence  no nontrivial proper subgroups of $N$ that are invariant under conjugation by $H_1$. In particular, $N$ is  $p$-torsion  for some prime $p$. Thus, the action of $H_1$ gives $N$ the structure of an irreducible $p$-torsion $B$-module. 

First consider the case where $p$ is coprime to $v=\abs{H_1}$.  
Let $[-1]\in A^\times$ be the map $P^{[-1]}=-P$, and $\sigma_Q\in T$ the translation  $P^\sigma_Q = P+Q$ for $P,Q\in \tilde Y_1$. 
If $v=2$, then $H_1= \mu_2=\langle [-1]\rangle$, and $\sigma_P^{[-1]} = [-1]\circ \sigma_P\circ [-1] = \sigma_{-P}$ for $\sigma_P\in N$. As $N$ is an irreducible $\mathbb Z$-module, in this case $\abs N=p$, so that $G = N\rtimes H_1\cong D_{2p}$. 

For $v=3$, since $N$ is an irreducible $p$-torsion $B$-module, we have
\begin{itemize} 
%
\item  $N\cong B/pB$ as a $B$-module if  $p\equiv 2$ mod $3$, so that $N\cong C_p^2$ as an abelian group;
\item  $N\cong B/\pi B$ as a $B$-module if  $p\equiv 1$ mod $3$, where $\pi$  is an irreducible element in $B$ of norm  $p$, so that $N\cong C_p$ as an abelian group. 
\end{itemize}
Similarly, one obtains a description of the $B$-module $N$ for  $v=6$ and $4$, according to whether the ideal $(p)$ is prime or not in $B$, giving the groups $G$ in cases E of Table \ref{table:solvable}. 

Next assume $p$ divides $v=\abs{H_1}$, so that $p=2$ or $3$.
If $p=2$ and $v$ is even, then $[-1]\in H_1$ as it fixes $O$, and  $\sigma_P^{[-1]}=\sigma_{-P}=\sigma_P$ for every point $P$ in $N$.
Hence $[-1]$ is in the center of $G$, so that the natural projection $\tilde Y_1/[-1]\ra \mP^1$ is Galois, 
and factors through $h$ as $[-1]\in H_1$, contradicting the choice of $\tilde h$ as the Galois closure of $h$.
Assume $p=3$ and $3\divides v$.  Since an element of order $3$ fixes a nontrivial subgroup of $T[3]$ (in any action), and every subgroup of $T[3]$ is invariant under conjugation by $[-1]$,  the $B$-module $T[3]$ is reducible.  Since $N\le T[3]$ is irreducible, $N\cong C_3$ as an abelian group.
As an element of order $3$ acts trivially on $C_3$, we get that $\mu_3\leq H_1$ is in the center of $G$.
It follows that the natural projection $\tilde Y_1/\mu_3\ra \mP^1$ is Galois, and factors through $h$, contradicting the choice of $\tilde h$ as
the Galois closure of $h$. 

Henceforth we may assume $p$ is coprime to $v$.
We next determine the ramification types in each possibility for $H_1$. If $v=2$, then the ramification points of the natural projection $\rho_1:\tilde Y_1\ra Y_1$ (resp., $\rho_2:Z_1\ra\mP^1$) are the fixed points of the involution $[-1]\in H_1$ (resp., $[-1]\in \Aut(Z_1)$), that is,
the four $2$-torsion points of $\tilde Y_1$ (resp., $Z_1$).
Since $\eta$ is of odd degree it is injective on $2$-torsion points and hence maps the ramification points of $\rho_1$ to the ramification points of $\rho_2$.
Since $\eta$ is unramified, 
 the multiplicativity of ramification indices in \eqref{equ:EC} implies that $e_h(\rho_1(Q)) = e_{\rho_2}(\eta(Q))/e_{\rho_1}(Q)$ for every point $Q$ of $\tilde Y_1$. As $\rho_1$ and $\rho_2$ both have four ramification points, and $\deg h$ is odd as $p$ is odd, this forces the ramification of $h$ to be type (E1). Conversely, (E1) occurs as the ramification type of the map $\tilde Y_1/\langle[-1]\rangle\ra Z_1/\langle[-1]\rangle$ induced from a degree-$\ell$ isogeny $\tilde Y_1\ra Z_1$ by taking the quotients by the order-two automorphism $[-1]$.

If $v>2$, we identify  $\tilde Y_1$ with the unique elliptic curve $E$ whose automorphism group contains $\mu_v$, cf.~Remark \ref{rem:extra-facts}.(2). 
Note that as $H_1$ preserves $\ker\eta$, it also acts on $Z_1$ by $ \eta(P)^\sigma:=\eta(P^\sigma)$ for $\sigma\in H_1$, $P\in \tilde Y_1$. 
Since this action is faithful, and preserves $O'$, the automorphism group of the elliptic curve $Z_1$  also contains $\mu_v$, and hence we also identify $Z_1$ with $E$. 
As $H_1$ preserves the fibers of $h\circ \rho_1$, it also preserves the fibers of $\rho_2$. As in addition $\deg\rho_1=\deg\rho_2$, the map  $\rho_2$ is also the quotient map by the same subgroup $\mu_v$ of automorphisms of $E$. 

A point $Q$ of $E$ is then a ramification point of order $e$ under $\rho_1$ if and only if $e$ is the maximal integer for which $Q$ is fixed by an $e$-th root of unity $z$ in $A$. 
Computing $\abs{\ker(1-z)}$ for $z\in \mu_v$ via Remark \ref{rem:extra-facts}.(2),
we get that the ramification of $\rho_i$ is  $[3],[3],[3]$ if $v=3$, and $[2^2],[4],[4]$ if $v=4$, and $[2^3],[3^2],[6]$ if $v=6$, for $i=1,2$.
In particular, $\abs{\ker(1-z)}$ is  of order prime to $\ell$, and hence the map $\eta$ maps this kernel injectively to itself. 
In particular, if $Q$ is a ramification point of $\rho_1$, then $\eta(Q)$ is a ramification point of $\rho_2$ with $e_{\rho_1}(Q) = e_{\rho_2}(\eta(Q))$. 
Since in addition $e_{h}(\rho_1(Q))=e_{\rho_2}(\eta(Q))/e_{\rho_1}(Q)$ as $\eta$ is unramified for every point $Q$ of $\tilde Y_1$, the ramification type of $h$ is forced to be one of the types (E2)-(E7) in Table \ref{table:solvable}.
Conversely, (E2)-(E3) (resp.\ (E4)-(E5) and (E6)-(E7)) occur as the ramification type of the map $E/\mu_v\ra E/\mu_v$ induced from a degree-$\ell$ isogeny $E\ra E$ from the Elliptic curve $E$ as above upon taking the  quotient by its automorphisms $\mu_v$ of order dividing $v=3$ (resp.\ $v=4$ and $v=6$).
\end{proof}
\begin{cor}\label{cor:normal} Let $h:Y_1\ra\mP^1$ be an indecomposable covering with branch points $P_1,\ldots, P_r$. If the multiset $\{\lcm(E_{h}(P_i))\suchthat i=1,\ldots,r\}$ is one of  $\{2,2,2,2\}$, $\{3,3,3\}$, $\{2,4,4\}$, $\{2,3,6\}$, then the monodromy group of $h$ is solvable. 
\end{cor}
\begin{proof}
Let $\tilde h:\tilde Y_1\ra \mP^1$ be the Galois closure of $h$. 
By  Abhyankar's lemma~\ref{rem:abh}, $e_{\tilde h}(P)$ is the least common multiple  of the entries in $E_{h}(P)$ for every point $P$. Thus, the Riemann--Hurwitz formula for $\tilde h$ gives  $g_{\tilde Y_1} = 1$. Thus, $\Mon(h)$ is solvable by Proposition \ref{lem:normal-closure}. 
\end{proof}

Lemmas \ref{lem:hurwitz1} and \ref{lem:isogeny-mon} suffice to rule out all ramification data we shall encounter with the exception of the ones appearing in the following lemma. 
\begin{lem}\label{lem:non-existing} There is no degree $\ell$ covering with monodromy group $A_\ell$ or $S_\ell$ whose ramification type appears in Table \ref{table:non-existence}. 

\begin{table}
\caption{Ramification data that does not correspond to a covering with monodromy group $A_\ell$ or $S_\ell$. }
\begin{equation*}\label{table:non-existence}
\begin{array}{| l | l |}
\hline
F1.N1 & [1^2,2^{(\ell-2)/2}], [2^{\ell/2}]\text{ thrice}, [2,1^{\ell-2}]; \\
F1.N2 & [1,3,2^{(\ell-4)/2}], [2^{\ell/2}]\text{ thrice}; \\
F1.N3 & [1^2,4,2^{(\ell-6)/2}], [2^{\ell/2}]\text{ thrice};\\
F1.N4 & [4,2^{(\ell-4)/2}], [1^2,2^{(\ell-2)/2}], [2^{\ell/2}]\text{ twice}; \\
\hline
F4.N1 & [2^{\ell/2}], [2, 3^{(\ell-2)/3}], [1^2, 6^{(\ell-2)/6}]; \\
\hline
I2.N1 & [\ell],  [4,2^{(\ell-4)/2}],  [2^{\ell/2}];  \\
I2.N2 & [\ell],  [2^{\ell/2}],  [2^{\ell/2}],  [2,1^{\ell-2}]. \\
\hline
\end{array}
\end{equation*}
\end{table}
\end{lem}
The proof of Lemma \ref{lem:non-existing} is given in Section \ref{sec:non-existence}.


\section{Proof of Theorem \ref{thm:main}}\label{sec:proof}

We start by deriving Theorem \ref{thm:main} from the following propositions which deal with the case of ramification of almost Galois type. The propositions are separated according to the cases $t=2$ and $t\geq 3$. We use the notation of Setup \ref{sec:setup}, so that for a covering $h:Y_1\ra Y$, the quotient by the  stabilizer of a $t$-set (resp., $t$-tuple of distinct elements) is denoted by $X_t$ (resp., $Y_t$).
\begin{prop}\label{prop:two-set}
For every integer $\alpha>0$, there exist constants $c_{2,\alpha},d_{2,\alpha}>0$ satisfying the following property. 
For every covering $h: Y_1\ra Y$ of degree $\ell$ with monodromy group $A_\ell$ or $S_\ell$, and $2$-point genus bound $\alpha\ell$,
either
$
g_{X_2} - g_{X_{1}} > c_{2,\alpha}\ell - d_{2,\alpha}$
or $g_{Y_1}=0$ and the ramification type of $h$ appears in Table~\ref{table:two-set-stabilizer}. 
\end{prop}

\begin{prop}\label{prop:t-set}
There exists an absolute constant $c_{4}>0$, and for every integer $\beta_1>0$ a constant $d_{4,\beta_1}>0$ satisfying the following property. Let $k\geq 2$ and $t\geq 3$ be integers. 
If  $h: Y_1\ra Y$ is a covering of degree $\ell\geq 2t$ with monodromy group $A_\ell$ or $S_\ell$, and $2$-point genus bound $\alpha\ell$ for $\alpha=\beta_1 t^k$, then
\[
g_{X_t} - g_{X_{t-1}} > (c_{4}\ell - d_{4,\beta_1}t^{2(k+1)})\frac{\binom{\ell}{t}}{\binom{\ell}{2}}.
\]
\end{prop}
\begin{proof}[Proof of Theorem \ref{thm:main}]
Let $\beta,c_2,d_2$ be the constants from Proposition \ref{lem:t=3,4}, let $c_3,d_3$ be the constants from Proposition \ref{prop:Castel}, and let $\lambda_{2}$ be the constant from Corollary \ref{cor:DZ}.  Set $\beta_1:=\lceil  \beta/(1-10^{-6})\rceil +1$. 

We shall define constants $c,d>0$ such that for every degree $\ell$ covering $h:Y_1\ra Y$ with monodromy group $A_\ell$ or $S_\ell$ and integer $2\leq t\leq \ell/2$ satisfying \begin{equation}\label{equ:thm3.1-claim}
g_{X_t}-g_{X_{t-1}}\leq (c\ell-dt^{15})\frac{\binom{\ell}{t}}{\binom{\ell}{2}},
\end{equation} one has $t=2$ and the ramification of $h$ is in Table \ref{table:two-set-stabilizer}. 

We pick $0<c\leq \min\{1/2, c_2, c_3, c_{2,3}, c_{2,\beta_1}, c_{4}\}$, and let $d\geq \max\{1,d_2,d_3,d_{2,3},d_{4,\beta_1}\}$ be sufficiently large so that  $dt^{15}/c\geq  \max\{3^3\lambda_2,\lambda_2 (\beta_1t^{5})^3\}$ 
for every $2\leq t\leq \ell/2$. 
Since $c\leq c_2$ and $d\geq d_2$, the theorem follows from Proposition~\ref{lem:t=3,4} if $g_{Y_1}\geq 2$ and $t=2$, or if $g_{Y_1}\geq \beta t^4$ and $t>2$. 
Hencefore assume that $g_{Y_1}\leq \beta t^4$ if $t>2$, and $g_{Y_1}\leq 1$ if $t=2$.
We may also assume that $c\ell-dt^{15}\geq 0$ and hence that $\ell\geq dt^{15}/c\geq \max\{2t^{15},3^3\lambda_2,\lambda_2(\beta_1t^{5})^3\}$. 
As $c\leq 1/2$ and $d\geq 1$ we have: 
$$g_{X_t} - g_{X_{t-1}}\leq (c\ell-dt^{15})\frac{\binom{\ell}{t}}{\binom{\ell}{2}}<\frac{1}{\ell}\binom{\ell}{t}. $$
Thus, 
Proposition~\ref{prop:Castel} with $\alpha=1$ gives
\begin{equation}\label{equ:from-cast}
\begin{split}
g_{Y_2} & < g_{Y_1}(\ell+1) + 2(\ell-1) 
\quad\text{ for $t=2$, and} \\
g_{Y_2} & < \frac{1}{1-\eps}\Bigl((t-1)g_{Y_1} + \binom{t}{2} + 1\Bigr)\ell  \quad\text{ for $t\geq 3$,}
\end{split}
\end{equation}
where $\eps:=t/(\ell-t+1)$. 
Since $\ell\geq 2t^{15}$, we have $\eps<10^{-6}$ for $t\geq 3$. Thus, the bound $g_{Y_1}\leq \beta t^4$ and  \eqref{equ:from-cast} give  $g_{Y_2}<\beta_1 t^5\ell$ for $t\geq 3$. 
Since $g_{Y_1}\leq 1$ for $t=2$, in this case \eqref{equ:from-cast} gives $g_{Y_2}<3\ell$.
Since in addition $\ell\geq \lambda_2\alpha^3$ for $\alpha=3$ if $t=2$, and for $\alpha=\beta_1t^5$ if $t\geq 3$, 
we deduce that $h$ has $2$-point genus bound $\alpha\ell$.
Since $c\leq \min\{c_{2,3},c_{4,\beta_1}\}$ and $d\geq \max\{d_{2,3},d_{4,\beta_1}\}$, 
the theorem follows from Proposition \ref{prop:two-set} with $\alpha=3$ and Proposition \ref{prop:t-set} with $\alpha=\beta t^k$ and $k=5$. 
%
%
\end{proof}

The proof of Proposition \ref{prop:two-set}  relies on the following estimates of the contributions 
$R_{h_1^2}(P)$, $P\in Y(\K)$, for the natural projection $h_1^2:Y_2\ra Y_1$.
For fixed $\alpha>0$, we write $A=B + O_\alpha(1)$ for expressions $A,B$ to denote that there exists a constant $c_\alpha$ depending only on $\alpha$ such that $\abs{A-B}\leq c_\alpha$ for all values of $A,B$. Let $\nu$ be the constant from Corollary \ref{cor:DZ}. 
\begin{lem}\label{lem:RS-estimates} 
Fix $\alpha>0$.
Let $h:Y_1\ra Y$ be a degree $\ell$ cover with monodromy group $A_\ell$ or $S_\ell$ and $2$-point genus bound $\alpha\ell$.
Let $P$ be a point of $Y$, and  $m:=m_h(P)$. 
If $m<\infty$, then $R_{h_1^2}(P) = S_{h}(P)+O_\alpha(1)$, 
where $S_h(P)$ is 
\begin{align*}
& \ell R_{h}(P) \quad\text{ if }m=1; \\
&\ell\Bigl(\frac{\ell}{2} - \abs{E_{h}(P)} + \abs{\{r\in E_{h}(P)\suchthat 2\nmid r\}}\Bigr)  \quad\text{ if }m=2; \\
 &\ell\Bigl(\frac{\ell}{3} - \abs{E_{h}(P)} + \frac{4}{3}\abs{\{r\in E_{h}(P) \suchthat 3\nmid r\}}\Bigr)  \quad\text{ if }m=3; \\
  &\ell\Bigl(\frac{\ell}{4} - \abs{E_{h}(P)} + \abs{\{r\in E_{h}(P) \suchthat r\equiv 2\!\!\!\!\pmod 4\}}
   + \frac{3}{2}\abs{\{r\in E_{h}(P) \suchthat 2\nmid r\}}\Bigr)  \quad\text{ if }m=4; \\
&\ell\Bigl(\frac{\ell}{6} - \abs{E_{h}(P)} + \abs{\{r\in E_{h}(P) \suchthat r\equiv 3\!\!\!\!\pmod 6\}} + \frac{4}{3}\abs{\{r\in E_{h}(P) \suchthat r\equiv \pm 2\!\!\!\!\pmod 6\}} \\
& \quad + \frac{5}{3}\abs{\{r\in E_{h}(P) \suchthat r\equiv \pm 1\!\!\!\!\pmod 6\}}\Bigr)  \quad\text{ if }m=6.
\end{align*}
If  $m=\infty$,  then either $E_{h}(P) =  [\ell/u,\ldots, \ell/u]$ (in which case $R_{h_1^2}(P) = 0$) or $R_{h_1^2}(P)$ is at least $\ell/(2\nu\alpha^2)$. 
%
\end{lem}
\begin{proof}
First consider the case $m<\infty$. 
Since $\Mon(h)=A_\ell$ or $S_\ell$, and since the sum of the entries of $h$ different from $m$ is at most $\eps_h(P)=O_\alpha(1)$, Lemma \ref{cor:RH-up-left} gives
\begin{align*}
R_{h_1^2}(P) &=   \sum_{r_1,r_2\in E_{h}(P),r_1\neq m} \bigl(r_1-(r_1,r_2)\bigr) +   \sum_{r_1,r_2\in E_{h}(P),r_1= m} \bigl(m-(m,r_2)\bigr)  \\
 &=  \frac{\ell}{m}\Bigl( \sum_{r_1\in E_{h}(P)} 
 \bigl(r_1-(r_1,m)\bigr) +   \sum_{r_2\in E_{h_1}(P)} \bigl(m-(m,r_2)\bigr)\Bigr) +O_\alpha(1)  \\
 &=  \frac{\ell}{m}\sum_{r\in E_{h}(P)}\bigl( r + m - 2(r,m)\bigr) + O_\alpha(1).
 \end{align*}
Evaluating the last sum for each of the possibilities for  finite $m$ gives the desired estimates for $R_{h_1^2}(P)$. For example, for $m=2$ this sum $S_h(P):=\ell/m\sum_{r\in E_h(P)}(r+m-2(r,m))$  is: 
\begin{align*}
     S_h(P) & = \frac{\ell}{2}\Bigl(\ell+2\abs{E_h(P)}-2\abs{\{r\in E_h(P)\suchthat 2 \nmid r\}}-4\abs{\{r\in E_h(P)\suchthat 2 \mid r\}}\Bigr) \\
    & = \ell\Bigl(\frac{\ell}{2}+\abs{E_h(P)}-\abs{\{r\in E_h(P)\suchthat 2 \nmid r\}}-2\abs{\{r\in E_h(P)\suchthat 2 \mid r\}}\Bigr) \\
    & = \ell\Bigl(\frac{\ell}{2} - \abs{E_{h}(P)} + \abs{\{r\in E_{h}(P)\suchthat 2\nmid r\}}\Bigr).
\end{align*}

Let $u:=\abs{E_h(P)}$. If $m=\infty$ and $E_h(P)\neq [\ell/u,\ldots,\ell/u]$, 
then Lemma \ref{lem:bound-rs}.(2) implies that $R_{h_1^2}(P)\geq \ell/(2u)$. The claim follows in this case since $u\leq \nu\alpha^2$ by Corollary \ref{cor:DZ}. 
\end{proof}

\begin{proof}[Proof of Proposition \ref{prop:two-set}]
As in Lemma \ref{lem:RS-estimates}, write $A=B+O_\alpha(1)$ (resp., $A\geq B+O_\alpha(1)$) to denote that $\abs{A-B}$ is bounded (resp. $A-B$ is bounded from below) by a constant depending only on $\alpha$. 

We divide the proof into cases according to the possible multisets 
\[
M_h:=\{m_h(P)>1 \suchthat P\in Y(\K)\}.
\]
By Corollary \ref{cor:DZ}, the set $M_h$ is one of the following: 
(I1) $\{\infty,\infty\}$; 
(I2) $\{\infty,2,2\}$; 
(F1) $\{2,2,2,2\}$; 
(F2) $\{3,3,3\}$; 
(F3) $\{2,4,4\}$; 
(F4) $\{2,3,6\}$;
 or (F5) $\emptyset$. Moreover, the corollary implies $g_{Y_1}\leq \alpha+1$, and also that $g_{Y} = 0$ in all cases but (F5) where $g_Y=1$. 

It suffices to prove that for some constants $c_{2,\alpha},d_{2,\alpha}>0$, either $g_{X_2}>c_{2,\alpha}\ell-d_{2,\alpha}$ or the ramification of $h$ appears in Table \ref{table:two-set-stabilizer}. 
In each case we add a constraint on $c_{2,\alpha}$ and $d_{2,\alpha}$, and determine the ramification types of $h$ for which $g_X\leq c_{2,\alpha}\ell-d_{2,\alpha}$.  
The proposition then follows by taking $c_{2,\alpha}$ and $d_{2,\alpha}$ which satisfy the constraints in each case.  Since it suffices to prove the proposition when $c_{2,\alpha}\ell-d_{2,\alpha}\geq 0$, by requiring $d_{2,\alpha}/c_\alpha>\nu^2\alpha^4$ we may assume $\ell>\nu^2\alpha^4$. 

{\bf Case I1:}
Assume $m_{h}(P_1)=m_{h}(P_2)=\infty$ for two points $P_1,P_2$ of $Y$, 
and $m_h(P)=1$ for any other point $P$ of $Y$.

Let $u=u_{P_1}:=\abs{E_{h}(P_1)}$ and $v=v_{P_2}:=\abs{E_{h}(P_2)}$. Note that $u$ and $v$ are both less than the bound $\nu\alpha^2$ 
from Corollary \ref{cor:DZ}. 
By \eqref{equ:M-RH-t=2} with $g_Y=0$ and $g_{Y_1}=O_\alpha(1)$, and by  Lemma~\ref{lem:RS-estimates}, one has
\begin{equation}\label{equ:Sn-I1-gXY1}
\begin{split}
4(g_{X_2}-1)  &=  \displaystyle 2\ell(g_{Y_1}-1) + \ell\sum_{P\neq P_1,P_2}R_{h}(P) + \sum_{j=1}^2\sum_{r_1,r_2\in E_{h}(P_j)}\bigl(r_1-(r_1,r_2)\bigr) +  O_\alpha(1) \\
&\geq   \ell\Bigl( 2(g_{Y_1}-1) + \sum_{P\neq P_1,P_2}R_{h}(P) + \epsilon_h \Bigr) + O_\alpha(1),
\end{split}
\end{equation}
for  $\epsilon_h = 1/(2\nu\alpha^2)$ whenever $E_{h}(P_1)\neq [\ell/u,\ldots, \ell/u]$ or $E_{h}(P_2)\neq [\ell/v,\ldots,\ell/v]$, and for $\epsilon_h =0$ otherwise.
On the other hand, the Riemann--Hurwitz formula for $h$ gives
\begin{equation}\label{equ:Sn-I1-Y1Y}
\sum_{P\neq P_1,P_2}R_{h}(P) = 2(g_{Y_1}-1) + u + v.
\end{equation}
Substituting the latter into \eqref{equ:Sn-I1-gXY1} gives
\begin{equation}\label{equ:gX2} 
4g_{X_2} \geq 4\ell(g_{Y_1}-1) + (u + v + \epsilon_h)\ell -B_\alpha,
\end{equation} for some constant $B_\alpha$ depending only on $\alpha$. 
Multiplying the latter inequality by $2\nu\alpha^2$, we get that all summands are integers, except perhaps  $2\nu\alpha^2 B_\alpha$. Thus choosing $d_{2,\alpha}$ to be larger than  $2\nu\alpha^2 B_\alpha$, and $c_{2,\alpha}<1/(8\nu\alpha^2)$,
the condition $g_{X_2}\leq c_{2,\alpha}\ell-d_{2,\alpha}$ and \eqref{equ:gX2} force $0\geq 4(g_{Y_1}-1) + u + v + \epsilon_h$,
or equivalently  $u+v\leq 4(1-g_{Y_1})-\epsilon_h$.

It follows that $g_{Y_1}=0$. Moreover, if $u+v=4$ then $\eps_h=0$ and hence $E_{h}(P_1)=[\ell/u,\ldots, \ell/u]$, and $E_{h}(P_2)=[\ell/v,\ldots, \ell/v]$. In this case, since $\ell>\nu^2\alpha^4\geq uv$ the number $\ell/(uv)$ is greater than $1$ and divides the greatest commond divisor of all entries of $E_{h}(P_1)$ and $E_{h}(P_2)$, contradicting the indecomposability of $h$ by Lemma~\ref{lem:hurwitz1}.(1).
Hence  $u+v\leq 3$.
If $u=v=1$ then $E_{h}(P_1)=E_{h}(P_2)=[\ell]$, contradicting the indecomposability of $h$ by Lemma~\ref{lem:hurwitz1}.(1).
We may therefore assume without loss of generality that  $u=1$ and $v=2$.
Then \eqref{equ:Sn-I1-Y1Y} gives $\sum_{P\neq P_1,P_2} R_{h}(P) = 1$,
showing that $h$ has a single branch point $Q$ different from $P_1,P_2$, and moreover  $E_{h}(Q)=[2,1,\dots, 1]$.
Put $E_{h}(P_2)=[a,\ell-a]$. Since $h$ is indecomposable, Lemma \ref{lem:hurwitz1}.(1) shows that 
$(a,\ell)=1$.  Thus, the ramification type of $h$ corresponding to $P_1,P_2,Q$ is
$[\ell], [a,\ell-a], [2,1,\dots ,1]$, with $(a,\ell)=1$, corresponding to type I1.1 in Table \ref{table:two-set-stabilizer}.

{\bf Case I2:} Assume that $m_{h}(P_0)=\infty, m_{h}(P_1)=m_{h}(P_2)=2$ for points $P_0,P_1,P_2$ of $Y$, and $m_h(P)=1$ for all other points $P$ of $Y$.
Denote $u=u_{P_0}:=\abs{E_{h}(P_0)}$ and $O:= \abs{\{r\in E_{h}(P_i) \suchthat \text{$r$ is odd},i=1,2\}}$.
Note that $O$ is even.
By \eqref{equ:M-RH-t=2} and Lemma~\ref{lem:RS-estimates} one has
\begin{equation}\label{equ:Sn-inf-2-2-gXY1}
\begin{split}
 4g_{X_2}  &=  2\ell(g_{Y_1}-1) + \sum_{r_1,r_2\in E_{h}(P_0)}\bigl(r_1-(r_1,r_2)\bigr) + \ell\sum_{P\neq P_1,P_2}R_{h}(P) \\
&  \quad + \ell\sum_{i=1}^2\Bigl(-\frac{1}{2} + \frac{\ell}{2} -\abs{E_{h}(P_i)} +  \abs{\{\text{odd }r\in E_{h}(P_i)\}}\Bigr) +O_\alpha(1).
\end{split}
\end{equation}
On the other hand the Riemann--Hurwitz formula for $h$ gives
\begin{equation}\label{equ:Sn-inf-2-2-gY1Y}
2(g_{Y_1}-1) + u =   \sum_{P\neq P_1,P_2}R_{h}(P) + \sum_{i=1}^2\Bigl(\frac{\ell}{2} - \abs{E_{h}(P_i)}\Bigr).
\end{equation}

Substituting \eqref{equ:Sn-inf-2-2-gY1Y} into  \eqref{equ:Sn-inf-2-2-gXY1} and applying Lemma \ref{lem:RS-estimates} we get
\begin{equation}\label{equ:Sn-inf-2-2-gXY1-refined}
\begin{split}
4g_{X_2} & = 4\ell(g_{Y_1}-1) + \sum_{r_1,r_2\in E_{h}(P_0)}\bigl(r_1-(r_1,r_2)\bigr) + (u-1)\ell + O\ell + O_\alpha(1) \\
& \geq  4\ell(g_{Y_1}-1)  + \epsilon_h \ell + (u-1)\ell + O\ell + O_\alpha(1),
\end{split}
\end{equation}
for $\epsilon_h  = 1/(2\nu\alpha^2)$ whenever $E_{h}(P_0)\neq [\ell/u,\ldots,\ell/u]$ and for $\epsilon_h =0$ otherwise.
Thus, by taking $d_{2,\alpha}$ to be sufficiently large and $0<c_{2,\alpha}<1/(8\nu\alpha^2)$, the assumption $g_{X_2}\leq c_{2,\alpha}\ell-d_{2,\alpha}$ and \eqref{equ:Sn-inf-2-2-gXY1-refined} force
\begin{equation}\label{equ:Sn-gY1Y-conclusion} 
4(1-g_{Y_1}) + 1  -\epsilon_h \geq u + O. 
\end{equation}

Hence $g_{Y_1}\leq 1$. First consider the case $g_{Y_1}=1$ in which  \eqref{equ:Sn-gY1Y-conclusion} gives  $u =1$, $\epsilon_h=0$, and  $O=0$.
The only ramification type of $h$ satisfying the latter constraints and \eqref{equ:Sn-inf-2-2-gY1Y} are types I2.N1, I2.N2 in Table~\ref{table:non-existence}.
These do not correspond to any covering by Lemma \ref{lem:non-existing}.

Now assume $g_{Y_1}=0$. In this case \eqref{equ:Sn-gY1Y-conclusion} gives $u + O \leq 5-\epsilon_h$.
Since $h$ is indecomposable, Lemma \ref{lem:hurwitz1}.(1) shows that at least one of $E_{h}(P_i)$, $i=1,2$, contains an odd number. Since the sum of entries of $E_h(P_1)$ has the same parity as that of $E_h(P_2)$, the number $O$ is even, and hence $O\geq 2$. It follows from \eqref{equ:Sn-gY1Y-conclusion} that $u\leq 3$.

If $u=1$, then $E_{h}(P_0)=[\ell]$, and hence $\eps_h=0$, and \eqref{equ:Sn-inf-2-2-gXY1-refined} and  \eqref{equ:Sn-gY1Y-conclusion} become equalities.
It follows that $O=4$. Moreover, if $\ell$ is even $E_h(P_i),i=1,2$ do not consist of even entries, by Lemma \ref{lem:hurwitz1}. 
The ramification types of $h$ satisfying these constraints and  \eqref{equ:Sn-inf-2-2-gY1Y} correspond, over $P_0,P_1,P_2$ and possibly another branch point $Q$,
to types I2.1-I2.8 in Table~\ref{table:two-set-stabilizer}.

If $u=2$, since $O$ is even and nontrivial by Lemma \ref{lem:hurwitz1}.(1), by \eqref{equ:Sn-gY1Y-conclusion} we have $O=2$. Write $E_{h}(P_0)=[a,\ell-a]$ and let $d:=(a,\ell)$. If $d>1$, Lemma \ref{lem:hurwitz1}.(2) applied with a prime $p$ dividing $d$ contradicts the indecomposability of $h$. Thus $d=(a,\ell)=1$.
The only ramification types of $h$ which satisfy these constraints and \eqref{equ:Sn-inf-2-2-gY1Y}
are types I2.9-I2.15 in Table \ref{table:two-set-stabilizer} corresponding to the points $P_0,P_1,P_2$ and possibly another point~$Q$.

It remains to treat the case $u=3$. In this case \eqref{equ:Sn-gY1Y-conclusion} is necessarily an equality with $O=2$ and $\epsilon_h =0$, forcing $E_{h}(P_0)=[\ell/3,\ell/3,\ell/3]$.
Since $\ell>3$, Lemma \ref{lem:hurwitz1}.(2) applied with a prime dividing $\ell/3$ contradicts the indecomposability of $h$.

{\bf Case F1:} Assume $m_h(P_i)=2$  for four points $P_i,i=1,2,3,4$ of $Y$, and $m_h(P)=1$ for all other points of $Y$.
Then 
\eqref{equ:M-RH-t=2} and Lemma \ref{lem:RS-estimates} give
\begin{equation}\label{equ:F1-main}
\begin{split}
4(g_{X_2}-1)  &=  2\ell(g_{Y_1}-1) +  \ell\biggl( -2 + \sum_{j=1}^4\Bigl(\frac{\ell}{2}-\abs{E_{h}(P_j)}+\abs{\{\text{odd }r\in E_{h}(P_j)\}}\Bigr)\biggr) \\ &\quad + \ell\sum_{P\neq P_1,\ldots, P_4}R_{h}(P) + O_\alpha(1).
\end{split}
\end{equation}
On the other hand the Riemann--Hurwitz formula for $h$ gives 
\[
2(g_{Y_1}-1) = \sum_{j=1}^2\Bigl(\frac{\ell}{2}-\abs{E_{h}(P_j)}\Bigr) + \sum_{P\neq P_1,\ldots,P_4}R_{h}(P).
\]
Substituting the latter equality into the former gives 
\begin{equation*}
4(g_{X_2}-1) = 4\ell(g_{Y_1}-1) + \ell\Bigl(-2+\sum_{j=1}^4\abs{\{\text{odd }r\in E_{h}(P_j)\}}\Bigr) + O_\alpha(1).
\end{equation*}
Hence for sufficiently large $d_{2,\alpha}$ and $c_{2,\alpha}<1/4$, the latter equality and the assumption $g_{X_2}\leq c_{2,\alpha}\ell-d_{2,\alpha}$ force
\begin{equation}\label{equ:Sn-2222-gY1Y}
 \abs{\{\text{odd }r\in E_{h}(P_j),j=1,\ldots, 4\}} =4(1-g_{Y_1}) + 2.
\end{equation}
Hence $g_{Y_1}\leq 1$. If $g_{Y_1}=1$, there are exactly two odd entries among $E_{h}(P_j)$, for $j=1,\ldots, 4$. The only possible ramification types of $h$ which satisfy these constraints and the Riemann--Hurwitz formula for $h$ are types F1.N1-F1.N4 in Table \ref{table:non-existence}.
These ramification types do not correspond to any covering by Lemma \ref{lem:non-existing}.

Assume $g_{Y_1}=0$, in which case there is a total of $6$ odd entries in $E_{h}(P_j)$,  for $j=1,\ldots, 4$, by \eqref{equ:Sn-2222-gY1Y}.
Since $h$ is indecomposable, $E_{h}(P_j)$ can have no odd entries for at most one point among $P_1,\ldots, P_4$, by Lemma \ref{lem:hurwitz1}.(1).
The ramification types of $h$ satisfying these constraints and the Riemann--Hurwitz formula for $h$ are types F1.1-F1.9 in Table~\ref{table:two-set-stabilizer}.

{\bf Case F2:}  Assume $m_h(P_j)=3,j=1,2,3$ for three points $P_1,P_2,P_3$ of $Y$, and $m_h(P)=1$ for all other points $P$ of $Y$.
By \eqref{equ:M-RH-t=2} and Lemma \ref{lem:RS-estimates} one has
\begin{equation}\label{equ:Sn-333}
\begin{split}
4(g_{X_2}-1) = &  2\ell(g_{Y_1}-1) + \ell\sum_{j=1}^3\Bigl(\frac{\ell}{3} -\abs{E_{h}(P_j)}+\frac{4}{3}\abs{\{r\in E_{h}(P_j) \suchthat 3\nmid r\}}\Bigr) \\
&  + \ell\sum_{P\neq P_1,P_2,P_3}R_{h}(P) + O_\alpha(1).
\end{split}
\end{equation}
The Riemann--Hurwitz formula for $h$ gives 
\[
-2 = \sum_{j=1}^3 \Bigl(\frac\ell3-\abs{E_{h}(P_j)}\Bigr) + \sum_{P\neq P_1,P_2,P_3}R_{h}(P).
\]
Substituting the latter equality into \eqref{equ:Sn-333} gives:
\begin{equation*}
4(g_{X_2}-1) = 4\ell(g_{Y_1}-1) + \frac{4\ell}{3}\sum_{j=1}^3\abs{\{r\in E_{h}(P_j)\suchthat 3\nmid r\}} + O_\alpha(1).
\end{equation*}
When $d_{2,\alpha}$ is sufficiently large and $c_{2,\alpha}<1/4$, this equality and the assumption  $g_{X_2}\leq c_{2,\alpha}\ell-d_{2,\alpha}$ force
\begin{equation}\label{equ:Sn-333-gXY1}
4(1-g_{Y_1}) = \frac{4}{3}\sum_{j=1}^3\abs{\{r\in E_{h}(P_j)\suchthat 3\nmid r\}}. 
\end{equation}
 In particular $g_{Y_1}\leq 1$, and if $g_{Y_1}=1$ then the ramification type of $h$ is $[3^{\ell/3}]$ thrice, contradicting  $\Mon(h)=A_\ell$ or $S_\ell$ by Corollary \ref{cor:normal}.
Assume $g_{Y_1}=0$. Then \eqref{equ:Sn-333-gXY1} 
 shows that there are a  total of three coprime to $3$ entries in $E_{h}(P_j)$, for $j=1,2,3$.
Since $h$ is indecomposable, at most one point has all indices divisible by $3$ by Lemma \ref{lem:hurwitz1}.(1).
The only ramification types satisfying these constraint and the Riemann--Hurwitz formula for $h$ is $[1,3^{(n-1)/3}]$ thrice,  contradicting 
$\Mon(h)=A_\ell$ or $S_\ell$, by Lemma \ref{cor:normal}.

{\bf Case F3:}
Assume  $m_h(P_0)=2$, $m_h(P_1)=m_h(P_2)=4$ for points $P_0, P_1, P_2$ of $Y$, and $m_h(P)=1$ for all other points $P$ of $Y$.
By \eqref{equ:M-RH-t=2} and Lemma \ref{lem:RS-estimates} one has
\begin{equation}\label{equ:Sn-F3}
\begin{split}
 4g_{X_2}  & = \displaystyle 2\ell(g_{Y_1}-1) + \ell\bigl(\frac{\ell}{2}-\abs{E_{h}(P_0)} + \abs{\{\text{odd }r\in E_{h}(P_0)\}} - \frac{1}{2}\bigr)  + \ell \sum_{P\neq P_0,P_1,P_2}R_{h}(P)  \\
   & \quad  +  \ell\sum_{i=1}^2\Bigl(-\frac{1}{2} + \frac{\ell}{4} - \abs{E_{h}(P_i)} + \abs{\{r\in E_{h}(P_i) \suchthat r\equiv 2\!\!\!\!\pmod 4 \}} \\
 & \qquad\quad\qquad +\frac{3}{2}\abs{\{\text{odd }r\in E_{h}(P_i)\}}\Bigr) +  O_\alpha(1).
 \end{split}
\end{equation}
The Riemann--Hurwitz formula for $h$ gives
\[
 2(g_{Y_1}-1) = \Bigl(\frac{\ell}{2} -\abs{E_{h}(P_0)}\Bigr) + \sum_{P\neq P_0,P_1,P_2}R_{h}(P)  +  \sum_{i=1}^2\Bigl(\frac{\ell}{4} - \abs{E_{h}(P_i)}\Bigr).
\]
Substituting the latter into \eqref{equ:Sn-F3} gives
\begin{equation}\label{equ:Sn-F3-gXY1}
\begin{split}
4g_{X_2}  = & 4\ell(g_{Y_1}-1)  + \ell\abs{\{\text{odd }r\in E_{h}(P_0)\}} + \ell\sum_{i=1}^2\Bigl(\abs{\{r\in E_{h}(P_i) \suchthat r\equiv 2\!\!\!\!\pmod 4\}} \\
& +\frac{3}{2}\abs{\{\text{odd }r\in E_{h}(P_i),i=1,2\}}\Bigr) - \ell + O_\alpha(1).
\end{split}
\end{equation}
Note that since the total number of odd entries in $E_h(P_i),i=1,2$ is even, the sum on the right hand side of the above equality is an integer. 
Hence for sufficiently large $d_{2,\alpha}$ and for $c_{2,\alpha}<1/4$ we get

\begin{equation}\label{equ:Sn-244-conclusion}
\begin{split}
5-4g_{Y_1}  & =\abs{\{\text{odd }r\in E_{h}(P_0)\}} + \abs{\{r\in E_{h}(P_i) \suchthat r\equiv 2\!\!\!\!\pmod{4}, i=1,2\}} \\
 & \quad + \frac{3}{2}\abs{\{\text{odd }r\in E_{h}(P_i),i=1,2\}}.
 \end{split}
\end{equation}
Hence $g_{Y_1}\leq 1$. If $g_{Y_1}=1$, there is no ramification type of $h$ which satisfies \eqref{equ:Sn-244-conclusion} and the Riemann--Hurwitz formula for $h$. Assume $g_{Y_1}=0$.
Note that since $h$ is indecomposable, Lemma \ref{lem:hurwitz1}.(1) implies that the total number of odds in $E_{h}(P_i),i=1,2,$ is at least two (recall it is even). Moreover, if $\ell$ is even then there are at least four odd numbers in $E_{h}(P_0),E_{h}(P_1), E_{h}(P_2)$ (since at least two of these contain at least two odd numbers).  The only ramification types of $h$ which satisfy these constraints  
and the Riemann--Hurwitz formula for $h$ are types F3.1-F3.3 over the points $P_0,P_1,P_2$. 

{\bf Case F4:}
Assume  $m_h(P_1)=2, m_h(P_2)=3,m_h(P_3)=6$ for points $P_1,P_2,P_3$ of $Y$ and that $m_h(P)=1$ for all other points of $Y$.
By \eqref{equ:M-RH-t=2} and Lemma \ref{lem:RS-estimates}, one has
\begin{equation}\label{equ:F4-main}
\begin{split}
 4g_{X_2}  & = 2\ell(g_{Y_1}-1) + \ell\Bigl(- \frac{1}{2} + \frac{\ell}{2}-\abs{E_{h}(P_1)} + \abs{\{\text{odd }r\in E_{h}(P_1)\}}\Bigr)   \\
  &  \quad + \ell \!\!\!\sum_{P\neq P_1,P_2,P_3}\!\!\!R_{h}(P) 
   +  \ell\Bigl(\frac{\ell}{3} - \abs{E_{h}(P_2)} + \frac{4}{3}\abs{\{r\in E_{h}(P_2) \suchthat 3\nmid r\}}\Bigr) \\
   & \quad + \ell\Bigl( -\frac{1}{6} + \frac{\ell}{6}-\abs{E_{h}(P_3)}
    + \abs{\{r\in E_{h}(P_3) \suchthat r\equiv 3\!\!\!\!\pmod 6 \}}  \\
    &\qquad\quad +  \frac{4}{3}\abs{\{r\in E_{h}(P_3) \suchthat r\equiv \pm 2\!\!\!\!\pmod 6 \}} \\    
    & \qquad\quad  + \frac{5}{3}\abs{\{r\in E_{h}(P_3) \suchthat r\equiv \pm 1\!\!\!\!\pmod 6 \}}\Bigr) +  O_\alpha(1).
\end{split}
\end{equation}
The Riemann--Hurwitz formula for $h$ gives
\[ 2(g_{Y_1}-1) = \Bigl(\frac{\ell}{2} -\abs{E_{h}(P_1)}\Bigr) + \!\!\!\sum_{P\neq P_1,P_2,P_3}\!\!\! R_{h}(P)  +  \Bigl(\frac{\ell}{3} - \abs{E_{h}(P_2)}\Bigr) + \Bigl(\frac{\ell}{6} - \abs{E_{h}(P_3)}\Bigr).
\]
Substituting the latter into \eqref{equ:F4-main} we get
\begin{equation}\label{equ:Sn-F4-gXY1}
\begin{split}
4g_{X_2} = & 4\ell(g_{Y_1}-1) -2\ell/3  + \ell\Bigl(\abs{\{\text{odd }r\in E_{h}(P_1)\}} + 
\frac{4}{3}\abs{\{r\in E_{h}(P_2) \suchthat (r,3)=1\}}  \\
   &  + \abs{\{r\in E_{h}(P_3) \suchthat r\equiv 3\!\!\!\! \pmod{6}\}}   + \frac{4}{3}\abs{\{r\in E_{h}(P_3) \suchthat r\equiv \pm 2\!\!\!\! \pmod{6}\}}   \\
   &  + \frac{5}{3}\abs{\{r\in E_{h}(P_3)\suchthat r\equiv \pm 1\!\!\!\!\pmod{6}\}}\Bigr) + O_\alpha(1).
\end{split}
\end{equation}
Note that the coefficient of $\ell$ in the right hand side is a third of an integer. 
Hence for sufficiently large $d_{2,\alpha}$ and for $c_{2,\alpha}<1/12$, \eqref{equ:Sn-F4-gXY1} and the assumption $g_{X_2}\leq c_{2,\alpha}\ell-d_{2,\alpha}$ force
\begin{equation}\label{equ:Sn-236-conclusion}
\begin{split}
\frac{14}{3}-4g_{Y_1} &=  \abs{\{\text{odd }r\in E_{h}(P_1)\}} + \frac{4}{3}\abs{\{r\in E_{h}(P_i) \suchthat (r,3)=1\}} \\
   & \quad + \abs{\{r\in E_{h}(P_3) \suchthat r\equiv 3 \!\!\!\!\pmod{6}\}}  \\
   & \quad+  \frac{4}{3}\abs{\{r\in E_{h}(P_3) \suchthat r\equiv \pm 2\!\!\!\!\pmod{6}\}}    \\
   &  \quad + \frac{5}{3}\abs{\{r\in E_{h}(P_3) \suchthat r\equiv \pm 1\!\!\!\! \pmod{6}\}}.
\end{split}
\end{equation}
If $g_{Y_1}=1$, there is no ramification type of $h$ which satisfies \eqref{equ:Sn-236-conclusion} and the Riemann--Hurwitz formula for $h$.  Assume $g_{Y_1}=0$.
Note that since $h$ is indecomposable, Lemma \ref{lem:hurwitz1}.(1) implies that there are odd entries in at least one of  $E_{h}(P_1), E_{h}(P_3)$ and there are coprime to $3$ entries  in at least one of $E_{h}(P_2),E_{h}(P_3)$.
The only ramification types of $h$ over $P_1,P_2,P_3$ which satisfy these restrictions and the Riemann--Hurwitz formula for $h$ are types F4.1--F4.6 and the following:
\begin{equation}\label{table:non-236}
\begin{array}{| l |}
\hline
\left[2^{n/2}], [3^{n/3}\right], [1^2,4,6^{(n-6)/6}]  \\
\left[1^2,2^{(n-2)/2}\right], [3^{n/3}], [2,4,6^{(n-6)/6}] \\
\left[1,2^{(n-1)/2}\right], [3^{n/3}], [2,3,4,6^{(n-9)/6}]  \\
\hline
\left[2^{n/2}\right], [1,2,3^{(n-3)/3}], [3^2,6^{(n-6)/6}]  \\ 
\left[2^{n/2}\right], [1,3^{(n-1)/3}], [3^2,4,6^{(n-10)/6}]  \\ 
\left[2^{n/2}\right], [2,3^{(n-2)/3}], [2,3^2,6^{(n-8)/6}]   \\ 
\left[2^{n/2}\right], [3^{n/3}], [2,3^2,4,6^{(n-12)/6}]  \\ 
\hline
\left[2^{n/2}\right], [2,3^{(n-2)/3}], [1^2,6^{(n-2)/6}]  \\ 
\hline
\end{array}
\end{equation}
The first three types in \eqref{table:non-236} do not occur as a ramification type of a covering by Lemma \ref{lem:hurwitz1}.(2), the next four do not occur by Lemma \ref{lem:hurwitz1}.(3), and the last is type F4.N1 which does not appear by  Lemma \ref{lem:non-existing}. 

{\bf Case F5:} Assume $m_h(P)=1$  for all points $P$ of $Y$, and that $g_Y=1$. 
Then 
\eqref{equ:M-RH-t=2} and Lemma \ref{lem:RS-estimates} give
\begin{equation}\label{equ:F5-main}
4(g_{X_2}-1) =  2\ell(g_{Y_1}-1) +  \ell\sum_{P\neq P_1,\ldots, P_4}R_{h}(P) + O_\alpha(1).
\end{equation}
Hence for sufficiently large $d_{2,\alpha}$ and $c_{2,\alpha}<1/4$, \eqref{equ:F5-main} and the assumption $g_{X_2}\leq c_{2,\alpha}\ell-d_{2,\alpha}$ force
$2(1-g_{Y_1})  =  \sum_{P\in Y(\K)}R_{h}(P)$.
On the other hand the Riemann--Hurwitz formula for $h$ gives 
$2(g_{Y_1}-1) =  \sum_{P\in Y(\K)}R_{h}(P).$
These two equalities force $g_{Y_1}=1$ and $h$ is unramified, 
contradicting the assumption that $\Mon(h)=A_\ell$ or $S_\ell$, by Remark \ref{rem:extra-facts}.(4). 
\end{proof}

The proof for $t\geq 3$ relies on the following bound on the Riemann--Hurwitz contribution $R_{\pi_t}(f_t^{-1}(P))$, where  
$\pi_t:Y_t\ra X_t$ and $f_t:X_t\ra Y$ are the natural projections. 
Let $E_0$ (resp.~$\nu$) be the constant from Proposition \ref{lem:pi-form} (resp.~Corollary \ref{cor:DZ}). 
\begin{lem}\label{cor:bounding-mP} 
Let $h:Y_1\ra Y$ be a degree $\ell$ covering with monodromy group $A_\ell$ or $S_\ell$ and $2$-point genus bound $\alpha\ell$.
Let $3\leq t\leq \ell/2$ an integer.
Assume $P$ is a point of $Y$ with error $\eps_h(P) <\ell/2-1$.  
Then 
\[
R_{\pi_t}(f_t^{-1}(P)) < 
\begin{cases}\left((\eps_h(P)+\frac{1}{2})\binom{t}{2} + E_0t^4\right)\frac{(\ell-2)!}{(\ell-t)!} & \text{ if }m_h(P)<\infty \\
\left(\nu\alpha^2\binom{t}{2} + E_0t^4\right)\frac{(\ell-2)!}{(\ell-t)!} & \text{ if }m_h(P)=\infty.
 \end{cases}
 \]
\end{lem}
\begin{proof}
Let  $m:=m_h(P)$. 
We estimate the main term of Proposition \ref{lem:pi-form}.(2). 
Let $\Mon(h)$ act on the set $S$. 
Let $x\in \Mon(h)$ be a branch cycle over $P$, 
and let 
\[
M_h:=\sum_{(\theta_1,\ldots,\theta_{t-1})\in O_{t-1}}\frac{\hat r_1\cdots \hat r_{t-1}}{\lcm(r_1,\ldots,r_{t-1})},
\]
where $r_i := \abs{\theta_i}$, and $\hat r_i$ is $r_i$ minus the number of $j<i$ with $\theta_j = \theta_i$ for $i=1,\ldots, t-1$, and $O_{t-1}$ is the set of tuples $(\theta_1,\ldots,\theta_{t-1})$ of orbits of $x$ such that $r_1$ is even and $v_2(r_1)>v_2(r_k)$ for $k>1$.
Note that since $\lcm(r_1,\ldots,r_{t-1})\geq r_1 = \hat r_1$, we have 
\begin{equation}\label{equ:sum-split}
\begin{split}
M_h&\leq \sum_{(\theta_1,\theta_2,\ldots,\theta_{t-1})\in O_{t-1}}\hat r_2\cdots \hat r_{t-1} \\
&=  \sum_{\substack{(\theta_1,\ldots,\theta_{t-1})\in O_{t-1} \\ r_1\neq m}} \hat r_2\cdots \hat r_{t-1} +  \sum_{\substack{(\theta_1,\ldots,\theta_{t-1})\in O_{t-1} \\ r_1 = m}} \hat r_2\cdots \hat r_{t-1}.
\end{split}
\end{equation}
We bound each of the summands on the right hand side, noting that the last sum is zero if $m=\infty$. 

Fix $2\leq k\leq t-1$ and orbits $\theta_1,\ldots,\theta_{k-1}$, and let $U_{\theta_1}$ be the set of orbits $\theta$ of $x$ such that  
$v_2(\abs\theta)<v_2(r_1)$. 
We claim that the sum $\sum_{\theta_{k}\in U_{\theta_1}}\hat r_{k}$
is bounded by $\ell-k$ if $r_1\neq m$, and by $\eps_h(P)-(k-2)$ if $r_1=m$. 
Indeed, note that 
the sum 
$\sum_{\theta_{k}\in U_{k}}r_{k}$ is at most $\ell-r_1$,
and  at most $\eps_h(P)$ if further $r_1=m$. 
Since in addition $r_1\geq 2$, and 
$\sum_{\theta_{k}\in U_{\theta_1}}\abs{\{1<j<k \suchthat \theta_j = \theta_k\}} =  k-2$, 
we get
\begin{align*}
 \sum_{\theta_{k}\in U_{\theta_1}}\hat r_{k} & = \sum_{\theta_{k}\in U_{\theta_1}}r_{k} -  \sum_{\theta_{k}\in U_{\theta_1}}\abs{\{1<j<k \suchthat \theta_j = \theta_k\}} \\
 & \leq 
\begin{cases} \eps_h(P)-(k-2) & \text{if }r_1=m \\
\ell-2-(k-2) & \text{otherwise,}
\end{cases}
\end{align*}
proving the claim. 

Applying the claim for $k=t-1,t-2,\ldots, 2$, we get
\begin{align*}
 \sum_{(\theta_1,\ldots,\theta_{t-1})\in O_{t-1}, r_1\neq m} \hat r_2\cdots \hat r_{t-1} &  \leq  
 (\ell-t+1)\sum_{(\theta_1,\ldots,\theta_{t-2})\in O_{t-2}, r_1\neq m} \hat r_2\cdots \hat r_{t-2}   \\
 & \leq \cdots \leq  \sum_{\theta_1\in O_1,r_1\neq m}\frac{(\ell-2)!}{(\ell-t)!}  = \tilde N_h\frac{(\ell-2)!}{(\ell-t)!},
 \end{align*}
where $\tilde N_h$ is the number of even entries in $E_h(P)$ that are different from $m$. 
Similarly if $r_1=m<\infty$, applying the claim for $k=t-1,\ldots,2$ gives
\[ \sum_{(\theta_1,\ldots,\theta_{t-1})\in O_{t-1}, r_1 = m} \hat r_2\cdots \hat r_{t-1}  \leq 
 \displaystyle \sum_{\theta_1\in O_1, r_1=m} \frac{\eps_h(P)!}{(\eps_h(P)-(t-2))!} \leq \frac{\ell}{m}\frac{\eps_h(P)!}{(\eps_h(P)-(t-2))!},\]
 where $(\eps_h(P)-t+2)!$ is defined to be $1$ if $\eps_h(P)\leq t-2$. 
 Note that in case $r_1=m$, we have $m\geq 2$ since $r_1$ is even. 
In total,  estimating each of summands in \eqref{equ:sum-split},
we get
\[
M_h < \begin{cases}
\tilde N_h\frac{(\ell-2)!}{(\ell-t)!} + \frac{\ell}{2}\frac{\eps_h(P)!}{(\eps_h(P)-t+2)!} & \text{ if }m<\infty \\ 
\tilde N_h\frac{(\ell-2)!}{(\ell-t)!} & \text{ if }m=\infty. 
\end{cases}
\]
Since $t\geq 3$, and $\eps_h(P)+1<\ell/2$, a straightforward check shows that  $\ell \frac{\eps_h(P)!}{(\eps_h(P)-t+2)!} < (\eps_h(P)+1)\frac{(\ell-2)!}{(\ell-t)!}$. 
Note that  $\tilde N_h\leq \eps_h(P)/2$ if $m$ is finite, and $\tilde N_h\leq \nu\alpha^2$ if $m=\infty$, by Corollary \ref{cor:DZ}. 
In total we get $M_h< (\eps_h(P)+1/2)\frac{(\ell-2)!}{(\ell-t)!}$ if $m$ is finite, and $M_h<\nu\alpha^2\frac{(\ell-2)!}{(\ell-t)!}$ if $m=\infty$. 
Thus Proposition \ref{lem:pi-form}  gives
\[ R_{\pi_t}(f_t^{-1}(P))  < \binom{t}{2}M_h + E_0t^4\frac{(\ell-2)!}{(\ell-t)!} <
\begin{cases} ((\eps_h(P)+\frac{1}{2})\binom{t}{2} + E_0t^4)\frac{(\ell-2)!}{(\ell-t)!} & \text{ if }m<\infty \\
(\nu\alpha^2\binom{t}{2} + E_0t^4)\frac{(\ell-2)!}{(\ell-t)!} & \text{ if }m=\infty. \qedhere
\end{cases} 
\]
\end{proof}
\begin{proof}[Proof of Proposition \ref{prop:t-set}]
Denote $R_{\pi_t}:=\sum_{P\in Y(\K)}R_{\pi_t}(f_t^{-1}(P))$. 
By inequality \eqref{equ:M-RH-inequ} of Remark \ref{rem:main-RH}, and the Riemann--Hurwitz formula the natural projection $Y_{t-1}\ra Y_2$, we get: 
\begin{equation}\label{equ:main-RH2}
\begin{split}
2t!(g_{X_t}-g_{X_{t-1}}) & \geq
2(\ell-2t+1)(g_{Y_{t-1}}-1)  - R_{\pi_t} \\ 
 & \geq 2(1-\eps)\frac{(\ell-2)!}{(\ell-t)!}(g_{Y_2}-1)  - R_{\pi_t}, 
\end{split}
\end{equation}
where $\eps := t/(\ell-t+1)$. We show that $R_{\pi_t}$ is bounded by a constant times $t^{2k+2}\frac{(\ell-2)!}{(\ell-t)!}$ while the first term on the right hand side is at least a linear factor in $\ell$ times $\frac{(\ell-2)!}{(\ell-t)!}$. 

{\bf Step I:} {\it Bounding $R_{\pi_t}$ using Lemma \ref{cor:bounding-mP}. }
Let $\eps_h(P)$ denote the error over a point $P$ of $Y$, and let $\nu\alpha^2$ denote the bound on $\abs{h^{-1}(P)}$ from Corollary \ref{cor:DZ} in case $m_h(P)=\infty$.
By Lemma \ref{cor:bounding-mP},  
 $R_{\pi_t}(f_t^{-1}(P))$ is at most 
$((\eps_h(P)+1/2)\binom{t}{2} + E_0t^4))\frac{(\ell-2)!}{(\ell-t)!}$ if $m_h(P)<\infty$ and at most 
$(\nu\alpha^2\binom{t}{2} + E_0t^4)\frac{(\ell-2)!}{(\ell-t)!}$ if $m_h(P)=\infty$.  
By definition, 
$\eps_h(P)+1/2< 84(\alpha+3/2)$ for every point $P$ of $Y$. 
Moreover by Corollary \ref{cor:DZ}, there are at most four branch points $P$ with $m_h(P)>1$, at most $2(\alpha+1)$ with $m_h(P)=1$, 
and at most two with $m_h(P)=\infty$.
In total we have the following bound on $R_{\pi_t}$
\begin{equation*}
\begin{split}
 R_{\pi_t} &   =  \sum_{P\in Y(\K)\suchthat m_h(P)<\infty}R_{\pi_t}(f_t^{-1}(P)) + \sum_{P\in Y(\K),m_h(P)=\infty}R_{\pi_t}(f_t^{-1}(P)) \\
& < \Biggl(\sum_{P\in Y(\K)\suchthat m_h(P)<\infty}\Bigl(\binom{t}{2}(\eps_h(P)+1/2) + E_0t^4 \Bigr) + 2\Bigl(\nu\alpha^2\binom{t}{2}  + E_0t^4\Bigr))\Biggr)\frac{(\ell-2)!}{(\ell-t)!}
\\
&  <  \Biggl( (2\alpha+6)\Bigl(84(\alpha+3/2)\binom{t}{2} + E_0t^4\Bigr) 
+ 2\nu\alpha^2\binom{t}{2}  + 2E_0t^4 \Biggr)\frac{(\ell-2)!}{(\ell-t)!}. 
\end{split}
\end{equation*}
As $\alpha=\beta_1t^k$ for $k\geq 2$, this gives
$R_{\pi_t}\leq d_{5,\beta_1}t^{2k+2}(\ell-2)!/(\ell-t)!$ for some constant $d_{5,\beta_1}>0$ depending only on $\beta_1$. 


{\bf Step II:} {\it Bounding $g_{Y_2}$ from below and estimating \eqref{equ:main-RH2}. }
We next claim similarly to the proof of Theorem \ref{thm:main} that if the ramification type of $h$ is not $[\ell],[a,\ell-a],[2,1^{\ell-2}]$ then $g_{Y_2}>2c'\ell-2d'-1$, for constants $c'\leq \min\{1/3,c_2,c_{2,3}\}$ and $d'\geq\max\{2,d,2^8d_2,d_{2,3}\}$ such that $d'/c'\geq 3^3\lambda_2$.

First consider coverings $h$ whose ramification type does not appear in Table \ref{table:two-set-stabilizer}.  
 It suffices to prove the claim when $c'\ell-d'\geq 0$, and hence when $\ell\geq d'/c'\geq \max\{6,3^3\lambda_2\}.$
If $g_{Y_1}\leq 1$ then  $g_{Y_2}<3\ell$ by Proposition \ref{prop:Castel}. 
Since in addition $\ell\geq 3^3\lambda_2$ and the ramification type of $h$ is not in Table \ref{table:two-set-stabilizer}, Proposition \ref{prop:two-set} with $\alpha=3$ implies that
$$g_{X_2}\geq g_{X_2}-g_{X_1}>c_{2,3}\ell -d_{2,3}\geq c'\ell-d'.$$  
If $g_{Y_1}>1$, then 
$g_{X_2}\geq g_{X_2}-g_{X_1}>c_2\ell-2^8d_2\geq c'\ell-d'$ by Proposition \ref{lem:t=3,4}. 
Thus, in combination with the Riemann--Hurwitz formula for $\pi_2$, we have 
$$g_{Y_2}-1 \geq 2(g_{X_2}-1)> 2(c'\ell-d'-1),$$ proving the claim when the ramification of $h$ is not in Table \ref{table:two-set-stabilizer}.

If the ramification type of $h$ does appear in Table \ref{table:two-set-stabilizer} but is not $[\ell],[a,\ell-a],[2,1^{\ell-2}]$, then $R_{\pi_2}=\sum_{P\in Y(\K)}R_{\pi_2}(f_2^{-1}(P))$ is at least $(2\ell-5)/3$ by Remark \ref{rem:Y2-table}. Since in addition $c'\leq 1/3$ and $d'\geq 2$, the Riemann--Hurwitz formula for $\pi_2$ gives
\[
g_{Y_2}-1\geq 2(g_{X_2}-1) + R_{\pi_2} 
\geq -2 +  \frac{2\ell-5}{3}>2(c'\ell-d'-1), 
\]
completing the proof of the claim. 

Set $c_4:=\min\{(99/25)c',1\}$ and let $d_{4,\beta_1}\geq 1$ 
be sufficiently large so that 
\begin{equation}\label{equ:d4}
d_{4,\beta_1}t^{2k+2}\geq (99/25)(d'+1)+d_{5,\beta_1}t^{2k+2}
\end{equation}
 and $d_{4,\beta_1}t^{6}\geq (3E_0+1)t^4+8$ for all $t\geq 3$.  
Since the claim is trivial when $c_4\ell-d_{4,\beta_1}t^{2k+2}<0$, we may assume $\ell\geq d_{4,\beta_1}t^{2k+2}/c_4\geq t^{2k+2}$. 

Since $g_{Y_2}-1>2(c'\ell-d'-1)$ and $R_{\pi_t}\leq d_{5,\beta_1}t^{2k+2}(\ell-2)!/(\ell-t)!$ by Step~II, \eqref{equ:main-RH2}  gives
\begin{equation}\label{equ:final}
2t!(g_{X_t}-g_{X_{t-1}})  > \frac{(\ell-2)!}{(\ell-t)!}\Bigl( 4(1-\eps)(c'\ell - d'-1) -d_{5,\beta_1}t^{2k+2}\Bigr)
\end{equation}
As $t\geq 3$ and $k\geq 2$, we have $\ell\geq t^6$, so that $\eps<1/100$ and $4(1-\eps)c'>25/99 c'>c_4$. Thus \eqref{equ:final} gives 
$$ 2t!(g_{X_t}-g_{X_{t-1}}) > \frac{(\ell-2)!}{(\ell-t)!}\Bigl(c_4\ell - 4(1-\eps)(d'+1) 
-d_{5,\beta_1}t^{2k+2}\Bigr). $$
The right hand side is at least $(c_4\ell-d_{4,\beta_1}t^{2k+2})\frac{(\ell-2)!}{(\ell-t)!}$ by \eqref{equ:d4}.
This completes the proof when the ramification type of $h$ is not $[\ell], [a,\ell-a], [2, 1^{\ell-2}]$. 

{\bf Step III:} {\it The case where $h$ has ramification type $[\ell], [a,\ell-a], [2,1^{\ell-2}]$}. Let $P_1,P_2,P_3$ be the branch points with ramification $[\ell], [a,\ell-a], [2,1^{\ell-2}]$, respectively. 
In this case, applying the Riemann--Hurwitz formula to the natural projections $Y_2\ra Y_1$ and $Y_3\ra Y_2$ gives $g_{Y_2}=0$, and 
\begin{equation}\label{equ:gY3} 
2(g_{Y_3}-1) = (\ell-5)(\ell-2) + 2(a-1)(\ell-a-1)\geq (\ell-5)(\ell-2). 
\end{equation}
On the other hand, computing the bounds on $R_{\pi_t}$ via Proposition \ref{lem:pi-form} give
\begin{equation*}\label{equ:special}
\begin{split}
R_{\pi_t}(f_t^{-1}(P_1)) & < \biggl(\binom{t}{2}\delta_\ell+E_0 t^4\biggr)\frac{(\ell-2)!}{(\ell-t)!} \\
 R_{\pi_t}(f_t^{-1}(P_2)) & < \biggl(\binom{t}{2}(1-\delta_\ell)+E_0 t^4\biggr)\frac{(\ell-2)!}{(\ell-t)!} \\ 
R_{\pi_t}(f_t^{-1}(P_3)) & < \biggl(E_0t^4+\binom{t}{2}\cdot\frac{1}{\ell-2}\biggr)\frac{(\ell-2)!}{(\ell-t)!},
\end{split}
\end{equation*}
where $\delta_\ell=1$ if $\ell$ is even and $0$ otherwise. 
Note that as $\ell\geq t^6\geq 3^6$, in total we have $R_{\pi_t}<(3E_0+1)t^4(\ell-2)!/(\ell - t)!$. Thus for $t=3$, in combination with \eqref{equ:gY3}, the Riemann--Hurwitz formula for $\pi_3$ gives 
\begin{equation*}
    \begin{split}
    12(g_{X_3}-1) & =  12(g_{Y_3}-1) - \sum_{P\in Y(\K)}R_{\pi_3}(f_3^{-1}(P)) \\
    & > (\ell-2)(\ell-5) - (3E_0+1)t^4(\ell-2) > (c_4\ell-d_{4,\beta_1}t^{2k+2})(\ell-2). 
    \end{split}
\end{equation*}
As $g_{X_2}=0$, this gives $12(g_{X_3}-g_{X_2})>(c_4\ell-d_{4,\beta_1}t^{2k+2})(\ell-2)$ as desired. 
Henceforth assume $t\geq 4$. By 
\eqref{equ:M-RH-inequ} and Riemann--Hurwitz for the natural projection $Y_{t-1}\ra Y_3$, we have 
\begin{equation}\label{equ:done!!!}
2t!(g_{X_t}-g_{X_{t-1}})\geq 2\frac{\ell-5}{\ell-2}(g_{Y_{t-1}}-1)-R_{\pi_t} \geq 2\frac{\ell-5}{\ell-2}\cdot \frac{(\ell-3)!}{(\ell-t)!}(g_{Y_3}-1)  - R_{\pi_t}. 
\end{equation}
Since  $R_{\pi_t}<(3E_0+1)t^4(\ell-2)!/(\ell - t)!$ and $k\geq 2$, \eqref{equ:gY3} and \eqref{equ:done!!!} give
\begin{align*}\label{equ:main-RH4}
2t!(g_{X_t}-g_{X_{t-1}})&\geq  2\frac{\ell-5}{\ell-2}\cdot \frac{(\ell-3)!}{(\ell-t)!}(g_{Y_3}-1)  - R_{\pi_t}  \\
& >  \frac{(\ell-2)!}{(\ell-t)!}\biggl(\frac{(\ell-5)^2}{\ell-2}-(3E_0+1)t^4 \biggr) \\
& > (\ell-(3E_0+1)t^4-8)\frac{(\ell-2)!}{(\ell-t)!}> (c_4\ell-d_{4,\beta_1}t^{2k+2})\frac{(\ell-2)!}{(\ell-t)!}. \qedhere
\end{align*}
\end{proof}
\section{Nonoccurring ramification data} \label{sec:non-existence}
It remains to prove that the ramification data 
in Table \ref{table:non-existence} do not correspond to an indecomposable covering (Lemma \ref{lem:non-existing}). 
We shall use the following lemma of Guralnick--Shareshian \cite[Lemma 2.0.12]{GS}. We follow the notation of Setup \ref{sec:setup}. Permutation multiplication is left to right, that is, $(1,2)(1,3)=(1,2,3)$. 
\begin{lem}\label{lem:gXk}
Let $h:Y_1\ra Y$ be a degree $\ell$ covering with monodromy group $G\in \{A_\ell,S_\ell\}$, 
and $X_k$ the quotient by a $k$-set stabilizer. Then $g_{X_k}\leq g_{X_{k+1}}$ for all $1\leq k< \ell/2$. 
\end{lem}
\begin{rem}\label{rem:jordan} 
Let $G$ be a primitive subgroup of $S_\ell$ which does not contain $A_\ell$.
Classical results of Jordan \cite[Theorem 3.3.E and Example 3.3.1]{DM} imply that
$G$ does not contain a $p$-cycle for any prime $p<\ell-2$, and also that if $\ell\ge 9$ then
$G$ does not contain a product of two disjoint $2$-cycles.
\end{rem}
\begin{proof}[Proof of Lemma \ref{lem:non-existing}]
Assume on the contrary that there exists an indecomposable degree $\ell$ covering $h:Y_1\ra \mP^1$ whose ramification type is in Table~\ref{table:non-existence}, and let $G\leq S_\ell$ be its (primitive) monodromy group. 
By Riemann's existence theorem 
there exists a product $1$ tuple $a,b,c,d\in G$  that generates $G$, and whose cycle structures correspond to the given ramification data. We divide the argument into cases according to the types in Table~\ref{table:non-existence}.

{\bf\noindent Cases F1.N1-F1.N4}:
Since $G$ is primitive (as $h$ is indecomposable) and  
 contains a $2$-cycle in case F1.N1, 
 a $3$-cycle in case F1.N2,  
and an element of cycle structure $[2^2,1^{\ell-4}]$ in cases F1.N3, F1.N4,  $G$ contains $A_\ell$ if $\ell\geq 9$ by Remark \ref{rem:jordan}.
Riemann--Hurwitz for $h$ implies that $g_{X_1}=g_{Y_1}=1$ in all cases F1.N1-F1.N4. 
However, in each case we get $g_{X_2}=0$ by Remark \ref{lem:2set-genus-formula}, equality \eqref{equ:M-RH-t=2}. 
As $G\supseteq A_\ell$,  this contradicts $g_{X_2}\geq  g_{X_1}$ by Lemma \ref{lem:gXk}. 

{\bf \noindent Case F4.N1}: 
Let $P_1,P_2,P_3$ be the branch points of $h$ corresponding to 
$[2^{\ell/2}]$, $[2,3^{(\ell-2)/3}]$, $[1^2,6^{(\ell-2)/6}]$, respectively. 
Let $\alpha:\mP^1\ra \mP^1$ be a Galois degree $3$ covering totally ramified over $P_2$ and $P_3$, and unramified anywhere else. 
Consider a minimal covering $p:Z\ra \mP^1$ which factors through $h$ and $\alpha$, say as $p=h\circ  h'=\alpha\circ\alpha'$. 

We first claim that $\alpha'$ is a degree $\ell$ covering with monodromy group $S_\ell$ and ramification $[2^{\ell/2}]^3,  [1^2,2^{\ell/2-1}], [2,1^{\ell-2}]$. 
Let $p \circ \tilde p:\tilde Z\ra \mP^1$ be the Galois closure of $p$. 
Since $G$ is a primitive subgroup of $S_\ell$ containing the $2$-cycle $y^3$,
we have $G=S_\ell$ by Remark \ref{rem:jordan}. 
Since $G$ and $\Mon(\alpha)=C_3$ have no common nontrivial quotient, we have $\Mon(p\circ \tilde p)\cong S_\ell\times C_3$, $\Mon(\alpha'\circ \tilde p)=S_\ell$,
and $h$ and $\alpha$ have no common factorization $\alpha = w\circ \tilde\alpha$, $h=w\circ \tilde h$, with $\deg w>1$.
In particular, $\deg p = \deg \alpha\cdot \deg h$ by Lemma~\ref{lem:no-common}
and the ramification type of $\alpha'$ is $[2^{\ell/2}]^3,  [1^2,2^{\ell/2-1}], [2,1^{\ell-2}]$ by Lemma~\ref{lem:abh}.(2). 
Since $\Mon(\alpha'\circ\tilde p) = S_\ell$, and $\deg \alpha' =\ell$, 
$ \alpha'\circ\tilde p$ is the Galois closure of $\alpha'$, completing the proof of the claim. 

 
Letting $Z_2$ be the quotient  of $\tilde Z$ by a stabilizer of a $2$-set in $\Mon(\alpha')=S_\ell$, we get that $g_Z=1$ and $g_{Z_2}=0$ by Remark \ref{lem:2set-genus-formula}, equality \eqref{equ:M-RH-t=2}, contradicting  $g_{Z_2}\geq g_{Z}$ by Lemma~\ref{lem:gXk}. 

{\bf\noindent Case I2.N1} 
Without loss of generality assume $c$ has cycle structure $[2^{\ell/2}]$,  and $b$ has structure $[4,2^{\ell/2-2}]$,
so that $bc$ is an $\ell$-cycle.
View the letters $1,2,\ldots,\ell$ as vertices in a directed bicolored graph,
with  red edges representing $b$ and blue edges representing $c$. Since edges corresponding to a transposition go both way we denote them by an undirected edge. 

Each of the four vertices in the $4$-cycle $X$ of $b$ meets an edge of $c$,
and if the other vertex of that edge is not in $X$ then it meets an edge
of $b$ whose other vertex is not in $X$, and so on, where eventually we
have an edge of $c$ whose other vertex is in $X$.  
So from each vertex in $X$
we get a path of edges colored $cbcbcb\ldots bc$ in which the first and last
vertices are in $X$ (and are distinct) but no intermediate vertex is in $X$.

First suppose that one such path connects two adejacent vertices $u,v \in X$
where $u^b=v$; then $bc$ maps $u$ to [all vertices in the $uv$ path occurring after $c$'s] to $u$,
so $bc$ is not an $\ell$-cycle.
Thus the two paths connect the two pairs of non-adjacent vertices in $X$,
say $u,v$ and $w,x$, where the path connecting $u$ and $v$ contains $m$ edges from $b$
and $m+1$ edges from $c$ (for some $m\geq 0$), and likewise the path connecting $w$ and $x$
contains $n$ edges from $b$ and $n+1$ edges from $c$ (for some $n\geq 0$), and where $X$ is
$(u,w,v,x)$.  Here $\ell=2m+2n+4$.

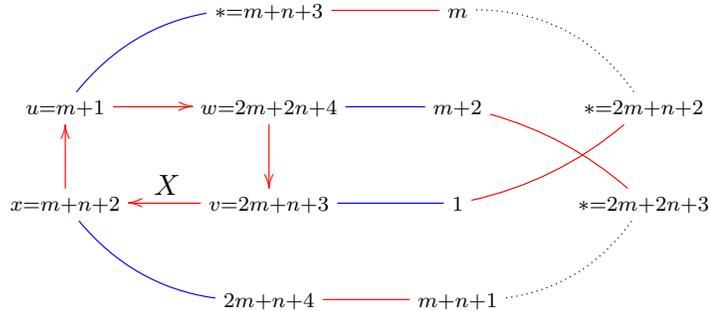
\[
\xymatrix{ 
 & \labelstyle \ast=m+n+3 \ar@/_1.0pc/@{-}@[blue][dl]\ar@{-}@[red][r]  & \labelstyle m \ar@/^1.0pc/@{..}[rd]   & \\
\labelstyle u=m+1 \ar@[red][r]  &  \labelstyle w=2m+2n+4 \ar@[red][d] \ar@{-}@[blue][r] & \labelstyle m+2 \ar@/^0.5pc/@{-}@[red][rd] & \labelstyle \ast = 2m+n+2\\
\labelstyle x= m+n+2 \ar@[red][u]   & \labelstyle v=2m+n+3 \ar@[red][l]_{\textstyle X} \ar@{-}@[blue][r] & \labelstyle 1 \ar@/_0.5pc/@{-}@[red][ur] & \labelstyle \ast=2m+2n+3 \\
& \labelstyle 2m+n+4 \ar@/^1.0pc/@{-}@[blue][ul] \ar@{-}@[red][r] & \labelstyle m+n+1 \ar@/_1.0pc/@{..}[ru]  & 
}
\]       

Now label the vertices so that $bc$ is $(1,2,\ldots ,\ell)$, starting by putting $w^{bc}=1$.
Then the path from $v$ to $u$ is
$v,1,\ast,2,\ast,3,\ast,\ldots ,m,\ast,u$  where the $\ast$ represent as-yet-unlabelled vertices.
Thus $u$ gets label $m+1$, and the path from $w$ to $x$ is
$$w,m+2,\ast,m+3,\ast,\ldots ,m+n+1,\ast,x$$
so that $x$ gets label $m+n+2$, and the path from $u$ to $v$ is
$$m+1,m+n+3,m,m+n+4,m-1,m+n+5,\ldots ,2,2m+n+2,1,v$$
so that $v$ gets label $2m+n+3$, and the path from $x$ to $w$ is
$$m+n+2,2m+n+4, m+n+1,2m+n+5,\ldots ,m+3,2m+2n+3,m+2,w$$
so that $w$ gets label $2m+2n+4$.

Thus $c$ is the product of transpositions $(i,3m+2n+6-i)$, for $m+2\leq i\leq m+n+2$,  and
 $(i,2m+n+4-i)$ for $1\leq i \leq m+1$,
and $b$ is the product of the transpositions $(i,3m+2n+5-i)$ for $m+2\leq i\leq m+n+1$,  and $(i,2m+n+3-i)$ for $1\leq i\leq m$,  and the $4$-cycle  $(m+1,2m+2n+4,2m+n+3,m+n+2)$.

The $(\ell/2)$-th power of the $\ell$-cycle $bc$ is the product of $(i,\ell/2+i)$, for $1\leq i \leq \ell/2$.
For $1\leq i \leq m+1$, we have $2m+n+4-1\geq m+n+2+i\geq 2m+n+4-(m+1)$
so that $c$ maps $\{i,\ell/2+i\}$ to $\{2m+n+4-i,2m+n+4-\ell/2-i\}$.
For $m+2\leq i \leq m+n+2$, we have 
$$3m+2n+6-(m+2)\geq m+n+2+i\geq 3m+2n+6-(m+n+2)$$ so that
$c$ maps $\{i,\ell/2+i\}$ to $\{3m+2n+6-i,3m+2n+6-\ell/2-i\}$.
Therefore both $bc$ and $c$ preserve the partition $\{i,\ell/2+i\}$, $1\leq i\leq \ell/2$,
so also $\langle b,c\rangle$ preserves this partition, whence $\langle b,c\rangle$ is not primitive.

{\bf\noindent Case I2.N2} 
Without loss of generality assume $b$ and $c$ have cycle structure $[2^{\ell/2}]$, and $d$ is a $2$-cycle $(u,v)$
such that $dcb$ is an $\ell$-cycle.  Then each vertex of $d$ is an endpoint of
a unique $b$-edge and a unique $c$-edge, and each $b$-or-$c$-edge emanating from
a vertex of $d$ gives rise to a path of alternating $b$-and-$c$-edges which
ends in one of the vertices of $d$.  If each such path ends at the same
vertex of $d$ at which it starts, then $dcb$ maps $u\mapsto v^{cb} \mapsto v^{cbcb} \mapsto \ldots
 \mapsto v \mapsto u^{cb} \mapsto u^{cbcb} \mapsto \ldots  \mapsto u$  so that $dcb$ is not an $\ell$-cycle, as it does not cover all vertices in the alternating $b$-and-$c$  paths from $u$ to $v$ and  from $v$ to $u$.
Hence there are two paths from $u$ to $v$ consisting of $b$-and-$c$-edges, one
starting with $b$ and one starting with $c$.
If the path starting with $c$ ends in $b$ then $dcb$ maps $v\ra u^{cb}\ra u^{cbcb}\ra \ldots\ra v$,
so $dcb$ is not an $\ell$-cycle.
Thus the path starting with $c$ ends in $c$; say it consists of $m$ $b$-edges and
$m+1$ $c$-edges.  Then the path starting with $b$ ends in $b$; say it consists of $n+1$ $b$-edges and $n$ $c$-edges.  Here $\ell=2m+2n+2$.

Now label the vertices so that $dcb$ is $(1,2,\ldots,\ell)$, starting by putting $v^{cb}=1$.
Then the path from $v$ to $u$ which starts with $c$ is
$v,\ast,1,\ast,2,\ast,3,\ast,\ldots ,\ast,m,u$  where the $\ast$ represent as-yet-unlabelled vertices.
Thus the path from $u$ to $v$ which starts with $b$ is
$u,m+1,\ast,m+2,\ast,\ldots ,m+n,\ast,v$  where $v$ gets label $m+n+1$.
Then the path from $u$ to $v$ which starts with $c$ is
$$u,m,m+n+2,m-1,m+n+3,\ldots ,1,2m+n+1, v=m+n+1$$
so the path from $v$ to $u$ which starts with $b$ is
$$m+n+1,2m+n+2, m+n,2m+n+3,\ldots ,m+1,u=2m+2n+2=\ell.$$
Thus $c$ is the product of $(m+n+1,2m+n+1)$, $(m,2m+2n+2)$, $(i,2m+n+1-i)$ for $1\leq i\leq m-1$, and $(i,3m+2n+2-i)$ for $m+1\leq i \leq m+n$; 
$b$ is the product of $(i,2m+n+2-i)$ for $1\leq i\leq m$, and $(i,3m+2n+3-i)$ for $m+1\leq i\leq m+n+1$;
 and $d$ is $(\ell/2,\ell)$.

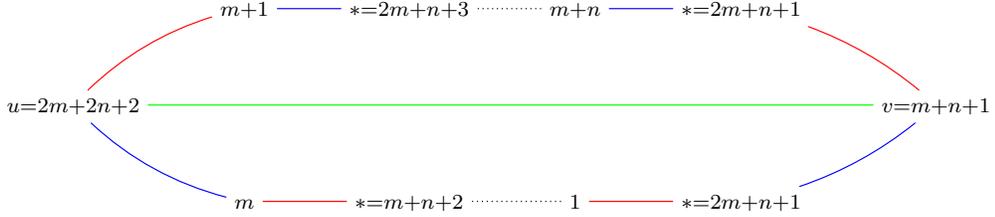
\[
\xymatrix{ 
& \labelstyle m+1 \ar@{-}@[blue][r] & \labelstyle \ast = 2m+n+3 \ar@{..}[r] & \labelstyle m+n \ar@{-}@[blue][r] &\labelstyle  \ast = 2m + n + 1 \ar@/^0.5pc/@{-}@[red][dr] & \\
\labelstyle u=2m+2n+2 \ar@{-}@[green][rrrrr] \ar@/^0.5pc/@{-}@[red][ur] \ar@/_0.5pc/@{-}@[blue][dr] &   &  & & & \labelstyle v = m+n + 1 \\
& \labelstyle m \ar@{-}@[red][r] & \labelstyle \ast = m+n+2  \ar@{..}[r] & \labelstyle 1 \ar@{-}@[red][r] &\labelstyle  \ast = 2m + n + 1 \ar@/_0.5pc/@{-}@[blue][ur] & \\
}
\]

The $\ell/2$-th power of the $\ell$-cycle $dcb$ is the product of $(i,\ell/2+i)$ over all $1\leq i \leq \ell/2$.
Plainly $d$ preserves the partition $\{i,\ell/2+i\}$, $1\leq i\leq \ell/2$.
For every $1\leq i\leq m$ we have $2m+n+2-1\geq \ell/2+i\geq 2m+n+2-m$, so that $b$ maps
$\{i,n/2+i\}$ to $\{2m+n+2-i,2m+n+2-\ell/2-i\}$.
For $m+1\leq i\leq m+n+1$ we have $3m+2n+3-(m+1)\geq \ell/2+i\geq 3m+2n+3-(m+n+1)$, so that $b$ maps
$\{i,\ell/2+i\}$ to $\{3m+2n+3-i,3m+2n+3-\ell/2-i\}$.
Therefore the partition $\{i,\ell/2+i\}$, $1\leq i\leq \ell/2$ is preserved by $b,d$ and $dcb$,
so also by $\langle b,c,d\rangle$, whence $\langle b,c,d\rangle$ is not primitive.
\end{proof}

\appendix 

\section{Existence of coverings with ramification as in Table \ref{table:two-set-stabilizer}}\label{sec:app}
We use Riemann's existence theorem to show: 
\begin{prop}\label{prop:An-Sn-ram-types}
Each ramification data in Table \ref{table:two-set-stabilizer} is the ramification type of a degree $\ell$ covering $h:\mP^1\ra\mP^1$ with $\Mon(h)\cong A_\ell$ or $S_\ell$.
\end{prop}
\begin{proof}
The branch cycles in types I1, I2, and types F1.1-F1.8 were proven to occur in Propositions 3.0.24--3.0.28 of \cite{GS}. It remains to treat the cases F1.9 and the cases in F3 and F4. 
For each of these ramification data, we write a product $1$ tuple $x_1,\ldots,x_r\in S_\ell$ of elements whose cycle structures correspond to the ramification type, and show that that the group $G:=\langle x_1,\ldots,x_r\rangle$ contains $A_\ell$. 
This prove the proposition by Riemann's existence theorem.  
Each element in a product $1$ relation is written within parenthesis. 

{\bf \noindent Case F1.9}: Write the product $1$ relation as $x_1x_2 = x_4x_3$ where $x_3^2=x_4^2=1$: 
\begin{align*}
& \biggl((1,4)(3,5)\prod_{i=3}^{\ell/2-1}(2i,2i+1)\biggr) \cdot \biggl((1,2,3,4)\prod_{i=2}^{\ell/2-1}(2i+1,2i+2)\biggr) \\
& = (4,2,3, 6,8,10,\ldots,\ell-2,\ell, \ell-1,\ell-3,\ell-5,\ldots ,7,5) \\
& =  \biggl(\prod_{i=2}^{\ell/2-1}(2i+1,2i+2)\biggr) \cdot \biggl((2,3)(4,6)\prod_{i=1}^{\ell/2-2}(2i+1,2i+4)\biggr).
 \end{align*}
The subgroup $G=\langle x_1,x_2,x_3,x_4\rangle\leq S_\ell$ is transitive and contains an $(\ell-1)$-cycle, hence doubly transitive and in particular primitive.  But  $x_2^2\in G$  has cycle structure $[2^2,1^{\ell-4}]$, so $G$
 contains $A_\ell$ if $\ell>8$ by Remark \ref{rem:jordan}. 

{\bf \noindent Case F3.1}: 
For $\ell=4k$, $k\geq 2$, write the relation as $x_1x_2=x_3$ where $x_3^2=1$:
\begin{align*}
& \biggl((2,1,4,5)(4k-2,4k-5,4k,4k-1) \prod_{i=1}^{k-2} (4i+2,4i-1,4i+4,4i+5)\biggr) \\
&\cdot \biggl((1)(2,3,4) \prod_{i=1}^{k-1} (4i+1,4i+2,4i+3,4i+4)\biggr)
 \\
& =\biggl((2,1)(4k-1)(4k) \prod_{i=1}^{k-1} (4i,4i+2)(4i+1,4i-1)\biggr).
\end{align*}
Consider a block which contains $2$ and has size at least $2$.
Since $x_2^4=(2,3,4)$, if this block contains a point outside $\{2,3,4\}$ then that point (and hence the whole block) is
fixed by $(2,3,4)$; and if the block contains either $3$ or $4$ then the block must be fixed by $(2,3,4)$.  In any case,
the block is fixed by $(2,3,4)$, so it contains $3$ and $4$.  Thus the block is fixed by $x_2$ (and not just by
$x_1^4$), and by $x_1^2$. Since $\langle x_2,x_1^2\rangle\leq S_\ell$ is transitive the block is $\{1,2,\ldots ,\ell\}$. 
So $G$ is primitive and contains a $3$-cycle, so it contains $A_\ell$ by Remark \ref{rem:jordan}. Note that $G=S_\ell$, as $x_3$ has cycle structure $[1^2,2^{2k-1}]$.
A similar argument applies in {\bf cases F3.2 and F3.3}. In case F3.2 with $\ell=4k+5$, $k\geq 1$, we write the relation $x_1x_2=x_3$ as
\begin{align*}
&\biggl((1)(4,5,4k+5,4k+4)(4k-2,4k-1,4k+2,4k+3)   \prod_{i=1}^{k-1} (4i-2,4i-1,4i+4,4i+5)\biggr) \\
&\cdot  \biggl( (1,2,3)(4k+4,4k+5) \prod_{i=1}^k (4i,4i+1,4i+2,4i+3)\biggr)  
 \\
& = \biggl( (1,2)(5,4k+4)(4k-1,4k+3)(4k,4k+2)(4k+5) \prod_{i=1}^{k-1} (4i-1,4i+5)(4i,4i+2)\biggr).
\end{align*}
 and in case F3.3 with $\ell=4k+3$, $k\geq 1$, we write it as\footnote{Note that when applying the argument from case F3.1 to case F3.3, one uses the $3$-cycle $x_1^4$ to replace the role of the $3$-cycle $x_2^4$ in case F3.1. 
}
\begin{align*}
&\biggl( (1,4,5)(4k-2,4k-1,4k+2,4k+3) \prod_{i=1}^{k-1} (4i-2,4i-1,4i+4,4i+5)\biggr)  \\
&\cdot \biggl( (1)(2,3) \prod_{i=1}^k (4i,4i+1,4i+2,4i+3)\biggr) 
 \\
&=\biggl((1,5)(2)(4k-1,4k+3)(4k,4k+2) \prod_{i=1}^{k-1} (4i-1,4i+5)(4i,4i+2)\biggr).
\end{align*}

\noindent{\bf Case F4.3}:  For $\ell=6k+7$, $k\geq1$, write $x_1x_2=x_3$, where $x_3^2=1$: 
%
%
\begin{align*}
& \biggl( (1,6k+7,6k+5)(2,6k+4,6)(6k-3,6k-1,6k+1)(6k-2,6k+3,6k+2)\\
 &\qquad (6k+6)\prod_{i=1}^{k-1} (6i-3,6i-1,6i+1)(6i+2,6i-2,6i+6)\biggr)
 \\
 &\cdot \biggl( (6k+1,6k+2,6k+3)(6k+4,6k+5,6k+6,6k+7) \\
& \qquad\prod_{i=0}^{k-1} (6i+1,6i+2,6i+3,6i+4,6i+5,6i+6) \biggr) \\
&=\biggl((1,6k+4)(2,6k+5)(6k+3)(6k+6,6k+7) \\
&\qquad \prod_{i=1}^k (6i-3,6i)(6i-2,6i+1)(6i-1,6i+2)\biggr).
\end{align*}
Consider a block which contains $6k+4$ and has size at least $2$.  Since $x_2^6$ equals  $(6k+4,6k+6)(6k+5,6k+7)$,  if the block does not contain $6k+6$ then the block is not fixed by $x_2^6$ and hence
does not contain any elements fixed by $x_2^6$, so the block is contained in $\{6k+4,6k+5,6k+7\}$ (and must have
size $2$ since it is not fixed by $x_2^6$).  
But then the block is fixed by $x_2$, so it contains
$6k+6$, contradiction.  Thus the block contains $6k+6$, so it is fixed by $x_1$ and $x_2^2$. Since $\langle x_2^2,x_1\rangle\leq S_\ell$ is transitive the block  is $\{1,2,\ldots ,\ell\}$.
So $G$ is primitive and contains an element of cycle structure $[2^2,1^{\ell-4}]$, hence $G$ contains $A_\ell$ for $\ell\geq 9$ by Remark \ref{rem:jordan}. 
A similar argument applies in
 {\bf case F4.5}\footnote{When applying the argument from case F4.3 to case F4.5, one uses  $x_2^2$ to replace the role of the element $x_2^6$ in case F4.3}. Here $\ell=6k+4$, $k\geq 1$, and  the relation  $x_1x_2=x_3$ is written as
\begin{align*}
& \biggl( (1,6,2)(3,5,7)(6k-2,6k+4,6k+2)(6k+3)  \\
& \qquad \prod_{i=2}^k (6i-4,6i-8,6i)(6i-3,6i-1,6i+1)\biggr)
 \\
 &\cdot \biggl( (6k+1,6k+2,6k+3,6k+4) \prod_{i=0}^{k-1} (6i+1,6i+2,6i+3,6i+4,6i+5,6i+6)  \biggr)    \\
& = \biggl((1)(2)(6k+3,6k+4) \prod_{i=1}^k (6i-3,6i)(6i-2,6i+1)(6i-1,6i+2)\biggr).
\end{align*}

\noindent {\bf Case F4.6}: 
%
%
 For $\ell=6k+5$, $k\geq 1$, write the relation as $x_1x_2=x_3$, where $x_3^2=1$: 
\begin{align*}
& \biggl( (1,6k+5)(2,6k+4,6)(6k-3,6k-1,6k+1)(6k+2,6k-2,6k+3) \\
&\qquad \prod_{i=1}^{k-1} (6i-3,6i-1,6i+1)(6i+2,6i-2,6i+6)\biggr)
 \\
&\cdot \biggl( (6k+1,6k+2,6k+3)(6k+4,6k+5) \prod_{i=0}^{k-1} (6i+1,6i+2,6i+3,6i+4,6i+5,6i+6)\biggr) \\
&=\biggl((1,6k+4)(2,6k+5)(6k+3) \prod_{i=1}^k (6i-3,6i)(6i-2,6i+1)(6i-1,6i+2)\biggr).
\end{align*}
Consider a block which contains $1$ and has size at least $2$.  Since $x_1^3=(1,6k+5)$, it
must fix the block so the block contains $6k+5$ and hence is fixed by $x_2$.  Also $x_2^2$ fixes $6k+5$ and hence fixes the block. Since $\langle x_2^2,x_1\rangle\leq S_\ell$ is transitive, the block is $\{1,\ldots,\ell\}$. 
So $G$ is primitive and contains a $2$-cycle,
hence  $G=S_\ell$ by Remark \ref{rem:jordan}. 
A similar argument applies in {\bf cases F4.1, F4.2, and F4.4}.
In case F4.1 with $\ell=6k$, $k\geq 1$, write the relation  $x_1x_2=x_3$ as
\begin{align*}
& \biggl( (1)(2,6)(6k-1,6k-2,6k-3) \prod_{i=1}^{k-1} (6i-3,6i-1,6i+1)(6i+2,6i-2,6i+6)\biggr)
 \\
 &\cdot \biggl( \prod_{i=0}^{k-1} (6i+1,6i+2,6i+3,6i+4,6i+5,6i+6)  \biggr)  \\
& = \biggl( (6k-2)(6k-1)(1,2)(6k-3,6k) \prod_{i=1}^{k-1} (6i-3,6i)(6i-2,6i+1)(6i-1,6i+2)\biggr).
\end{align*}
In case F4.2 with
 $\ell=6k+2$, $k\geq 1$, we write this relation  as
\begin{align*}
&  \biggl( (1,6k+1)(2,6k+2,6)(6k-1,6k-2,6k-3)\\
&\qquad \cdot \prod_{i=1}^{k-1} (6i-3,6i-1,6i+1)(6i+2,6i-2,6i+6) \biggr)
 \\
 &\biggl( (6k+1,6k+2) \prod_{i=0}^{k-1} (6i+1,6i+2,6i+3,6i+4,6i+5,6i+6) \biggr) \\
& = \biggl((1,6k+2)(2,6k+1)(6k,6k-3)(6k-1)(6k-2) \\
&\qquad \prod_{i=1}^{k-1} (6i-3,6i)(6i-2,6i+1)(6i-1,6i+2)\biggr).
\end{align*}
In case F4.4 with $\ell=6k+3$, $k\geq 1$, we write this relation as 
\begin{align*}
& \biggl( (1)(2,6)(6k-2,6k+3,6k+2)(6k-3,6k-1,6k+1) \\
&\qquad  \prod_{i=1}^{k-1} (6i-3,6i-1,6i+1)(6i+2,6i-2,6i+6) \biggr)
 \\
 & \cdot \biggl( (6k+1,6k+2,6k+3) \prod_{i=0}^{k-1} (6i+1,6i+2,6i+3,6i+4,6i+5,6i+6)  \biggr)  \\
& =  \biggl((1,2)(6k+3) \prod_{i=1}^k (6i-3,6i)(6i-2,6i+1)(6i-1,6i+2)\biggr).
\end{align*}
\end{proof}

\bibliographystyle{plain}

\end{document}